\documentclass[12pt]{article}
\usepackage{array}
\usepackage{amscd}
\usepackage{xspace}
\usepackage{amsmath}
\usepackage{amsfonts}
\usepackage{amssymb}
\usepackage{amsthm}
\usepackage{graphicx}
\usepackage{makeidx}
\usepackage{color}
\usepackage{url,hyperref}
\usepackage{fancyvrb}
\usepackage{tikz}

\definecolor{dbluecolor}{rgb}{0,0,0.6}
\definecolor{dredcolor}{rgb}{.5,0.0,0.0}
\definecolor{dgreencolor}{rgb}{0,0.4,0}
\definecolor{blue}{rgb}{.02, .02, .908}

\newcommand{\dgreen}{\color{dgreencolor}\bf}
\def\red#1{\textcolor{red}{#1}}
\def\blue#1{\textcolor{blue}{#1}}
\def\magenta#1{\textcolor{magenta}{#1}}

\newif\ifcomments
\commentstrue
\ifcomments
\newcommand{\djp}[1]{\magenta{DP: #1}}
\newcommand{\wdj}[1]{\red{DJ: #1}}
\newcommand{\cgm}[1]{\blue{CGM: #1}}
\else
\newcommand{\djp}[1]{}
\newcommand{\wdj}[1]{}
\newcommand{\cgm}[1]{}
\fi

\newif\ifpngs
\pngsfalse

\ifpngs
\newcommand{\drawpng}[2]{\includegraphics[#1]{#2}}
\else
\newcommand{\drawpng}[2]{}
\fi

\def\qqq{\mathbb{Q}}
\def\ccc{\mathbb{C}}
\def\zzz{\mathbb{Z}}

\def\trace{{\rm trace}\, }

\newcommand{\SAGE}{{\dgreen {\sf Sage}}\xspace}
\newcommand{\sage}{\SAGE}

\newtheorem{theorem}{Theorem}
\newtheorem{corollary}[theorem]{Corollary}
\newtheorem{analog}[theorem]{Analog}
\newtheorem{problem}[theorem]{Problem}
\newtheorem{conjecture}[theorem]{Conjecture}
\newtheorem{lemma}[theorem]{Lemma}
\newtheorem{proposition}[theorem]{Proposition}
\newtheorem{definition}[theorem]{Definition}
\newtheorem{example}[theorem]{Example}
\newtheorem{question}{Question}
\newtheorem{remark}[theorem]{Remark}

\newcommand{\supp}{{\rm supp}}

\makeindex

\begin{document}

\author{Charles Celerier, David Joyner\thanks{USNA, Mathematics Department;
email: wdj@usna.edu }, Caroline Melles, 
\\
David Phillips, Steven Walsh}

\title{Explorations of edge-weighted Cayley graphs and $p$-ary bent functions}

\maketitle

\begin{abstract}
  Let $f:GF(p)^n \rightarrow GF(p)$. 
When $p=2$, Bernasconi et al
have shown that there is a correspondence between certain 
properties of $f$ (eg, if it is bent) 
and properties of its associated Cayley graph.
Analogously, but much earlier, Dillon showed that $f$ is bent
if and only if the ``level curves'' of $f$ had certain
combinatorial properties (again, only when $p=2$).
The attempt is to investigate an analogous theory when $p>2$
using the (apparently new) combinatorial concept of a 
weighted partial difference set. More precisely,
we try to investigate which graph-theoretical properties of 
$\Gamma_f$ can be characterized in terms of function-theoretic 
properties of $f$, 
and which function-theoretic properties of $f$ correspond
to combinatorial properties of the set of ``level curves'' $f^{-1}(a)$
($a\in GF(p)$). While the natural generalizations of the
Bernasconi correspondence and Dillon correspondence are not
true in general, using extensive computations, we are able to
determine a classification in small cases:
$(p,n)\in \{(3,2),(3,3),(5,2)\}$. Our main conjecture is
Conjecture \ref{conjecture:main}.
\end{abstract}

\tableofcontents

\section{Introduction}

Fix $n\geq 1$ and let $V=GF(p)^n$, where $p$ is a prime.

\begin{definition} ({\bf Walsh(-Hadamard) transform})
{\rm
The {\it Walsh(-Hadamard) transform} of a function 
$f:GF(p)^n \rightarrow GF(p)$ is 
a complex-valued function on $V$ that can be defined as

\begin{equation}
\label{eqn:WT}
W_f(u) =
\sum_{x \in V}
\zeta^{f(x)- \langle u,x\rangle},
\end{equation}
where $\zeta = e^{2\pi i/p}$. 
}
\end{definition}
\index{Walsh(-Hadamard) transform}
\index{Walsh transform}

We call $f$ {\it bent} if

\[
|W_f(u)|=p^{n/2},
\]
for all $u\in V$. 
The class of $p$-ary bent functions are ``maximally non-linear'' in 
some sense, and can be used to generate pseudo-random sequences rather
easily.
\index{bent}

Some properties of the Walsh transform:

\begin{enumerate}
\item
The Walsh coefficients satisfy {\it Parseval's equation}

\[
\sum_{u \in V}
|W_f(u)|^2 = p^{2n}.
\]
\index{Parseval's equation}

\item
If $\sigma =\sigma_k:\qqq(\zeta)\to \qqq(\zeta)$ 
is defined by sending $\zeta \mapsto \zeta^k$ then
$W_f(u)^\sigma = W_{kf}(ku)$.

\end{enumerate}

If $f:V\to GF(p)$ then we let $f_\ccc:V\to \ccc$ be the function whose
values are those of $f$ but regarded as integers
(i.e., we select the congruence class residue
representative in the interval $\{0,\dots,p-1\}$).

\begin{definition} ({\bf Fourier transform})
{\rm
When $f$ is complex-valued, we define the
analogous {\it Fourier transform} of the function $f$
as

\begin{equation}
\label{eqn:FT}
\hat{f}(y) = f^\wedge(y) = \sum_{x\in V} f_\ccc(x)\zeta^{-\langle x,y\rangle}.
\end{equation}
}
\end{definition}
\index{Fourier transform}

Note

\[
\hat{f}(0)=\sum_{x\in V} f_\ccc(x),
\]
and note

\[
W_f(y)=(\zeta^f)^\wedge (y).
\]
We say $f$ is {\it even}, if $f(-x)=f(x)$ for all $x\in GF(p)^n$.
It is not hard to see that if $f$ is even then the Fourier transform of
$f$ is real-valued. (However, this is not necessarily true of the
Walsh transform.)

\begin{example}
{\rm 
It turns out that there are a total of $3^4=81$ even functions 
$f:GF(3)^2\to GF(3)$ with $f(0)=0$, of which exactly $18$ are bent.
Section \ref{sec:bent-gf3**2} discusses this in more detail.
}
\end{example}

\begin{example}
{\rm 
It turns out that there are a total of $3^{13}=1594323$ even functions 
$f:GF(3)^3\to GF(3)$ with $f(0)=0$, of which exactly $2340$ are bent.
Section \ref{sec:bent-gf3**3} discusses this in more detail.
}
\end{example}

\begin{example}
{\rm 
It turns out that there are a total of $5^{12}=244140625$ even functions 
$f:GF(5)^2\to GF(5)$ with $f(0)=0$, of which exactly $1420$ are bent.
Section \ref{sec:bent-gf5**2} discusses this in more detail.
}
\end{example}

\begin{definition} ({\bf Hadamard matrix})
{\rm
We call an $N\times N$ $\{0,1\}$-matrix $M$ a {\it Hadamard matrix} if

\[
M\cdot {M}^t=NI_N, 
\]
where $I_N$ is the $N\times N$ identity matrix. 
}
\end{definition}
\index{Hadamard matrix}

\begin{remark}
There is a concept similar to the notion of bent, called ``CAZAC.''
We shall clarify their connection in this remark.

Constant amplitude zero autocorrelation (CAZAC) functions have
been studied intensively since the 1990's \cite{BD}. 
We quote the following characterization, which is due to J. Benedetto
and S. Datta \cite{BD}: 

\vskip .1in
\noindent
Theorem: Given a sequence $x : \zzz/N\zzz \to \ccc$, 
and let $C_x$ be a circulant matrix with first row 
$x = (x[0],x[1],\dots ,x[N-1])$. Then $x$ is a CAZAC
sequence if and only if $C_x$ is a Hadamard matrix.

\vskip .1in
This is due to the fact that the definition of CAZAC functions
uses the Fourier transform on $\zzz/N\zzz$. The corresponding definition of
bent functions $f$ uses the Fourier transform 
(of $\zeta^f$)  on $GF(p)^n$, $N=p^n$. The analogous Hadamard 
matrix for a bent function $f$ is not circulant but ``block
circulant.''

\end{remark}
\index{CAZAC functions}

In the Boolean case, there is a nice simple relationship between the
Fourier transform and the Walsh-Hadamard transform. 
In equation (\ref{eqn:WTvFT}) below, we shall try to connect these two
transforms, (\ref{eqn:WT}) and (\ref{eqn:FT}),  in the $GF(p)$ case as well.
In this context, it is worth noting that
it is possible (see Proposition \ref{prop:bent}) to characterize 
a bent function in terms of the Fourier transform of
its derivative.

Suppose $f:GF(p)^n\to GF(p)$ is bent. 

\begin{definition} ({\bf regular})
{\rm
Suppose $f$ is bent. We say $f$ is
{\it regular} if and only if $W_f(u)/p^{n/2}$ is a $p$th root 
of unity for all $u\in V$.
}
\end{definition}
\index{bent!regular}
\index{bent!weakly regular}
\index{bent!weakly regular dual}

If $f$ is regular then there is a function $f^*:GF(p)^n\to GF(p)$, called the
{\it dual} (or {\it regular dual}) of $f$, such that
$W_f(u)=\zeta^{f^*(u)}p^{n/2}$,
for all $u\in V$.
We call $f$ {\it weakly regular}\footnote{If $\mu$ is fixed
and we want to be more precise, we call this {\it $\mu$-regular}.},
if there is a function $f^*:GF(p)^n\to GF(p)$, called the
{\it dual} (or {\it $\mu$-regular dual}) of $f$, such that
$W_f(u)=\mu\zeta^{f^*(u)}p^{n/2}$,
for some constant $\mu\in \ccc$ with absolute value $1$.

\begin{proposition}
\label{remark:KSW}
\label{prop:KSW}
(Kumar, Scholtz, Welch)
 If $f$ is bent then
there are functions $f_*:GF(p)^n\to \zzz$ and $f^*:GF(p)^n\to GF(p)$ such that

\[
W_f(u)p^{-n/2} = 
\left\{
\begin{array}{rl}
(-1)^{f_*(u)} \zeta^{f^*(u)},& {\rm if}\ n\ {\rm is\ even,\ or\ }n\ {\rm is \ odd\ and}\ p\equiv 1\pmod 4,\\
i^{f_*(u)}\zeta^{f^*(u)},& {\rm if}\ n\ {\rm is \ odd\ and}\ p\equiv 3\pmod 4.\\
\end{array}
\right.
\]
\end{proposition}

The above  result is known (thanks to Kumar, Scholtz, Welch \cite{KSW})
but the form above is due to Helleseth and Kholosha \cite{HK3}
(although we made a minor correction to their statement).
Also, note \cite{KSW} Property 8 
established a more general fact than the statement above.

\begin{corollary} If $f$ is bent and $W_f(0)$ is rational
(i.e., belongs to $\qqq$) then $n$ must be even.
\end{corollary}

The condition $W_f(0)\in \qqq$ arises in Lemma \ref{lemma:sig} below,
so this corollary shall be useful later. 

Suppose $f:V=GF(p)^n\to GF(p)$ is bent.
In this case, for each $u\in V$, the quotient $W_f(u)/p^{n/2}$ is
an element of the cyclotomic field $\qqq (\zeta)$ having absolute value $1$.

Below we give a simple necessary and sufficient conditions to determine if
$f$ is regular. The next three lemmas are well-known but included 
for the reader's convenience.

\begin{lemma}
\label{lemma:weaklyregular}
Suppose $f:V\to GF(p)$ is bent.
The following are equivalent.

\begin{itemize}
\item
$f$ is weakly regular.

\item
$W_f(u)/W_f(0)$ is a $p$-th root of unity for all $u\in V$.

\end{itemize}

\end{lemma}

\begin{proof}
If $f$ is weakly regular with $\mu$-regular dual $f^*$, 
then $W_f(u)/W_f(0)=\zeta^{f^*(u) - f^*(0)}$, for each 
$u \in V$.
 
Conversely, if  $W_f(u)/W_f(0)$ is of the form $\zeta^{i_u}$, 
for some integer $i_u$ (for $u\in V$), then let 
$f^*(u)$ be $i_u \pmod p$ and let
$\mu=W_f(0)/(p^{n/2})$.  Then $f^*(u)$ is a $\mu$-regular dual of $f$. 
\end{proof}

\begin{lemma}
Suppose $f:V\to GF(p)$ is bent and weakly regular.
The following are equivalent.

\begin{itemize}
\item
$f$ is regular.

\item
$W_f(0)/p^{n/2}$ is a $p$-th root of unity.

\end{itemize}

\end{lemma}

\begin{proof}
One direction is clear.  Suppose that $f$ is a weakly regular bent 
function with $\mu$-regular dual $f^*$ and suppose that
$W_f(0)/(p^{n/2}) =\zeta^i$.  Note that
$W_f(0)=\mu \zeta^{f*(0)} p^{n/2} = \zeta^ip^{n/2}$ so that 
$\mu =\zeta^{i - f*(0)}$.
Let $g(u)=f^*(u) - f^*(0) + i$ (where we are treating 
$i$ as an element of $GF(p)$). 
Then 

\[W_f(u)=\mu \zeta^{f*(u)} p^{n/2} 
=\zeta^{i - f^*(0)}  \zeta^{g(u)+f^*(0)-i} p^n/2 = \zeta^{g(u)} p^{n/2},
\]
so $f$ is regular.
\end{proof}

\begin{lemma} Suppose that $f$ is bent and weakly regular, 
with $\mu$-regular dual $f^*$. Then $f^*$ is bent and weakly 
regular, with $\mu^{-1}$-regular dual $f^{**}$ given 
by $f^{**}(x)=f(-x)$.  If $f$ is also even, then $f^*$ is even and $f^{**}=f$.

\end{lemma}

\begin{proof}  Suppose that $f$ is bent and weakly regular with 
$\mu$-regular dual $f^*$.  Then 
\[ W_f(u)=\mu \zeta^{f^*(u)} p^{\frac n 2}
\]
for all $u$ in $V$.  The Walsh transform of $f^*$ is given by 
\begin{equation}
\begin{aligned}
W_{f^*}(u) 
&= \sum_{y \in V} \zeta^{f^*(y)}\zeta^{-\langle u,y \rangle} \\
&= \sum_{y \in V} \mu^{-1}p^{-\frac n 2} W_f(y) \zeta^{-\langle u,y \rangle}\\
&= \mu^{-1}p^{-\frac n 2}\sum_{y \in V}\sum_{x \in V} \zeta^{f(x)}
     \zeta^{-\langle y,x \rangle}\zeta^{-\langle u,y \rangle}\\
&= \mu^{-1}p^{-\frac n 2}\sum_{y \in V}\sum_{x \in V} \zeta^{f(x )}
     \zeta^{-\langle y,x+u \rangle}\\
&= \mu^{-1}p^{-\frac n 2}\sum_{w \in V} \zeta^{f(w-u)} \sum_{y \in V} 
     \zeta^{-\langle y,w \rangle}.
\end{aligned}
\label{eqn:Walshdual}
\end{equation}
Next we note that
\[
\sum_{y \in V}\zeta^{-\langle y,w\rangle} = 
\begin{cases} p^n & \text{if }w=0\\
 0 & \text{if }w \neq 0
\end{cases}
\]
since, if $y =(y_1, y_2, \ldots ,y_n)$ and $w=(w_1,w_2,\ldots , w_n)$, 
we have 
\begin{equation*}
\begin{aligned}
\sum_{y \in V}\zeta^{-\langle y,w\rangle} 
&= \sum_{y_1 \in GF(p)}\sum_{y_2 \in GF(p)}\ldots \sum_{y_n \in GF(p)}
     \zeta^{- y_1 w_1}\zeta^{- y_2 w_2}\ldots \zeta^{- y_n w_n}\\
&= \prod_{i=1}^n \left( \sum_{y \in GF(p)} \zeta^{-y w_i} \right)
\end{aligned}
\end{equation*}
and, if $w_i \neq 0$, 
\[ 
\sum_{y \in GF(p)} \zeta^{-y w_i} 
= \zeta^0 + \zeta^1 + \ldots + \zeta^{p-1} = 0.
\]
Therefore equation \ref{eqn:Walshdual} reduces to
\begin{equation*}
\begin{aligned}
W_{f^*}(u) 
&= \mu^{-1} p^{-\frac n 2}  \zeta^{f(-u)} p^n\\
&=\mu^{-1} \zeta^{f(-u)}p^{\frac n 2}.
\end{aligned}
\end{equation*}
It follows that $f^*$ is bent with $\mu^{-1}$-regular 
dual $f^{**}$ given by $f^{**}(x) = f(-x)$ and that if $f$ is even, 
$f^{**} = f$.

Furthermore, if $f$ is even, 
\begin{equation*}
\begin{aligned}
\zeta^{f^*(-u)}
&= \mu^{-1} p^{-\frac n 2} W_f(-u)&&\\
&=\mu^{-1} p^{-\frac n 2} \sum_{x \in V} \zeta^{f(x)} \zeta^{-\langle -u,x \rangle}&&\\
&=\mu^{-1} p^{-\frac n 2} \sum_{w \in V} \zeta^{f(-w)} \zeta^{-\langle u,w \rangle}&&\\
&=\mu^{-1} p^{-\frac n 2} \sum_{w \in V} \zeta^{f(w)} \zeta^{-\langle u,w \rangle}&&\text{since $f$ is even}\\
&=\mu^{-1} p^{-\frac n 2}W_f(u)&&\\
&=\zeta^{f^*(u)}.&&
\end{aligned}
\end{equation*}
Since $f^*$ takes values in $GF(p)$, it follows 
that $f^*(-u)=f^*(u)$ for all $u$ in $V$, so $f^*$ is even.
\end{proof}

\section{Partial difference sets}

Dillon's thesis \cite{D} was one of the first publications
to discuss the relationship between bent functions and 
combinatorial structures, such as difference sets. His work
concentrated on the Boolean case. Consider functions

\[
f:GF(p)^n\to GF(p),
\]
where $p$ is a prime, $n>1$ is an integer. In Dillon's work, it was proven that the
``level curve'' $f^{-1}(1)$ gives rise to a difference set
in $GF(2)^n$. 

\begin{definition} ({\bf difference set})
\label{defn:DS}
{\rm
Let $G$ be a finite abelian multiplicative group of order $v$, and let 
$D$ be a subset of $G$ with order $k$. 
$D$ is a {\it $(v,k,\lambda)$-difference set} (DS) if the list 
of differences $d_1 d_2^{-1}, d_1, d_2 \in D$, represents every non-identity 
element in $G$ exactly $\lambda$ times.
}
\end{definition}

A {\it Hadamard difference set} is one whose parameters are of the form
$(4m^2,2m^2-m,m^2-m)$, for some $m>1$. It is, in addition,
{\it elementary} if $G$ is an elementary abelian $2$-group
(i.e., isomorphic to $(\zzz/2\zzz)^n$).
\index{difference set}
\index{difference set!Hadamard}

Let $D^{-1}= \{d^{-1}\ |\ d \in D \}$.

\begin{lemma}
Let $G$ be a finite abelian multiplicative group of order $v$ and let 
$D$ be a subset of $G$ with order $k$, such that $D=D^{-1}$. 
If $(G,D)$ is an $(v,k,\lambda ,\lambda)$-partial difference set
then it is also a $(v,k,\lambda)$-difference set.
\end{lemma}

\begin{proof}
This follows from character theory: combine Theorem 1 and Theorem 2,
and (the proof of) Proposition 1 in Polhill \cite{Po}.
\end{proof}

\begin{theorem}
(Dillon Correspondence, \cite{D}, Theorem 6.2.10, page 78)
\label{thrm:DC}
The function $f:GF(2)^n\to GF(2)$ is bent if and only if 
$f^{-1}(1)$ is an elementary Hadamard difference set of
$GF(2)^n$.
\end{theorem}
\index{Dillon Correspondence}

Two (naive) analogs of this are formalized below in 
Analog \ref{analog:DC-PDS}
and
Analog \ref{analog:DC-AS}.

\subsection{Weighted partial difference sets}

In this paper, we consider the ``level curves''
$f^{-1}(a) \subset GF(p)^n$ ($a\in GF(p)$, $a\not=0$)
and investigate the combinatorial structure of these sets,
especially when $f$ is bent.

\begin{definition} ({\bf PDS})
\label{defn:PDS}
{\rm
Let $G$ be a finite abelian multiplicative group of order $v$, and let 
$D$ be a subset of $G$ with order $k$. 
$D$ is a {\it $(v,k,\lambda ,\mu)$-partial difference set} (PDS) if the list 
of differences $d_1 d_2^{-1}, d_1, d_2 \in D$, represents every non-identity 
element in $D$ exactly $\lambda$ times and every non-identity element in 
$G \setminus D$ exactly $\mu$ times.
}
\end{definition}
\index{PDS}
\index{partial difference set (PDS)}

This notion can be characterized algebraically in terms of the group ring
$\ccc [G]$.

\begin{lemma}
With the notation as in the defintion above,
$(G,D)$ forms a $(v,k,\lambda ,\mu)$-PDS if and only if 
(\ref{eqn:schur-ring2}) holds.

\end{lemma}

The well-known proof is omitted.

\begin{example}
{\rm
Consider the finite field

\[
GF(9) = GF(3)[x]/(x^2+1) = \{0,1,2,x,x+1,x+2,2x,2x+1,2x+2\},
\]
written additively.
The set of non-zero quadratic residues is given by

\[
D = \{1,2,x,2x\}.
\]
One can show that $D$ is a PDS with parameters

\[
v = 9,\ \ 
k=4,\ \
\lambda=1,
\ \ \mu=2. 
\]
}
\end{example}

We shall return to this example (with more details) below,
in Example \ref{example:17}.

\begin{definition} ({\bf Latin square type PDS})
{\rm
Let $(G,D)$ be a PDS. We say it is of {\it Latin square type} (resp.,
{\it negative Latin square type}) if there
exist $N>0$ and $R>0$ (resp., $N<0$ and $R<0$) such that 

\[
(v,k,\lambda,\mu)=(N^2,R(N-1), N+R^2-3R, R^2-R).
\]
}
\end{definition}
\index{PDS!Latin square type}

The example above is of Latin square type ($N=3$ and $R=2$)
and of negative Latin square type ($N=-3$ and $R=-1$).

Let $G$ be a finite abelian multiplicative group and let $D$ 
be a subset of $G$. Decompose $D$ into a union of 
disjoint subsets

\begin{equation}
\label{eqn:disjoint-union}
D = D_1\cup \dots \cup D_r,
\end{equation}
and assume $1\notin D$. Let $k_{i}=|D_i|$.

\begin{definition} ({\bf weighted PDS})
\label{defn:wtPDS}
{\rm
Let $W$ be a weight set of size $r$, and
$v\in \zzz$, $k\in \zzz^{|W|}$,
$\lambda\in \zzz^{|W^3|}$, and $\mu\in \zzz^{|W^2|}$.
We say $D$ is a {\it weighted $(v,k,\lambda,\mu)$-PDS},
if the following properties hold

\begin{itemize}
\item
The list of ``differences'' 

\[
D_iD_j^{-1} = \{d_1d_2^{-1}\ |\ d_1\in D_i, d_2\in D_j\},
\]
represents every non-identity element of $D_\ell$ 
exactly $\lambda_{i,j,\ell}$ times and every non-identity element
of $G\setminus D$ exactly $\mu_{i,j}$ times ($1\leq i,j,\ell\leq r$).
\item
For each $i$ there is a $j$ such that $D_i^{-1}=D_j$
(and if $D_i^{-1}=D_i$ for all $i$ then we say the
weighted PDS is {\it symmetric}).
\end{itemize}
}
\end{definition}
\index{PDS!weighted}

This notion can be characterized algebraically in terms of the group ring
$\ccc [G]$.

\begin{lemma}
With the notation as in the definition above,
$(G,D_1,\dots,D_s)$ forms a symmetric weighted 
$(v,k,\lambda ,\mu)$-PDS if and only if 
$D=D^{-1}$ and (\ref{eqn:schur-ring4}) holds.

\end{lemma}

The straightforward proof is omitted.

\begin{remark}
If $D = D_1\cup \dots \cup D_r$ 
is a symmetric weighted PDS then
$\mu_{i,j}=\mu_{j,i}$ and $\lambda_{i,j,\ell}=\lambda_{j,i,\ell}$.
\end{remark}

How does the above notion of a weighted PDS relate to the 
usual notion of a PDS?

\begin{lemma}
\label{lemma:unweightedPDS}
Let $(G,D)$, where $D = D_1\cup \dots \cup D_r$ 
(disjoint union) is as in (\ref{eqn:disjoint-union}),
be a symmetric weighted PDS, with
parameters $(v,(k_i),(\lambda_{i,j,\ell}), (\mu_{i,j}))$.
If
\[
\sum_{i,j} \lambda_{i,j,\ell}
\]
does not depend on $\ell$, $1\leq \ell\leq r$,
then $D$ is also an unweighted PDS with
parameters $(v,k,\lambda,\mu)$ where 

\[
k = \sum_{i} k_{i},\ \ 
\lambda = \sum_{i,j} \lambda_{i,j,\ell},\ \ \ \
\mu = \sum_{i,j} \mu_{i,j}.
\]
\end{lemma}

\begin{remark}
Case 1 in Proposition \ref{prop:gf33bent} does not
satisfy this hypothesis.
\end{remark}

\begin{proof}
The claim is that $(G,D)$ is a PDS with parameters
$(v,k,\lambda,\mu)$.
Since $v=|G|$ and 

\[
k = |D|=|D_1|+\dots +|D_r|=k_1+\dots+k_r,
\]
we need only verify the claim regarding 
$\lambda$ and $\mu$.

Does each element $d$ of $D$ occur the same number of times 
in the list $DD^{-1}$? Suppose $d\in D_\ell$, where $1\leq \ell\leq r$. 
By hypothesis, $d$ occurs in $D_i-D_j$ exactly $\lambda_{i,j,\ell}$ 
times. Since $DD^{-1}$ is the concatenation of the $D_iD_j^{-1}$,
for $1\leq i,j\leq r$, $d$ occurs in $D\cdot D^{-1}$ exactly

\[
\sum_{i,j} \lambda_{i,j,\ell}
\]
times. By hypothesis, this does not depend on $\ell$, so
the claim regarding $\lambda$ has been verified.

Does each non-zero element $d$ of $G\setminus D$ occur the same number of times 
in the list $D\cdot D^{-1}$? By hypothesis, $d$ occurs in $D_iD_j^{-1}$ 
exactly $\mu_{i,j}$ times. Since $D\cdot D^{-1}$ is the concatenation 
of the $D_iD_j^{-1}$, for $1\leq i,j\leq r$, $d$ occurs in $D\cdot D^{-1}$ exactly

\[
\sum_{i,j} \mu_{i,j}
\]
times. This verifies the claim regarding $\mu$ and completes
the proof of the lemma.

\end{proof}

\begin{example}
\label{example:17}
{\rm
Consider the finite field

\[
GF(9) = GF(3)[x]/(x^2+1) = \{0,1,2,x,x+1,x+2,2x,2x+1,2x+2\},
\]
written multiplicatively
The set of non-zero quadratic residues is given by

\[
D = \{1,2,x,2x\}.
\]
Let $D_1=\{1,2\}$ and $D_2=\{x,2x\}$.

Translating the multiplicative notation to the additive notation, we find

\[
D_1D_1^{-1} = [d_1-d_2\ |\ d_1\in D_1, d_2\in D_1]
= [0,0,1,2],
\]
\[
D_1D_2^{-1} = [d_1-d_2\ |\ d_1\in D_1, d_2\in D_2]
= [x+1,x+2,2x+1,2x+2],
\]
\[
D_2D_1^{-1} = [d_1-d_2\ |\ d_1\in D_2, d_2\in D_1]
= [x+1,x+2,2x+1,2x+2],
\]
\[
D_2D_2^{-1} = [d_1-d_2\ |\ d_1\in D_2, d_2\in D_2]
= [0,0,x,2x].
\]
Therefore, this describes a weighted PDS with parameters

\[
k_{1,1}=2,\ k_{2,2}=2,\ k_{1,2}=k_{2,1}=0,
\]
\[
\lambda_{1,1,1}=1,\ \lambda_{1,1,2}=0,\ 
\lambda_{1,2,1}=0,\ \lambda_{1,2,2}=0,\ 
\]
\[
\lambda_{2,1,1}=0,\ \lambda_{2,1,2}=0,\ 
\lambda_{2,2,1}=0,\ \lambda_{2,2,2}=1,
\]
and
\[
\mu_{1,1}=0,\ \mu_{1,2}=1,\ 
\mu_{2,1}=1,\ \mu_{2,2}=0. 
\]
}
\end{example}

As we will see, there a weighted analog of the correspondence between
PDSs and SRGs.

We have the following generalization of Theorem \ref{thrm:5}.

\begin{theorem}
Let $G$ be an abelian multiplicative group and let
$D\subset G$ be a subset such that $1\notin D$ and
with disjoint decomposition $D = D_1\cup D_2\cup \dots \cup D_r$.
The following are equivalent:

\begin{itemize}
\item[(a)]
$(G,D)$ is a symmetric weighted partial difference set having parameters
$(v,k,\lambda,\mu)$, where $v=|G|$, $k=\{k_i\}$ with
$k_i = |D_i|$, $\lambda = \{\lambda_{i,j,\ell}\}$, and
$\mu = \{\mu_{i,j}\}$.

\item[(b)]
$\Gamma(G,D)$ is a strongly regular edge-weighted (undirected) graph with parameters
$(v,k,\lambda,\mu)$ as in (a).
\end{itemize}

\end{theorem}

\begin{proof}
Let $D' = G\setminus D-\{1\}$.

((a) $\implies$ (b))
Suppose $(G,D)$ is a weighted partial difference set satisfying
$D_i=D_i^{-1}$, for all $i$. The graph $\Gamma=\Gamma(G,D)$ 
has $v=|G|$ vertices, by definition.
Each vertex $g$ of $\Gamma$ has $k_i$ neighbors of weight $i$, 
namely, $dg$ where $d\in D_i$.
(We say two vertices are ``neighbors having edge-weight $0$'' 
if they are not connected by an edge in the unweighted graph.) 
Let $g_1,g_2$ be distinct vertices in $\Gamma$.
Let $x$ be a vertex which is a neighbor  of each:
$x\in N(g_1,i)\cap N(g_2,j)$. By definition, $x=d_1g_1=d_2g_2$, for
some $d_1\in D_i$, $d_2\in D_j$. Therefore, $d_1^{-1}d_2 = g_1g_2^{-1}$.
If $g_1g_2^{-1}\in D_\ell$, for some $\ell\not=0$, then there are 
$\lambda_{i,j,\ell}$ solutions, by definition of a weighted PDS.
If $g_1g_2^{-1}\in D'$ then there are 
$\mu_{i,j}$ solutions, by definition of a weighted PDS.

((b) $\implies$ (a))
Note $f$ even implies the symmetric condition of a weighted PDS.
For the remainder of the proof, note the reasoning above is 
reversible. Details are left to the reader.
\end{proof}

Let $f$ be a $GF(p)$-valued function on $V$. The {\it Cayley graph of $f$} is 
defined to be the edge-weighted digraph

\begin{equation}
\label{eqn:wtdCayleyGraph}
  \Gamma_f = (GF(p)^n, E_f ), 
\end{equation}
whose vertex set is $V=V(\Gamma_f)=GF(p)^n$ and the set of edges is defined by 

\[
            E_f =\{(u,v) \in GF(p)^n\ |\ f(u-v)\not= 0\},
\]
where the edge $(u,v)\in E_f$ has weight $f(u-v)$.
However, if $f$ is even then we can (and do) regard $\Gamma_f$ as a 
weighted (undirected) graph. 
\index{Cayley graph!$p$-ary function}

\begin{theorem}
Let $f:GF(p)^{n} \rightarrow GF(p)$ be an even function such that $f(0) = 0.$ 
Let $D_i = f^{-1}(i)$, for $i=1,2,\dots,p-1$, and $D_0=\{0\}$. 
Let $D_p = GF(p)^n\setminus D_0 \cup D_1 \cup \cdots \cup D_{p-1}.$ 
If $(GF(p)^n,D_0,\dots,D_p)$ is a weighted partial 
difference set, where $G = GF(p)^{n}$, then the associated 
(strongly regular) graph is the (edge-weighted) Cayley graph of $f$.
\end{theorem}

\begin{remark}
\label{remark:level-curves}
Roughly speaking, this theorem says that ``if the level curves of $f$
form a weighted PDS then the (edge-weighted) Cayley graph 
corresponding to $f$ agrees with the (edge-weighted) strongly 
regular graph associated to the weighted PDS.''
\end{remark}

\begin{proof}
The adjacency matrix $A$ for the Cayley graph of $f$ is defined by 
$A_{ij} = f(j-i)$ for $i,j \in GF(p)^{n}$. So the top row of A is defined by $A_{0j} = f(j)$. 
The adjacency matrix for any Cayley graph can be determined from its top row, 
since it is a circulant matrix. Therefore, it is enough to show that that the 
adjacency matrix of the weighted partial difference set has the same top row 
as $A$. Let $B$ be the adjacency matrix of the weighted partial difference set. 
Then $B$ is defined by $B_{ij} = k$ if $j-i \in D_{k}$ 
(for all $i,j \in GF(p)^{n}$ and $k \in GF(p)$). 
The top row of $B$ is defined by 
$B_{0j} = k$ if $j \in D_{k}$. But if $j \in D_{k}$, then $f(j) = k$, 
so $B_{0j} = f(j)$. The top rows of $A$ and $B$ are equivalent, so 
$A = B$. Therefore, the strongly regular graph associated with the 
weighted partial difference set $(G,D)$ is the Cayley graph of $f$. 
\end{proof}
\index{adjacency matrix}

\begin{conjecture} (Walsh)
If $f:GF(p)^n\to GF(p)$, $p>2$ is weakly regular and bent and 
corresponds to a weighted SRG (via Analog \ref{analog:BC-SRG})
then $\mu_{ii}=0$, $1\leq i\leq s$, for the associated weighted PDS.
\end{conjecture}

\begin{remark}
If you drop the hypothesis that $f$ be weakly regular
then the conjecture is false.
\end{remark}

\subsection{Association schemes}

The following definition is standard, but we give
\cite{PTFL} as a reference.

\begin{definition} ({\bf association scheme})
{\rm
Let $S$ be a finite set and let $R_0, R_1, \dots, R_s$ denote 
binary relations on $S$ (subsets of $S\times S$).
The {\it dual} of a relation $R$ is the set

\[
R^* = \{(x,y)\in S\times S\ |\ (y,x)\in R\}.
\]
Assume $R_0=\Delta_S= \{ (x,x)\in S\times S\ |\ x\in S\}$.
We say $(S,R_0,R_1,\dots,R_s)$ is a 
{\it $s$-class association scheme on $S$} if
the following properties hold.

\begin{itemize}
\item
We have a disjoint union

\[
S\times S = R_0\cup R_1\cup \dots \cup R_s,
\]
with $R_i\cap R_j=\emptyset$ for all $i\not= j$.

\item
For each $i$ there is a $j$ such that $R_i^*=R_j$
(and if $R_i^*=R_i$ for all $i$ then we say the
association scheme is {\it symmetric}).

\item
For all $i,j$ and all $(x,y)\in S\times S$, define

\[
p_{ij}(x,y) = |\{z\in S\ |\ (x,z)\in R_i, (z,y)\in R_j\}|.
\]
For eack $k$, and for all $x,y\in R_k$, the integer $p_{ij}(x,y)$ is
a constant, denoted $p_{ij}^k$.
\end{itemize}
These constants $p_{ij}^k$ are called the {\it intersection numbers}
or {\it parameters} or {\it structure constants} of
the association scheme.
}
\end{definition}
\index{association scheme}
\index{association scheme!parameters}
\index{association scheme!symmetric}
\index{association scheme!structure constants}
\index{association scheme!intersection numbers}
\index{relation!dual}
\index{relation!symmetric}
\index{intersection numbers}

Next, we recall (see Herman \cite{He}) the matrix-theoretic 
version of this definition.

\begin{definition} ({\bf adjacency ring})
{\rm
Let $S$ be a finite abelian multiplicative group of order $m$. 
Let $(S,R_0,\dots, R_s)$ denote a tuple consisting of $S$ with relations $R_i$ for which we have a disjoint union

\[
S\times S = R_0\cup R_1\cup \dots \cup R_s,
\]
with $R_i\cap R_j=\emptyset$ for all $i\not= j$.
Let $A_i\in Mat_{m\times m}(\zzz)$ denote the adjacency matrix of $R_i$, $i=0,1,\dots ,s$.

We say that the subring of $\zzz [Mat_{m\times m}(\zzz)]$ is an {\it adjacency ring}
(also called the Bose-Mesner algebra) provided
the set of adjacency matrices satisfying the following five properties:

\begin{itemize}
\item
for each integer $i\in [0,\dots ,s]$, $A_i$ is a $(0, 1)$-matrix,
\item
$\sum_{i=0}^s A_i  = J$ (the all $1$'s matrix),
\item
for each integer $i\in [0\dots ,s]$, $^tA_i=A_j$, for some  integer $j\in [0,s]$,
\item
there is a subset $J \subset G$ such that $\sum_{j\in J} A_j = I$, and
\item
there is a set of non-negative integers 
$\{p_{ij}^k \ |\ i,j,k \in [0,\dots, s]\}$ such that 
equation (\ref{eqn:AiAj}) holds
%\[
%A_iA_j = \sum_{k=0}^s p_{ij}^k A_k ,
%\]
for all such $i,j$.
\end{itemize}
}
\end{definition}
\index{adjacency ring}
\index{Bose-Mesner algebra}

Regarding the Dillon correspondence, we have the following 
combinatorial analogs (which may or may not be true in general).
\index{Dillon correspondence}

\begin{analog}
\label{analog:DC-PDS}
If $f$ is an even bent function then the 
tuple 

\[
(GF(p)^n,D_0,D_{1}, D_{2}, \cdots ,D_{p-1},D_p)
\]
defines a weighted partial difference set.
\end{analog}

We reformulate this in an essentially equivalent way using the
language of association schemes.

\begin{analog}
\label{analog:DC-AS}
Let $f$ be as above and let $R_0, R_1, \dots, R_p$ denote 
binary relations on $GF(p)^n$ given by

\[
R_i = \{(x,y)\in GF(p)^n\times GF(p)^n\ |\ f(x,y) = i \},\ \ \ 
0\leq i \leq p.
\]
If $f$ is an even bent function then $(GF(p)^n,R_0,R_1,\dots,R_p)$ 
is a $p$-class association scheme.
\end{analog}
\index{Dillon correspondence!analogs}

It is well-known that a PDS $(G,D)$ is naturally associated to 
a $2$-class association scheme, namely $(G,R_0,R_1,R_2)$ where

\[
R_0 = \Delta_G,
\]
\[
R_1 = \{(g,h)\ |\ gh^{-1}\in D\},
\]
\[
R_2 = \{(g,h)\ |\ gh^{-1}\notin D,\ g\not= h\}.
\]
To verify this, let consider the ``Schur ring.''

For the following definition, we identify
any subset $S$ of $G$ with the formal sum of 
its elements in $\ccc[G]$.

\begin{definition} ({\bf Schur ring})
{\rm
Let $G$ be a finite abelian group and let 
$C_0, C_1,\dots, C_s$ denote finite subsets with the following properties.

\begin{itemize}
\item
$C_0=\{1\}$ is the singleton containing the identity.

\item
We have a disjoint union

\[
G = C_0\cup C_1\cup \dots \cup C_s,
\]
with $C_i\cap C_j=\emptyset$ for all $i\not= j$.

\item
for each $i$ there is a $j$ such that
$C_i^{-1}=C_j$ (and if $C_i^{-1}=C_i$ for all $i$ then we say the
Schur ring is {\it symmetric}).

\item
for all $i,j$, we have

\[
C_i\cdot C_j = \sum_{k=0}^s \rho_{ij}^k C_k,
\]
for some integers $\rho_{ij}^k$.
\end{itemize}
The subalgebra of $\ccc[G]$ generated by $C_0, C_1,\dots, C_s$
is called a {\it Schur ring} over $G$.
}
\end{definition}
\index{Schur ring}

In the cases we are dealing with, the Schur ring is commutative, so 
$\rho_{ij}^k=\rho_{ji}^k$, for all $i,j,k$.

If $(G,C_0,\dots, C_s)$ is a Schur ring then

\[
R_i = \{(g,h)\in G\times G\ |\ gh^{-1}\in C_i\},
\]
for $0\leq i,j\leq s$, gives rise to its corresponding association scheme.

\begin{remark}
Weighted PDSs in the notation of Definition \ref{defn:wtPDS}
naturally correspond to association schemes of class $r+1$.
For a precise version of this, see Proposition 
\ref{prop:wtPDS2AS} below.

\end{remark}

\begin{example}
{\rm
For an example of a Schur ring, we return to the PDS, $(G,D)$. Let

\[
D' = G\setminus (D\cup \{1\}).
\]
We have the well-known intersection

\begin{equation}
\label{eqn:schur-ring2}
\begin{array}{rl}
D\cdot D &= (k-\mu)\cdot I+(\lambda -\mu)\cdot D  + \mu \cdot G\\
 & = k\cdot I +\lambda \cdot D + \mu \cdot D'\ ,\\
\end{array}
\end{equation}

and
\begin{equation}
\label{eqn:schur-ring3}
\begin{array}{rl}
D\cdot D' &= (-k+\mu)\cdot 1 + (-1-\lambda +\mu)\cdot D  + (k-\mu )\cdot G\\
 & =  = 0\cdot I +(k-1-\lambda) \cdot D + (k-\mu) \cdot D'\ .\\
\end{array}
\end{equation}
Provided $k\geq {\rm max}(\mu,\lambda+1)$, $|G|\geq {\rm max}(k+1,2k-\mu+2)$, 
with these, one can verify that a PDS naturally 
yields an associated Schur ring, generated by $D$, 
$D'$, and $D_0=\{1\}$ in $\ccc[G]$, and a $2$-class association scheme.

Using (\ref{eqn:schur-ring3}), one can verify 
that $D'$ is $(v,k^\prime,\lambda^\prime, \mu^\prime )$-PDS with 
$(D')^{-1}=D'$ and $1 \notin D'$, where 
\begin{equation}
\label{eqn:primepar}
\begin{array}{rl}
k^\prime &= v - k - 1\\
\lambda^\prime&=v - 2k - 2 + \mu, \rm {and}\\
\mu^\prime&=v-2k+\lambda.
\end{array}
\end{equation}.  

We include a proof here for convenience.  

\begin{proof} 
We will show that $D^\prime$ is a 
$(v,k^\prime, \lambda^\prime, \mu^\prime)$-partial 
difference set.
The first of these three equations is immediate, from the 
definition of $D^\prime$. The fact that $D^\prime=(D^\prime)^{-1}$ 
also follows immediately the hypotheses.

By the definition of $D$, and because $D^{-1}=D$, we have 
\begin{equation}
D \cdot D = k{1}+\lambda D + \mu D^\prime.
\end{equation}
To find $D \cdot D^\prime$, we note that 
\[
\begin{array}{rl}
kG &=D \cdot G \\
&=D \cdot (\{1\} + D + D^\prime)\\
&=D + D \cdot D + D \cdot D^\prime
\end{array}
\]
so that 
\begin{equation}
D \cdot D^\prime = 
(k - \lambda - 1)D+(k - \mu) D^\prime.
\end{equation}

Similarly, we note that 
\[
\begin{array}{rl}
 k^\prime G &=G \cdot D^\prime  \\
&= ((\{1\} + D + D^\prime)\cdot D^\prime  \\
&=D^\prime + D \cdot D^\prime + D^\prime \cdot D^\prime
\end{array}
\]
so that 
\begin{equation}
\label{eqn:Dprime}
\begin{array}{rl}
D^\prime \cdot D^\prime &= 
k^\prime \{1\} +(k^\prime -k+\lambda+1)D+
(k^\prime -k-1-+ \mu) D^\prime\\
&=k^\prime \{1\} +(v -2k+\lambda)D+
(v -2k- 2+ \mu) D^\prime.
\end{array}
\end{equation}

Equation \ref{eqn:Dprime} shows that $D^\prime$ is a 
$(v,k^\prime, \lambda^\prime, \mu^\prime)$-partial 
difference set, with $\lambda^\prime$ and $\mu^\prime$ 
as in Equation \ref{eqn:primepar}.

\end{proof}

It can be shown that $\mu^\prime=k^\prime\left(1 - \frac \mu k\right).$

}
\end{example}

With the identities in the above example, one can verify 
that a PDS naturally yields an associated Schur ring
and a $2$-class association scheme.

We will now state a more general proposition concerning weighted partial difference sets.

\begin{proposition}
{\rm
Let $G$ be a finite abelian group. Let $D_0, \cdots, D_r \subseteq G$ such that $D_i \cap D_j = \emptyset$ if $i \neq j$, and 
\begin{itemize}
\item
$G$ is the disjoint union $D_0 \cup \cdots \cup D_r$
\item
for each $i$ there is a $j$ such that $D_i^{-1} = D_j$, and
\item
$D_i \cdot D_j = \sum\limits_{k=0}^r p_{ij}^k D_k$ for some positive integer $p_{ij}^k$.
\end{itemize}
Then the matrices $P_k = (p_{ij}^k)_{0 \leq i,j \leq l}$ satisfy the following properties:
\begin{itemize}
\item
$P_0$ is a diagonal matrix with entries $|D_0|, \cdots, |D_r|$
\item
For each $k$, the $j$th column of $P_k$ has sum $|D_j|$ ($j = 0, \cdots, l$). Likewise, the $i$th row of $P_k$ has sum $|D_i|$ ($i = 0, \cdots, l$).
\end{itemize}
}
\end{proposition}

\begin{proof}
{\rm
We begin by taking the sum
\begin{center}
$D_i \cdot D_j = \sum\limits_{k=0}^r p_{ij}^k D_k$
\end{center} 
over all $i, 0 \leq i \leq l$.
\begin{center}
$G \cdot D_j = \sum\limits_{k=0}^r (\sum\limits_{i=0}^r p_{ij}^k) D_k$
\end{center}
We know that $G \cdot D_j = |D_j| \cdot G$, and all the $D_k$ are disjoint. As an identity in the Schur ring, each element of G must occur $|D_j|$ times on each side of this equation. Therefore, 
\begin{center}
$|D_j| = \sum\limits_{i=0}^r p_{ij}^k$.
\end{center}
So the sum of the elements in the $j$th row of $P_k$ is $|D_j|$ for each $j$ and $k$. The analogous claim for the row sums is proven similarly.
}
\end{proof}

\section{Cayley graphs}

Let $(G,D)$ be a PDS.

\begin{definition} ({\bf Cayley graph})
\label{defn:CG}
{\rm
The {\it Cayley graph $\Gamma=\Gamma(G,D)$ associated to the PDS} $(G,D)$ 
is a graph constructed as follows: from a subset $D$ of $G$, let the 
vertices of the graph be the elements of the group $G$. 
Two vertices $g_1$ and $g_2$ are 
connected by a directed edge if $g_2 = d g_1$ for some $d \in D$. 
}
\end{definition}
\index{Cayley graph!PDS}

If $D$ is a partial difference set such that $\lambda \not= \mu$, 
then $D = D^{-1}$ (Proposition 1 in \cite{Po}). 
Thus, if $g_2 = d g_1$, then $g_1 = d^{-1} g_2$, so the 
Cayley graph $\Gamma(G,D)$ is an undirected graph.

\begin{definition} ({\bf SRG})
\label{defn:SRG}
{\rm
A connected graph $\Gamma=(V,E)$ is a 
{\it $(v,k,\lambda ,\mu)$-strongly regular graph} if:

\begin{itemize}
\item
$\Gamma$ has $v$ vertices such that each vertex is connected to $k$ other vertices
\item
Distinct vertices $g_1$ and $g_2$ share edges with either $\lambda$ or 
$\mu$ common vertices ,
depending on whether they are neighbors or not.
\end{itemize}
}
\end{definition}
\index{strongly regular graph}
\index{SRG}

The {\it neighborhood} of a vertex $g$ in a graph $\Gamma$ is the set 

\[
N(g) = \{ g' \in V\ | \ (g,g')\ {\rm is\ an\ edge\ in}\ \Gamma \}.
\]

The following result is well-known, but the proof is included for convenience.

\begin{theorem}
\label{thrm:5}
\label{thrm:PDSiffSRG}
Let $G$ be an abelian multiplicative group and let $D \subseteq G$ be a subset such that 
$1 \not\in D$. $D$ is a $(v,k,\lambda ,\mu)$-PDS such that $D = D^{-1}$ if and only if the 
associated (undirected) Cayley graph $\Gamma(G,D)$ is a $(v,k,\lambda ,\mu)$-strongly 
regular graph.
\end{theorem}

\begin{proof} 
Suppose $D$ is a $(v,k,\lambda ,\mu)$-PDS such that $D = D^{-1}$. Then $\Gamma(G,D)$ has 
$v$ vertices. $D$ has $k$ elements, and each vertex $g$ of $\Gamma(G,D)$ has neighbors $d g$, 
$d \in D$. Therefore, $\Gamma(G,D)$ is regular, degree $k$. Let $g_1$ and $g_2$ be distinct
 vertices in $\Gamma(G,D)$. Let $x$ be a vertex that is a common neighbor of $g_1$ and $g_2$, i.e. 
$x \in N(g_1) \cap N(g_2)$. Then $x = d_1 g_1 = d_2 g_2$ for some $d_1, d_2 \in D$, 
which implies that $d_1 d_2^{-1} = g_1^{-1} g_2$. If $g_1^{-1} g_2 \in D$, then there are exactly 
$\lambda$ ordered pairs $(d_1,d_2)$ that satisfy the previous equation 
(by Definition \ref{defn:PDS}). 
If $g_1^{-1} g_2 \notin D$, then $g_1^{-1} g_2 \in G \setminus D$, so there are exactly $\mu$ 
ordered pairs $(d_1,d_2)$ that satisfy the equation. If $g_1^{-1} g_2 \in D$, then $g_2 = d g_1$ 
for some $d \in D$, so $g_1$ and $g_2$ are adjacent. By a similar argument, if 
$g_1^{-1} g_2 \in G \setminus D$, then $g_1$ and $g_2$ are not adjacent. So $\Gamma(G,D)$ is a 
$(v,k,\lambda ,\mu)$-strongly regular graph. 

Conversely, suppose $\Gamma(G,D)$ is a $(v,k,\lambda ,\mu)$-strongly regular 
graph. If $\Gamma(G,D)$ is 
undirected, then for vertices $g_1$ and $g_2$, there is an edge from $g_1$ to $g_2$ if and 
only if there is an edge from $g_2$ to $g_1$. By definition, $g_1$ and $g_2$ are 
connected by an edge if and only if $g_1 = d g_2$, $d \in D$. This means that 
$g_1 = d_1 g_2$ if and only if 
$g_2 = d_2 g_1$, for some $d_1, d_2 \in D$. This implies that $d_2 = d_1^{-1}$, so 
$D = D^{-1}$. Since $\Gamma(G,D)$ is $(v,k,\lambda ,\mu)$-strongly regular, it is 
$k$-regular, so the order of $D$ is $k$. Let $x$ be a vertex in $\Gamma(G,D)$ such that 
$x \in N(g_1) \cap N(g_2)$. Then $x = d_1 g_1 = d_2 g_2$ for some $d_1, d_2 \in D$, 
which implies that $d_1 d_2^{-1} = g_1^{-1} g_2$. If $g_1$ and $g_2$ are adjacent, 
then $g_1^{-1} g_2 \in D$, so there are exactly $\lambda$ ordered pairs $(d_1,d_2)$ 
that satisfy the previous equation. If $g_1$ and $g_2$ are not adjacent, then 
$g_1^{-1} g_2 \in G \setminus D$, so there are exactly $\mu$ ordered pairs $(d_1,d_2)$ 
that satisfy the equation. Therefore, $D$ is a $(v,k,\lambda ,\mu)$-PDS and $D = D^{-1}$. 
\end{proof}

For any graph $\Gamma=(V,E)$, let dist$:V\times V\to \zzz\cup \{\infty\}$ 
denote the distance function. In other words, for any $v_1,v_2\in V$,
dist$(v_1,v_2)$ is the length of the shortest path from $v_1$ to 
$v_2$ (if it exists) and $\infty$ (if it does not). The diameter of
$\Gamma$, denoted diam$(\Gamma)$, is the maximum value 
(possibly $\infty$) of this distance function.

\begin{definition}
{\rm
Let $\Gamma=(V,E)$ be a graph, let
dist$:V\times V \to \zzz$ denote the distance function,
and let $G={\rm Aut}(\Gamma)$ denote the automorphism group.
For any $v\in V$, and any $k\geq 0$, let

\[
\Gamma_k(v)=\{u\in V\ |\ {\rm dist}(u,v)=k\}.
\]
For any subset $S\subset V$ and any $u\in V$, let
$N_u(S)$ denote the subset of $s\in S$ which are
a neighbor of $u$, i.e., let

\[
N_u(S) = S\cap \Gamma_1(u).
\]
We say a graph is {\it distance transitive} if,
for any $k\geq 0$, and any $(u_1,v_1)\in V\times V$,
$(u_2,v_2)\in V\times V$ with dist$(u_2,v_2)=k$,
there is a $g\in G$ such that $g(u_2)=v_2$
and $g(u_1)=v_1$.

We say a graph is {\it distance regular} if for
for any $k\geq 0$ and any $(v_1,v_2)\in V\times V$
with dist$(v_1,v_2)=k$, the numbers

\[
a_k=|N_{v_1}(\Gamma_k(v_2)|,
\]
\[
b_k=|N_{v_1}(\Gamma_{k+1}(v_2)|,
\]
\[
c_k=|N_{v_1}(\Gamma_{k-1}(v_2)|,
\]
are independent of $v_1,v_2$.
}
\end{definition}
\index{distance transitive}
\index{distance regular}

\begin{remark}
The following ``conjecture'' is false:
If $f:GF(p)^n\to GF(p)$ is any even bent function then 
the (unweighted) Cayley graph of $f$ is distance transitive.
In fact, this fails when $p=2$ for any bent function 
of $4$ variables having support of size $6$.
Indeed, in this case the Cayley graph of $f$ is isomorphic to
the Shrikhande graph (with strongly regular parameters
$(16,6,2,2)$), which is not a distance-transitive
(see \cite{BrCN}, pp 104-105, 136).
\end{remark}

\begin{proposition}
{\rm
If $f:GF(p)^n\to GF(p)$ is any even function then 
each connected component of the (unweighted) Cayley graph of 
$f$ is distance regular.
}
\end{proposition}

\begin{proof}
First, we prove the following Claim: 

\[
\Gamma_k(0)=\{v\in GF(p)^n\ |\ v\ {\rm is\ the\ sum\ of}\ k\ {\rm support\ vectors,\ and\ no\ fewer}\}.
\]
We prove this by indiction. The statement for $k=1$ is obvious, since
$\Gamma_1(0)={\rm Supp}(f)$. Assume the statement is true for $k$. We prove it
for $k+1$. Let $v'\in \Gamma_{k+1}(0)$, so dist$(0,v')=k+1$. There is a 
$v''\in \Gamma_k(0)$ such that $v'=v''+v'''$, for some $v'''\in {\rm Supp}(f)$.
By the induction hypothesis, $v''$ can be written as the sum of $k$ support vectors,
so $v'$ is the sum of $k+1$ vectors (thus, proving the claim).

Claim: If $v\in GF(p)^n$ is arbitrary then

\[
\Gamma_k(v)=v+\Gamma_k(0),
\]
for $1\leq k\leq {\rm diam}(\Gamma)$. This follows from the definitions.

Claim: For all $i,j$ with $1\leq i,j\leq {\rm diam}(\Gamma)$, and for
all $u,v$ with $u,v\in GF(p)^n$, the cardinalities

\[
|(u+\Gamma_i(0))\cap (v+\Gamma_j(0))|,
\]
are independent of $u,v$.  This follows from the definitions.

From these claims, the Proposition follows.
\end{proof}

If $\Gamma$ is any weighted graph (without loops or multiple edges), 
we fix a labeling of its set of 
vertices $V(\Gamma)$, which we often identify with the set
$\{1,2,\dots, N=|V(\Gamma)|\}$. Moreover, we assume that the edge weights
of $\Gamma$ are positive integers. If $u,v$ are vertices of 
$\Gamma$, then we say a walk $P$ from $u$ to $v$ 
has {\it weight sequence}
\index{weight sequence}
$(w_1,w_2,\dots,w_k)$ if is there is a sequence of edges in $\Gamma$ 
connecting $u$ to $v$, $(v_0=u,v_1)$, $(v_1,v_2)$, \dots , 
$(v_{k-1},v_k=v)$, say, where edge $(v_{i-1},v_i)$ has weight $w_i$.
If $A=(a_{ij})$ denotes the $N\times N$ weighted adjacency matrix of $\Gamma$,
so

\[
a_{ij} = 
\left\{
\begin{array}{ll}
w, & {\rm if }\ (i,j)\ {\rm is\ an\ edge\ of \ weight}\ w,\\
0, & {\rm if }\ (i,j)\ {\rm is\ not\ an\ edge\ of \ }\Gamma.
\end{array}
\right.
\]
From this adjacency matrix $A$, we can derive weight-specific
adjacency matrices as follows. For each weight $w$ of $\Gamma$, let
$A(w)=(a(w)_{ij})$ denote the $N\times N$ $(1,0)$-matrix defined by

\[
a(w)_{ij} = 
\left\{
\begin{array}{ll}
1, & {\rm if }\ (i,j)\ {\rm is\ an\ edge\ of \ weight}\ w,\\
0, & {\rm if }\ (i,j)\ {\rm is\ not\ an\ edge\ of \ weight}\ w.
\end{array}
\right.
\]
Let us impose the following conventions. 
\begin{itemize}
\item
If $u,v$ are distinct 
vertices of $\Gamma$ but $(u,v)$ is not an edge of $\Gamma$ then 
we say the {\it weight} of $(u,v)$ is $w=0$. 
\item
If $u=v$ is a vertex of $\Gamma$ (so $(u,v)$ is not an edge,
since $\Gamma$ has no loops) then 
we say the {\it weight} of $(u,v)$ is $w=-1$. 
\end{itemize}
This allows us to define the weight-specific adjacency
matrices $A(-1),A(0)$ as well, and we can (and do) extend the weight set of 
$\Gamma$ by appending $0,-1$. Clearly, these weight-specific adjacency 
matrices have disjoint supports: if $a(w)_{ij}\not= 0$ 
then $a(w')_{ij}=0$ for all weights $w'\not= w$.
\index{edge weights}

The well-known matrix-walk theorem can be formulated as follows.

\begin{proposition}
{\rm
For any vertices
$u,v$ of $\Gamma$ and any sequence of non-zero edge weights
$w_1,w_2,\dots , w_k$, the $(u,v)$ of $A(w_1)A(w_2)\dots A(w_k)$ is equal to
the number of walks of weight sequence $(w_1,w_2,\dots,w_k)$ from
$u$ to $v$. Moreover, ${\rm tr}\, A(w_1)A(w_2)\dots A(w_k)$ is equal to
the total number of closed walks of $\Gamma$ of weight sequence 
$(w_1,w_2,\dots,w_k)$.
}
\end{proposition}

Let us return to describing the Cayley graph in (\ref{eqn:wtdCayleyGraph}) above.
We identify  $\zzz/p^n\zzz$ with $\{0,1,\dots, p^n-1\}$, and let

\begin{equation}
\label{eqn:pary-rep-map}
\eta:\zzz/p^n\zzz\to GF(p)^n
\end{equation}
be the $p$-ary representation map. In other words, if
we regard $x\in \zzz/p^n\zzz$ as a polynomial in
$p$ of degree $\leq n-1$, then $\eta(x)$ is the
list of coefficients, arranged in order of 
decreasing degree.
This is a bijection. (Actually, for our purposes, any bijection will
do, but the $p$-ary representation is the most natural one.) 

\begin{lemma}
% Let $f$ be an even $GF(p)$-valued function on $V$. 
A graph $\Gamma$ having vertices $V$ is a 
(edge-weighted) Cayley graph for some 
even $GF(p)$-valued function $f$ on $V$ with $f(0)=0$ if and only if
$\Gamma$ is regular and the adjacency matrix $A=(A_{ij})$ 
of $\Gamma$ has the following properties:
(a) $A_{0,i}=1$ if and only if $f(\eta (i))\not= 0$,
(b) $A_{i,j}=1$ if and only if $A_{0,k}=1$, where 
$\eta (k)=\eta (i)-\eta (j)$.
\end{lemma}
\index{Cayley graph!edge-weighted}

\begin{proof}
Let $w\in GF(p)$. We know that $A_{i,j}=w$ if and only if 
there is an edge of weight $w$ from
$\eta (i)$ to $\eta (j)$ if and only if $f(\eta (i)-\eta (j))=w$.
The result follows.
\end{proof}

We assume, unless stated otherwise, that $f$ is even.
For each $u\in V$, define
\begin{itemize}
\item
$N(u)=N_{\Gamma_f}(u)$ to be the set of all
neighbors of $u$ in $\Gamma_f$,
\item
$N(u,a)=N_{\Gamma_f}(u,a)$ to be the set of all
neighbors $v$ of $u$ in $\Gamma_f$ for which the edge
$(u,v)\in E_f$ has weight $a$ (for each 
$a\in GF(p)^\times = GF(p)-\{0\}$),
\item
$N(u,0)=N_{\Gamma_f}(u,0)$ to be the set of all
non-neighbors $v$ of $u$ in $\Gamma_f$ (i.e., 
we have $(u,v)\notin E_f$),
\item
$
{\rm supp}(f)=\{v\in V\ |\ f(v)\not=0\}
$
to be the {\it support} of $f$. 
\end{itemize}
\index{support}
It is clear that
${\rm supp}(f)=N(0)$ is the set of neighbors of  
the zero vector. More generally, for any $u\in V$,

\begin{equation}
\label{eqn:nbhd1}
N(u)=u+{\rm supp}(f),
\end{equation}
where the last set is the collection of all vectors 
$u+v$, for some $v\in {\rm supp}(f)$.

We call a map
$g:GF(p)^n\to GF(p)$ {\it balanced} if the cardinalities
$|g^{-1}(x)|$ ($x\in GF(p)$) do not depend on $x$.
\index{balanced}
We call the {\it signature} of $f:GF(p)^n\to GF(p)$ the list 

\[
|S_0|, \ 
|S_1|, \ |S_2|, \dots, |S_{p-1}|,
\]
where, for each $i$ in $GF(p)$,
\index{signature} 

\begin{equation}
\label{eqn:support-i}
S_i = \{ x\ |\ f(x)=i\}.
\end{equation}
We can extend equation (\ref{eqn:nbhd1}) to the more precise statement

\begin{equation}
\label{eqn:nbhd2}
N(u,a)=u+S_a,
\end{equation}
for all $a\in GF(p)$.
We call $N(u,a)$ the {\it $a$-neighborhood} of $u$.
\index{neighborhood} 

A connected simple graph $\Gamma$ (without edge weights) 
is called {\it strongly regular} 
% with parameters $b,c,d, e$, denoted by $SRG(b,c,d,e)$, 
if it consists of $v$ vertices such that

\[
|N(u_1) \cap N(u_2)| =
\left\{
\begin{array}{ll}
k, & u_1=u_2,\\
\lambda, & u_1\in N(u_2),\\
\mu, & u_1\notin N(u_2).\\
\end{array}
\right.
\]
In the usual terminology/notation, such a graph is said
to have parameters $srg(\nu, k, \lambda, \mu)$.
\index{SRG}
\index{strongly regular graph} 

\begin{remark}
\label{remark:PDS2}
Let

\[
V = GF(p)^n,\ \ \ \ D=\supp(f),\ \ \ D'=S_0-\{0\}.
\]
These sets, because $f$ is even, have the property that
$D^{-1}=D$, $(D')^{-1}=D'$.
For each $d\in D$, let

\[
\lambda_d = |\{(g,h)\in D\times D\ |\ g-h=d\}|,
\]
and, for each $d'\in D'$, let

\[
\mu_{d'} = |\{(g,h)\in D\times D\ |\ g-h=d'\}|.
\]
It is known that if $D$ is a partial difference set (PDS) on the additive
group of $V$ then
(a) $\lambda_d$ does not depend on $d\in D$ (the common value
is denoted $\lambda$), and
(b) $\mu_{d'}$ does not depend on $d'\in D'$ 
(the common value is denoted $\mu$).

See Theorem \ref{thrm:PDSiffSRG}
for an equivalence between Cayley graphs of PDSs and 
strongly regular graphs.

Let $k=|D|$ and $\nu=|V|$.
Since $g-h\in D$ if and only if $f(g-h)\not= 0$ (for distinct
$g,h\in V$), there are $k\lambda$ non-neighbors
in $V^2$. Likewise, since $g-h\in D'$ if and only if 
$f(g-h)\not= 0$ (for distinct $g,h\in V$), there are 
$(\nu - k-1)\mu$ neighbors in $V^2$. Therefore, since vertex pairs 
must be neighbors or non-neighbors,

\begin{equation}
\label{eqn:pds2}
k^2-k = k\lambda + (\nu - k-1)\mu .
\end{equation}

\end{remark}

The concept of strongly regular simple graphs generalizes to 
edge-weighted graphs. 

%Let $W = GF(p)^\times= GF(p)-\{0\}$ be the set of edge weights.

\begin{definition} ({\bf edge-weighted SRG})
\label{def:SRG}
{\rm
Let $\Gamma$  be a connected edge-weighted graph which is regular
as a simple (unweighted) graph.
The graph $\Gamma$ is called {\it strongly regular} 
with parameters $v$, $k=(k_a)_{a\in W}$, 
$\lambda=(\lambda_a)_{a\in W^3}$, $\mu=(\mu_a)_{a\in W^2}$,
denoted $SRG_{W}(v,k,\lambda,\mu)$, if it consists of $v$ vertices
such that, for each $a=(a_1,a_2)\in W^2$ 

\begin{equation}
\label{eqn:wtd-srg}
|N(u_1,a_1) \cap N(u_2,a_2)| =
\left\{
\begin{array}{ll}
k_{a}, & u_1=u_2,\\
\lambda_{a_1,a_2,a_3}, & u_1\in N(u_2,a_3),\ u_1\not= u_2,\\
\mu_{a}, &u_1\notin N(u_2),\ u_1\not= u_2,\\
\end{array}
\right.
\end{equation}
where $k=(k_a\ |\ a\in W)\in \zzz^{|W|}$,
$\lambda=(\lambda_a\ |\ a\in W^3)\in \zzz^{|W^3|}$,
$\mu=(\mu_a\ |\ a\in W^2)\in \zzz^{|W^2|}$,
and $W= GF(p)$ is the set of weights,
including $0$ (recall an ``edge'' has weight $0$ if the vertices are not 
neighbors).
}
\end{definition}
\index{SRG!edge-weighted}

How does the above notion of an edge-weighted strongly 
regular graph relate to the 
usual notion of a strongly regular graph?

\begin{lemma}
\label{lemma:wtdsrg}
Let $\Gamma$ be an edge-weighted strongly regular 
graph as in (\ref{def:SRG}), with edge-weights $W$ and 
parameters $(v,(k_a),(\lambda_{a_1,a_2,a_3}), (\mu_{a_1,a_2}))$.
If
\[
\sum_{(a_1,a_2) \in W^2} \lambda_{a_1,a_2,a_3}
\]
does not depend on $a_3$, for $a_3 \in W$,
then $\Gamma$ is strongly regular (as an 
unweighted graph) with
parameters $(v,k,\lambda,\mu)$ where 

\[
k = \sum_{a \in W} k_{a},\ \ 
\lambda = \sum_{(a_1,a_2) \in W^2} \lambda_{a_1,a_2,a_3},\ \ \ \
\mu = \sum_{(a_1,a_2) \in W^2} \mu_{a_1,a_2}.
\]
\end{lemma}

The proof follows directly from the definitions.

Let $(G,D)$ be a symmetric weighted PDS.

\begin{definition}
\label{defn:wtdCG2}
{\rm
The {\it edge-weighted Cayley graph $\Gamma=\Gamma(G,D)$ 
associated to the symmetric weighted PDS} $(G,D)$ 
is the edge-weighted graph constructed as follows. Let the 
vertices of the graph be the elements of the group $G$. 
Two vertices $g_1$ and $g_2$ are 
connected by an edge of 
weight $i$ if $g_2 = d g_1$ for some $d \in D_i$. Since 
$(G,D)$ is symmetric, the graph $\Gamma$ is undirected.
}
\end{definition}

\begin{remark}
This notion of an edge-weighted strongly regular graph 
differs slightly from the 
notion of a strongly regular graph
decomposition in \cite{vD}, in which the 
individual graphs of the decomposition must 
each be strongly regular.  

\end{remark}

\begin{definition}
We say that an edge-weighted strongly regular graph is 
{\it amorphic} if it's corresponding association scheme is
amorphic in the sense of \cite{CP}.
\end{definition}

The following result is due to van Dam \cite{vD}
(see \cite{CP}).

\begin{proposition}
(van Dam)
\label{prop:vandam}
Let $f:GF(p)^n\to GF(p)$ be an even bent function with $f(x)=0$
if and only if $x=0$.
If the weighted Cayley graph of $f$, $\Gamma_f$, is an edge-weighted strongly regular 
amorphic graph then $\Gamma_f$ has a strongly regular
decomposition into subgraphs $\Gamma_i$ all of whose edges have weight
$i$ (where $i\in GF(p)$, $i\not= 0$), and each $\Gamma_i$ is, as
an unweighted graph, a strongly regular graph of either
Latin square type or of negative Latin square type.
\end{proposition}

The weighted adjacency matrix $A$ of the Cayley graph of $f$
is the matrix whose entries are 

\[
        A_{i,j} = f(\eta(i) - \eta(j)), 
\]
where $\eta(k)$ is the $p$-ary representation as in 
(\ref{eqn:pary-rep-map}).
Note $\Gamma_f$ is a regular digraph (each vertex has the same
in-degree 
and the same out-degree as each other vertex). The in-degree and the out-degree
both equal $wt(f)$, where $wt$ denotes
the Hamming weight of $f$, when regarded as a vector of integer values
(of length $p^n$). Let

\[
\omega=\omega_f = wt(f)
\]
denote the cardinality of 
$\supp(f) = \{v\in V\ |\ f(v)\not= 0\}$.
Note that $\hat{f}(0)=\omega\geq |\supp(f)|$.
If $f$ is even then $\Gamma_f$ is an $\omega$-regular graph.

If $A$ is the adjacency matrix of a (simple, unweighted) strongly
regular graph having parameters $(v,k,\lambda,\mu)$ then

\begin{equation}
\label{eqn:schur-ring1}
A^2 = kI+\lambda A + \mu (J-I-A),
\end{equation}
where $J$ is the all $1$s matrix and $I$ is the identity matrix.
This is relatively easy to verify, by simply computing
$(A^2)_{ij}$ in the three separate cases (a) $i=j$, (b) $i\not= j$
and $i,j$ adjacent, (c) $i\not= j$ and $i,j$ 
non-adjacent\footnote{
It can also be proven by character-theoretic methods, but this 
method seems harder to generalize to the edge-weighted case.
}.

If $A$ is the adjacency matrix of an edge-weighted strongly
regular graph having parameters $(v,k_a,\lambda_{a_1,a_2,a_3},\mu_{a_1,a_2})$ 
and positive weights $W\in \zzz$
one can compute $(A^2)_{ij}$ explicitly, again by looking at 
the three separate cases (a) $i=j$, (b) $i\not= j$
and $i,j$ adjacent, (c) $i\not= j$ and $i,j$ non-adjacent.
We obtain

\begin{equation}
(A^2)_{i,j}
=
\left\{
\begin{array}{ll}
\sum_{a\in W} a^2k_a, & i=j,\\
\sum_{(a,b)\in W^2} ab\lambda_{(a,b,c)}, & i\not=j, i\in N(j,c),\\
\sum_{(a,b)\in W^2} ab\mu_{(a,b)}, & i\not=j, i\notin N(j).\\
\end{array}
\right.
\end{equation}

As the following lemma illustrates, it is very easy to
characterize Cayley graphs in terms of its adjacency matrix.

\begin{lemma}
% Let $f$ be an even $GF(p)$-valued function on $V$. 
A graph $\Gamma$ having vertices $V$ is a Cayley graph for some 
 even $GF(p)$-valued function $f$ on $V$ with $f(0)=0$ if and only if
$\Gamma$ is regular and the adjacency matrix $A=(a_{ij})$ 
of $\Gamma$ has the following properties: for each $w\in GF(p)$,
(a) $a_{1,i}=w$ if and only if $f(\eta (i))=w$,
(b) $a_{i,j}=w$ if and only if $a_{1,k}=w$, where 
$\eta (k)=\eta (i)-\eta (j)$.
\end{lemma}

This statement follows from the definitions and its proof is omitted.

Note that

\[
W_f(0)=|S_0|+|S_1|\zeta+\dots +|S_{p-1}|\zeta^{p-1},
\]
which we can regard as an identity in the 
$(p-1)$-dimensional $\qqq$-vector space
$\qqq(\zeta)$. The relation

\[
1+\zeta+\zeta^2+\dots +\zeta ^{p-1}=0,
\]
gives

\[
\begin{array}{c}
W_f(0)-|S_0|+|S_1|\\
=(|S_2|-|S_1|)\zeta^2+\dots
+(|S_{p-1}|-|S_1|)\zeta^{p-1}.
\end{array}
\]
We have proven the following result.

\begin{lemma}
\label{lemma:sig}
If $f:GF(p)^n\to GF(p)$ has the property that $W_f(0)$ is a rational
number then

\[
|S_1| =  |S_2| = \dots =|S_{p-1}|,
\]
and

\[
W_f(0)=|S_0|-|S_1|.
\]
In particular,

\[
\begin{array}{rl}
|{\rm supp}(f)|&=|S_1| +|S_2|+ \dots + |S_{p-1}|\\
&=(p-1) |S_1|=(p-1)  (|S_0|-W_f(0)).
\end{array}
\]
\end{lemma}

\begin{remark}
It is also known that if $n$ is even and $f$ is bent then

\[
|S_1| =  |S_2| = \dots =|S_{p-1}|.
\]
\end{remark}

We have more to say about these sets later.

\subsection{Cayley graphs of bent functions}

For example, the Cayley graph of the even bent function in 
Example \ref{example:bent1} is given in Figure %\ref{fig:bent1}
%(see also Figure 
\ref{fig:bent1-tikz}.

\begin{figure}[t!]
\begin{minipage}{\textwidth}
\begin{center}
%This latex requires the tikz package
\definecolor{lorange}{rgb}{1.0,0.5,0}
\begin{center}
  \begin{tikzpicture}
    [font=\scriptsize,
    node/.style={shape=circle,draw=black,minimum width=0.5cm,thick,fill=lorange},
    edge1/.style={thick,red},
    edge2/.style={very thick,blue},
    edge/.style={thick}]

    \node (0) [node] at (0,4) {0};
    \node (1) [node] at (-3.5,3) {1};
    \node (2) [node] at (-6,1) {2};
    \node (3) [node] at (-5.5,-1) {3};
    \node (4) [node] at (-2.333,-3) {4};
    \node (5) [node] at (2.666,-3) {5};
    \node (6) [node] at (5.5,-1) {6};
    \node (7) [node] at (6,1) {7};
    \node (8) [node] at (3.5,3) {8};

    \draw [edge1] (0) to node[black]{$1$} (1);
    \draw [edge1] (0)  to node[black]{$1$} (2);
    \draw [edge1] (0)  to node[black]{$1$} (3);
    \draw [edge2] (0)  to node[black]{$2$} (4);
    \draw [edge2] (0)  to node[above left,black,pos=.7]{$2$} (5);
    \draw [edge1] (0)  to node[black]{$1$} (6);
    \draw [edge2] (0)  to node[above left,black,pos=.32]{$2$} (7);
    \draw [edge2] (0) -- node[above left,black,pos=.6]{$2$} (8);

    \draw [edge1] (1)  to node[black]{$1$} (2);
    \draw [edge2] (1)  to node[above left,black,pos=.42]{$2$} (3);
    \draw [edge1] (1)  to node[black]{$1$} (4);
    \draw [edge2] (1)  to node[above left,black,pos=.6]{$2$} (5);
    \draw [edge2] (1)  to node[above left,black,pos=.65]{$2$} (6);
    \draw [edge1] (1)  to node[black]{$1$} (7);
    \draw [edge2] (1)  to node[black]{$2$} (8);

    \draw [edge2] (2)  to node[black]{$2$} (3);
    \draw [edge2] (2)  to node[black]{$2$} (4);
    \draw [edge1] (2)  to node[above left,black,pos=.65]{$1$} (5);
    \draw [edge2] (2)  to node[above left,black,pos=.38]{$2$} (6);
    \draw [edge2] (2)  to node[black]{$2$} (7);
    \draw [edge1] (2)  to node[black]{$1$} (8);

    \draw [edge1] (3)  to node[black]{$1$} (4);
    \draw [edge1] (3)  to node[above left,black,pos=.52]{$1$} (5);
    \draw [edge1] (3)  to node[black]{$1$} (6);
    \draw [edge2] (3)  to node[above left,black,pos=.35]{$2$} (7);
    \draw [edge2] (3)  to node[above left,black,pos=.65]{$2$} (8);

    \draw [edge1] (4)  to node[black]{$1$} (5);
    \draw [edge2] (4)  to node[above left,black,pos=.43]{$2$} (6);
    \draw [edge1] (4)  to node[above left,black,pos=.6]{$1$} (7);
    \draw [edge2] (4)  to node[above left,black,pos=.7]{$2$} (8);

    \draw [edge2] (5)  to node[black]{$2$} (6);
    \draw [edge2] (5)  to node[black]{$2$} (7);
    \draw [edge1] (5) to node[black]{$1$} (8);

    \draw [edge1] (6) to node[black]{$1$} (7);
    \draw [edge1] (6) to node[above left,black,pos=.6]{$1$} (8);

    \draw [edge1] (7) to node[black]{$1$} (8);

  \end{tikzpicture}
\end{center}
%\drawpng{height=9cm,width=9cm}{hadamard_transform_GF3-bent1-FT2}
\end{center}
\end{minipage}
\caption{The undirected Cayley graph of an even $GF(3)$-valued bent function of two
  variables from Example \ref{example:bent1}. (The vertices are
  ordered as in the Example.))
}
\label{fig:bent1-tikz}
\end{figure}
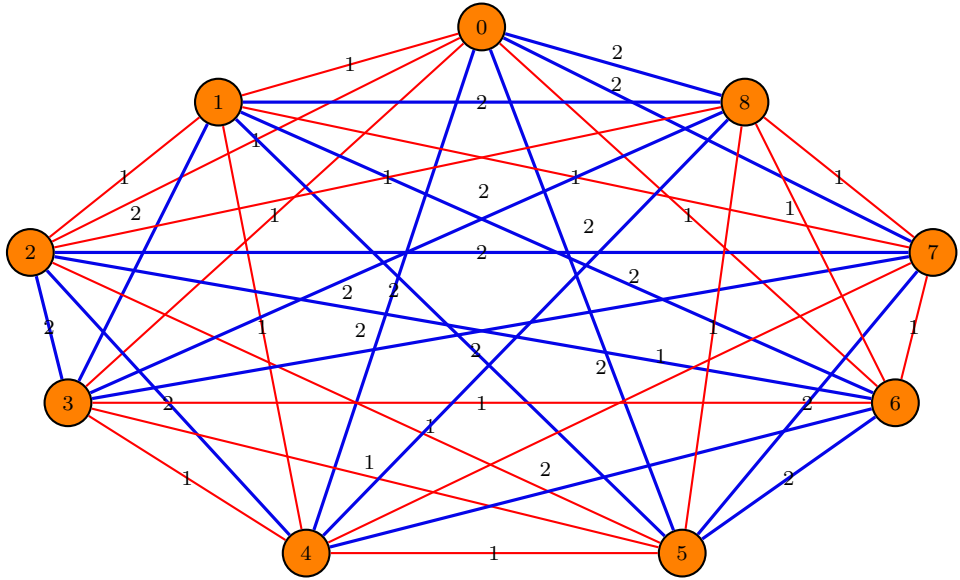

%\begin{figure}[t!]
%\begin{minipage}{\textwidth}
%\begin{center}
%\includegraphics[height=9cm,width=12cm]{images/hadamard_transform_GF3-bent1}
%\end{center}
%\end{minipage}
%\caption{The undirected Cayley graph of an even $GF(3)$-valued bent function of two
%  variables from Example \ref{example:bent1}. (The vertices are
%  ordered as in the Example.)}
%\label{fig:bent1}
%\end{figure}

\begin{remark}
\label{remark:PDS}
In Chee et al \cite{CTZ}, it is shown that if $n$ is even then the
unweighted Cayley graph of certain\footnote{By ``certain'' we mean that
$f$, regarded as a function $GF(p^n)\to GF(p)$, is homogeneous
of some degree.} weakly regular even bent functions 
$f:GF(p)^n\to GF(p)$, with $f(0)=0$, is strongly regular.

\end{remark}

\begin{problem}
\label{prob:cayleygraph1to4}
Some natural problems arise. For $f$ even, 

\begin{enumerate}
\item
find necessary and sufficient conditions for $\Gamma_f$ to be strongly
regular,
\item
find necessary and sufficient conditions for $\Gamma_f$ to be
connected (and more generally find a formula for the number of
connected components of $\Gamma_f$),
\item
classify the spectrum of $\Gamma_f$ in terms of the values of the
Fourier transform of $f$,
\item
in general, which graph-theoretic properties of $\Gamma_f$ can be tied
to function-theoretic properties of $f$?
\end{enumerate}
\end{problem}

\begin{theorem}
(Bernasconi Correspondence, \cite{B}, \cite{BC}, \cite{BCV})
\label{thrm:BC} 
Let $f:GF(2)^n\to GF(2)$. The function $f$ is bent if and only if 
the Cayley graph of $f$ is a strongly regular graph having
parameters $(2^n, k, \lambda, \lambda)$ for some $\lambda$, 
where $k=|\supp(f)|$.
\end{theorem}
\index{Bernasconi Correspondence}

The (naive) analog of this for $p>2$ is formalized below in 
Analog \ref{analog:BC-SRG}

Regarding Problem 1, we have the following natural expectation.

Regarding the Bernasconi correspondence, we have the 
following graph-theoretical generalization (whose statement 
may or may not be true).

\begin{analog}
\label{conj:weighted-cayley}
\label{analog:BC-SRG}
Assume $n$ is even.
If $f:GF(p)^n\to GF(p)$ is even bent then, for each $a\in GF(p)^\times$, we
have

\begin{itemize}
\item
if $u_1,u_2\in V$ are $a_3$-neighbors in the Cayley graph of $f$ 
then
$|N(u_1,a_1) \cap N(u_2,a_2)|$ does not depend on $u_1,u_2$
(with a given edge-weight), for each $a_1,a_2,a_3\in GF(p)^\times$;
\item
if $u_1,u_2\in V$ are distinct and not neighbors in the Cayley graph of $f$ 
then $|N(u_1,a_1) \cap N(u_2,a_2)|$ does not depend on $u_1,u_2$, 
for each $a_1,a_2\in GF(p)^\times$.
\end{itemize}
In other words, the associated Cayley graphs is edge-weighted 
strongly regular as in Definition \ref{def:SRG}.
\end{analog}

\index{Bernasconi Correspondence!analog}

Unfortunately, it is not true in general.

\begin{remark}
\begin{enumerate}
\item
This analog is false when $p=5$.

\item
This analog remains false if you replace
``$f:GF(p)^n\to GF(p)$ is even bent'' in the hypothesis by
``$f:GF(p)^n\to GF(p)$ is even bent and regular.''
However, when $p=3$ and $n=2$, see Lemma \ref{lemma:34}(a).

\item
In general, this analog remains false if you replace
``$f:GF(p)^n\to GF(p)$ is even bent'' in the hypothesis by
``$f:GF(p)^n\to GF(p)$ is even bent and weakly regular.''
However, when $p=3$ and $n=2$, see Lemma \ref{lemma:34}(b).

\item
This analog is false if $n$ is odd.

\item
The converse of this analog, as stated, is false if $p>2$.

\end{enumerate}
\end{remark}

Parts 2 and 3 in Problem \ref{prob:cayleygraph1to4}
are addressed below (see Lemmas \ref{lemma:connected1}
and \ref{lemma:connected2}, resp., and \S \ref{sec:FTs}).

The adjacency matrix $A=A_f$ is the matrix whose entries are 

\[
        A_{i,j} = f_\ccc(\eta(i) - \eta(j)), 
\]
where $\eta(k)$ is the $p$-ary representation as in 
(\ref{eqn:pary-rep-map}). Ignoring edge weights, we let

\[
        A^{*}_{i,j} = 
\left\{
\begin{array}{ll}
1, &f_\ccc(\eta(i) - \eta(j))\not=0,\\
0, &{\rm otherwise}.
\end{array}
\right. 
\]
Note $\Gamma_f$ is a regular edge-weighted 
digraph (each vertex has the same in-degree 
and the same out-degree as each other vertex). 
The in-degree and the out-degree
both equal $wt(f)$, where $wt$ denotes
the Hamming weight of $f$, when regarded as a vector (of length $p^n$) of 
integers. Let

\[
\omega=\omega_f = wt(f)
\]
denote the cardinality of 
$\supp(f) = \{v\in V\ |\ f(v)\not= 0\}$
and let

\[
\sigma_f = \sum_{v\in V} f_\ccc(v).
\]
Note that $\hat{f}(0)=\sigma_f\geq |\supp(f)|$.
If $f$ is even then $\Gamma_f$ is an $\sigma_f$-regular 
(edge-weighted) graph. If we ignore weights, then it is
an $\omega_f$-weighted graph.

Recall that, given a graph $\Gamma$ and its adjacency matrix $A$, 
the spectrum 
$\sigma(\Gamma)=\{\lambda_1,\lambda_2,\dots,\lambda_N\}$,
where $N=p^n$, is the multi-set of eigenvalues of $A$. 
Following a standard convention,
we index the elements $\lambda_i=\lambda_i(A)$
of the spectrum in such a way that they are monotonically increasing
(using the lexicographical ordering of $\ccc$).
Because $\Gamma_f$ is regular, the row sums of $A$ are all $\omega$
whence the all-ones vector is an eigenvector of $A$ with eigenvalue $\omega$.
We will see later (Corollary \ref{corollary:spectrum}) that 
$\lambda_N(A)=\sigma_f$.

Let $D$ denote the identity matrix multiplied by $\sigma_f$. 
The {\it Laplacian} of
$\Gamma_f$ can be defined as the matrix $L = D - A$. 
\index{Laplacian} 

\begin{lemma}
\label{lemma:connected1}
Assume $f$ is even.
As an edge-weighted graph, 
$\Gamma_f$ is connected if and only if 
$\lambda_{N-1}(A)<\lambda_N(A)=\sigma_f$.
If we ignore edge weights, then
$\Gamma_f$ is connected if and only if 
$\lambda_{N-1}(A^*)<\lambda_N(A^*)=\omega_f$.
\end{lemma}

\begin{proof}
We only prove the statement for the edge-weighted case.

Note that for $i = 1,\ldots, N$,
$\lambda_i(L) = \omega - \lambda_{N-i+1}(A)$, since
$\det (L-\lambda I) = \det(\sigma_f I-A-\lambda I)
=(-1)^n\det (A -(\sigma_f-\lambda)I)$.
Thus, $\lambda_i(L) \geq 0$, for all $i$. By a theorem of Fiedler 
\cite{Fiedler73}, 
$\lambda_2(L) > 0$ if and only if $\Gamma_f$ is connected. But
$\lambda_2(L) > 0$ is equivalent to $\sigma_f - \lambda_{N-1}(A) > 0$.
\end{proof}

Clearly, the vertices in $\Gamma_f$ connected to $0\in V$
is in natural bijection with $\supp(f)$. Let $W_j$ denote the 
subset of $V$ consisting of those vectors which can be 
written as the sum of $j$ elements in $\supp(f)$ but not
$j-1$. Clearly,

\[
W_1 = \supp(f) \subset W_2 \subset \dots \subset Span(\supp(f)).
\]

For each $v_0\in W_1=\supp(f)$, the vertices connected to
$v_0$ are the vectors in

\[
\supp(f_{v_0}) = \{v\in V\ |\ f(v-v_0)\not= 0\},
\]
where $f_{v_0}(v) = f(v-v_0)$ denotes the translation of $f$ by
$-v_0$. Therefore, 

\[
\supp(f_{v_0}) = v_0+\supp(f) .
\]
In particular, all the vectors in $W_2$ are connected to $0\in V$.
For each $v_0\in W_2$, the vertices connected to $v_0$ are the vectors in
$
\supp(f_{v_0}) = v_0+\supp(f) ,
$
so all the vectors in $W_3$ are connected to $0\in V$.
Inductively, we see that $Span(\supp(f))$ is the connected component
of $0$ in $\Gamma_f$.
Pick any $u \in V$ representing a non-trivial coset in
$V/Span(\supp(f))$. Clearly, $0$ is not connected with 
$u$ in $\Gamma_f$. However, the above reasoning implies
$u$ is connected to $v$ if and only if they
represent the same coset in $V/Span(\supp(f))$.
This proves the following result.

\begin{lemma}
\label{lemma:connected2}
The connected components of $\Gamma_f$ are in one-to-one
correspondence with the elements of the quotient space
$V/Span(\supp(f))$.
\end{lemma}

\subsection{Group actions on bent functions}

We note here some useful facts about the action of nondegenerate linear 
transforms on $p$-ary functions.  Suppose that  $f:V=GF(p)^n \to GF(p)$ 
and $\phi:V \to V$ is a nondegenerate linear transformation (isomorphism 
of $V$), and $g(x)=f(\phi(x))$.  The functions $f$ and $g$ both have the same 
signature, $(|f^{-1}(i)|\ |\ i = 1,\dots,p-1)$.

It is straightforward to calculate that 
$$W_f g(u) = W_f ( (\phi^{-1})^T u)$$  
(where $T$ denotes transpose).

It follows that if $f$ is bent, so is $g=f\circ \phi$, and if $f$ is bent and regular, so is $g$.
If $f$ is bent and weakly regular, with $\mu$-regular dual $f^*$, then $g$ 
is bent and weakly regular, with $\mu$-regular dual $g^*$, where $g^*(u)=f^*( (\phi^{-1})^T u)$.

Next, we examine the effect of the group action on bent functions and the corresponding 
weighted PDSs.

\begin{proposition}
{\rm
Let $f:GF(p)^{n} \rightarrow GF(p)$ be an even, bent function such that $f(0)=0$ 
and define $D_{i} = f^{-1}(i)$ for $i \in GF(p)-\{0\}$. 
Suppose $\phi :GF(p)^{n} \rightarrow GF(p)^{n}$ is a linear map that is invertible 
(i.e., det $\phi \neq 0 \pmod{p}$). Define the function $g = f \circ \phi$; $g$ 
is the composition of a bent function and an affine function, so it is also bent.
If the collection of sets 
$\{D_{1}, D_{2}, \cdots ,D_{p-1}\}$ forms a weighted partial difference
set for $GF(p)^{n}$ then so does its image under the 
function $\phi$.
}
\end{proposition}

\begin{proof}
We can explore this question by utilizing the Schur ring generated by the sets $D_{i}$.
Define $D_0=\{0\}$, where $0$ denotes the zero vector in $GF(p)^n$, and
define $D_p=GF(p)^n-\cup_{0\leq i\leq p-1} D_i$.

($D_0,D_{1}, D_{2}, \cdots ,D_{p-1},D_p$) forms a weighted partial difference set for $GF(p)^{n}$ 
if and only if ($C_{0}, C_{1}, C_{2}, \cdots ,C_{p}$) forms a Schur ring in $\ccc [GF(p)^n]$, where
\begin{center}
$C_{0}$ = $\{0\}$ (where $0$ denotes the zero element of $\ccc [GF(p)^n]$),
\end{center}
\begin{center}
$C_{1} = D_{1}, \cdots, C_{p-1} = D_{p-1}$
\end{center}
\begin{center}
$C_{p} = GF(p)^{n} - (C_{0} \cup \cdots \cup C_{p-1})$
\end{center}
\begin{center}
$C_{i} \cdot C_{j} = \sum\limits_{k=0}^{p} \rho_{ij}^{k} C_{k}$,
\end{center}
for some intersection numbers $\rho_{ij}^k\in \zzz$.
Note that $f$ is even, so $C_{i} = C_{i}^{-1}$ for all $i$, where $C_{i}^{-1} = \{-x\ |\ x \in C_{i}\}$.
Define $S_{i} = g^{-1}(i) = \{v \in GF(p)^{n}:g(v)=i\}$. 
$D_{i} = f^{-1}(i) = (g \circ \phi^{-1})^{-1}(i) = (\phi \cdot g^{-1})(i) = \phi(S_{i})$. 
So the map $\phi$ sends $S_i$ to $D_i$. $\phi$ can be extended to a map from 
$\ccc [GF(p)^n] \rightarrow \ccc [GF(p)^n]$ such that 
$\phi (g_1 + g_2) = \phi(g_1) + \phi(g_2)$ and $\phi(S_i) = D_i$. So $\phi$ is 
a homomorphism from the Schur ring of $g$ to the Schur ring of $f$. 
Therefore, the level curves of $g$ give rise to a Schur ring, and the
weighted partial difference set generated by $f$ is sent to a weighted 
partial difference set generated by $g$ under the map $\phi^{-1}$. 
We conclude that the Schur ring of $g$ corresponds to a  weighted partial difference set
for $GF(p)^{n}$, which is the image of that for $f$.
\end{proof}

\begin{remark}
It is known that for ``homogeneous'' weakly regular bent 
functions\footnote{Here, ``homogeneous'' is meant in the
sense of \cite{PTFL}, not in the sense we use in this paper.},
the level curves give rise to a weighted PDS. In fact, the weighted
PDS corresponds to an association scheme and the dual association
scheme corresponds to the dual bent function (see \cite{PTFL},
Corollary 3, and \cite{CTZ}). We know that any bent
function equivalent to such a bent function also has this property,
thanks to the proposition above.
\end{remark}

Our data seems to support the following statement.

\begin{conjecture}
\label{conjecture:main}
Let $f:GF(p)^n\to GF(p)$ be an even bent function, with
$p>2$ and $f(0)=0$. If the level curves of $f$ give rise to a weighted 
partial difference set\footnote{In the sense of 
Remark \ref{remark:level-curves}.} 
then $f$ is homogeneous and weakly regular.
\end{conjecture}

\section{Intersection numbers}

This section is devoted to stating some results on the 
$p_{ij}^k$'s.

\begin{theorem}
\label{thrm:pijk-formula}
{\rm
Let $f:GF(p)^n \rightarrow GF(p)$ be a function and let $\Gamma$ be
its Cayley graph. 
Assume $\Gamma$ is a weighted strongly regular graph. Let $A =
(a_{k,l})$ be the 
adjacency matrix of $\Gamma$. Let $A_i = (a^i_{k,l})$ be the $(0,1)$-matrix where 
\begin{center}
$a^i_{k,l} = 
\begin{cases}
1 & \text{if } a_{k,l} = i \\
0 & \text{otherwise}
\end{cases}
$
\end{center}
for each $i = 1,2,\dots ,p-1$. Let $A_0$ be the $p^n \times p^n$
identity matrix. 
Let $A_p$ be the $(0,1)$-matrix such that $A_0 + A_1 + \cdots +
A_{p-1} + A_p = J$, 
the $p^n \times p^n$ matrix with all entries 1. Let $R$ denote the
matrix ring generated 
by $ \{A_0, A_1, \cdots ,A_p\}$. The intersection numbers $p^k_{ij}$ defined by

\begin{equation}
\label{eqn:AiAj}
A_iA_j = \sum\limits_{k=0}^p p^k_{ij}A_k
\end{equation}
satisfy the formula

\[
p^k_{ij} = \left(\dfrac{1}{p^n|D_k|}\right) Tr(A_iA_jA_k),
\]
for all $i,j,k = 1,2,\dots,p$.
}
\end{theorem}

This is (17.13) in \cite{CvL}. We provide a different proof for the reader's convenience.

\begin{proof}
By the Matrix-Walk Theorem, $A_iA_j$ can be considered as counting
walks along the Cayley graph of specific edge weights. Supposed
$(u,v)$ is an edge of $\Gamma$ with weight $k$. If $k=0$, then $u=v$
and the edge is a loop. If $k=p$, then $(u,v)$ is technically not an
edge in $\Gamma$, but 
we will label it as an edge of weight $p$. 
\\
\\
The $(u,v)$-th entry of $A_iA_j$ is the number of walks of length 2
from $u$ to $v$ where the first 
edge has weight $i$ and the second edge has weight $j$; the entry is 0
if no such walk exists. 
If we consider the $(u,v)$-th entry on each side of the equation 
(\ref{eqn:AiAj})
%\begin{center}
%$A_iA_j = \sum\limits_{k=0}^p p^k_{ij}A_k$
%\end{center}
we can deduce that $p^k_{ij}$ is the number of walks of length 2 from
$u$ to $v$ where the first 
edge has weight $i$ and the second edge has weight $j$ (it equals 0 if
no such walk exists) 
for any edge $(u,v)$ with weight $k$ in $\Gamma$. 
\\
\\
Similarly, the Matrix-Walk Theorem implies that $Tr(A_iA_jA_k)$ is the
total number of walks of length 3 
having edge weights $i,j,k$. We claim that if $\triangle$ is any
triangle with edge weights $i,j,k$, then by 
subtracting an element $v \in GF(p)^n$, we will obtain a triangle in
$\Gamma$ containing the zero vector 
as a vortex with the same edge weights.  Suppose $\triangle =
(u_1,u_2,u_3)$, where $(u_1,u_2)$ has 
edge weight $i$, $(u_2,u_3)$ has edge weight $j$, and $(u_3,u_1)$ has
edge weight $k$. 
Let $\triangle ' = (0,u_2-u_1,u_3-u_1)$. We compute the edge weights of $\triangle '$:
\begin{center}
edge weight of $(0,u_2-u_1)=f((u_2-u_1)-0)=f(u_2-u_1)=i$
\\
edge weight of $(u_2-u_1,u_3-u_1)=f((u_3-u_1)-(u_2-u_1)=f(u_3-u_2)=j$
\\
edge weight of $(u_3-u_1,0)=f(0-(u_3-u_1))=f(u_1-u_3)=k$
\end{center}
Thus the claim is proven.
\\
\\
Therefore,
\begin{center}
$\left(\dfrac{1}{|GF(p)^n|}\right)Tr(A_iA_jA_k) = \left(\dfrac{1}{p^n}\right)Tr(A_iA_jA_k)$
\end{center}
is the number of closed walks of length 3 having edge 
weights $i,j,k$ and containing the zero vector as a 
vertex, incident to the edge of weight $i$ and the edge of weight $k$.
\\
\\
There are $|D_k|$ edges incident to the zero vector, so 
\begin{center}
$\left(\dfrac{1}{p^n}\right) \left(\dfrac{1}{|D_k|}\right)Tr(A_iA_jA_k)$
\end{center}
is the number of walks of length 2 from the zero vector to 
any neighbor of it along an edge of weight $k$. This is 
equivalent to the definition of the number $p^k_{ij}$ in the 
Matrix-Walk Theorem.  
\end{proof}

The following corollary is well-known (see \cite{CvL}, page 202).

\begin{corollary}
{\rm
Let $G = GF(p)^n$. Let $D_0, \cdots, D_r \subseteq G$ such that $D_i \cap D_j = \emptyset$ if $i \neq j$, and 
\begin{itemize}
\item
$G$ is the disjoint union of $D_0 \cup \cdots \cup D_r$
\item
for each $i$ there is a $j$ such that $D_i^{-1} = D_j$, and
\item
$D_i \cdot D_j = \sum\limits_{k=0}^r p_{ij}^k D_k$ for some positive integer $p_{ij}^k$.
\end{itemize}
Then, for all $i,j,k$, $|D_k| p_{ij}^k = |D_i| p_{kj}^i$.
}
\end{corollary}

\begin{proof}
For all $i,j,k,$ we have the following identity of adjacency matrices:
\begin{center}
Tr($A_iA_jA_k$) = $p^n|D_k|p_{ij}^k$
\end{center}
where $p^n$ is the order of $G$ and $p_{ij}^k$ is an intersection number. Since Tr($AB$) = Tr($BA$) for all matrices $A$ and $B$, Tr($A_iA_jA_k$) = Tr($A_kA_jA_i$), and the proposition follows.
\end{proof}   

We can apply this concept to a weighted partial difference set and achieve similar results. If $G$ is a set and $D=D_{1} \cup D_{2} \cup \dots \cup D_r$ (all $D_{i}$ distinct) is a weighted partial difference set of $G$, then we can construct an association scheme as follows:
\\
\begin{itemize}
\item
Define $R_0=\Delta_G= \{ (x,x)\in G\times G\ | \ x\in G\}$.
\item
For $1 \leq i \leq r$, define $R_{i}=\{(x,y)\in G\times G\ | \ xy^{-1} \in D_{i}, x \not= y\}$
\item
Define $R_{r+1}=\{(x,y)\in G\times G\ | \ xy^{-1} \notin D, x \not= y\}$
\end{itemize}

\begin{proposition}
\label{prop:wtPDS2AS}
{\rm
If $R_{r+1}$ is non-empty then
the collection ($G,R_{0},R_{1},\dots,R_{r},R_{r+1}$) as 
defined above produces an association scheme of class 
$r+1$.
}
\end{proposition}

\begin{proof}
{\rm
Consider the subring $S$ of $\ccc[G]$ generated by $D_0, \cdots , D_{r+1}$, where $D_0 = \{1\}$ and $D_{r+1} = G \setminus (D \cup \{1\})$. First, we show that $S$ is a Schur ring.

We know that for $0 \leq i \leq r$, $D_i^{-1} = D_j$ for some $j$. $D_{r+1}^{-1} = D_{r+1}$ because $(G,D)$ is a partial difference set if and only if $(G,G \setminus D)$ is a partial difference set. 

We can then compute $D_i \cdot D_j$ in $\ccc[G]$; by the definition of a weighted partial difference set,
\begin{equation}
\label{eqn:schur-ring4}
D_i \cdot D_j = \alpha_{ij} \cdot 1 + \sum_{l=1}^r \lambda_{i,j,l} D_l + \mu_{i,j} D_{r+1},
\end{equation}
for some integer $\alpha_{ij}$. So the Schur ring decomposition formula
\begin{center}
$D_i \cdot D_j$ = $\sum\limits_0^{r+1} p_{ij}^k D_k$
\end{center}
holds for some integer $p_{ij}^k$.

Furthermore, 
$p_{ij}^0 = \delta_{ij} k_i$ for $ 0 \leq i, j \leq r+1$ 
and $p_{0j}^l = 
\delta_{jl}$ for $0 \leq j, l \leq r+1$.  

By expanding out expressions for $D_i \cdot G$ and 
$D_{r+1} \cdot G$, it can be shown that 

\[
p_{i,r+1}^l = k_i \delta_{il} - 
\sum_{j=1}^r \lambda_{ijl}
\]
for $1 \leq i, l \leq r$, 

\[
p_{i,r+1}^{r+1} = k_i - \sum_{j=1}^r \mu_{ij},
\]
for $1 \leq i \leq r$, and 

\[
p_{r+1,r+1}^{r+1} = k_{r+1}-1-\sum_{i=1}^r k_i 
+ \sum_{i=1}^r\sum_{j=1}^r\mu_{ij}.
\]
Also, $p_{ij}^l=p_{ji}^l$ for all $i$ and $j$, by symmetry. 

Next, we will show that for all $i,j,k \in \{0, \cdots, r+1\}$ and for $(x,y) \in R_k$,
\[
|\{z \in G | (x,z) \in R_i, (z,y) \in R_j\}|
\]
is a constant that depends only on $k$ (and $i,j$).

Choose $(x,z) \in R_i, (z,y) \in R_j$; then $xz^{-1} \in D_i, zy^{-1} \in D_j$. 
Consider $(xz^{-1})(zy^{-1}) = xy^{-1} \in D_i \cdot D_j$. This is independent 
of $z$. There are exactly $p_{ij}^k$ such elements $z$ by the Schur ring 
structure identity, since every element in $D_k$ (e.g. $xy^{-1}$) is repeated $p_{ij}^k$ times.
}
\end{proof}

\subsection{Fourier transforms and graph spectra}
\label{sec:FTs}

In the case $p=2$,
the spectrum of $\Gamma_f$ is determined by the set of values of the
Walsh-Hadamard transform of
$f$ when regarded as a vector of (integer) $0,1$-values (of length $2^n$).
Does this result have an analog for $p>2$?

\begin{definition} ({\bf Butson matrix})
{\rm
We call an $N\times N$ complex matrix $M$ a {\it Butson matrix} if

\[
M\cdot \overline{M}^t=NI_N, 
\]
where $I_N$ is the $N\times N$ identity matrix. 
}
\end{definition}
\index{Butson matrix}

\begin{lemma}
\label{lemma:balance}
Consider a map $g:GF(p)^n\to GF(p)$,
where we identify $GF(p)$ with $\{0,1,2,\dots, p-1\}$.
The following are equivalent.
\begin{itemize}
\item[(a)]
$g$ is balanced.
\item[(b)]
$|g^{-1}(x)|=p^{n-1}$, for each $x\in GF(p)$.
\item[(c)]
The Fourier transform of $\zeta^g$ satisfies
$\hat{\zeta^g}(0)=0$.
\end{itemize}
\end{lemma}

\begin{proof}
It is easy to show that (a) and (b) are equivalent.
Also, it is not hard to establish (a) implies (c).

We show (c) implies (a).
This is proven by an argument similar to that used for Lemma
\ref{lemma:sig}.

Note that

\[
\hat{\zeta^g}(0)=|{\rm supp}(g)_0|+|{\rm supp}(g)_1| \zeta+
\dots +|{\rm supp}(g)_{p-1}|\zeta^{p-1},
\]
which we can regard as an identity in the $(p-1)$-dimensional 
$\qqq$-vector space $\qqq(\zeta)$. 
If $\hat{\zeta^g}(0)$ is rational then relation

\[
1+\zeta+\zeta^2+\dots +\zeta ^{p-1}=0,
\]
implies all the $|{\rm supp}(g)_j|$ are equal, for $j\not= 0$.
It also implies $\hat{\zeta^g}(0)=|{\rm supp}(g)_0|-|{\rm supp}(g)_1|$.
Therefore, $\hat{\zeta^g}(0)=0$ implies $g$ is balanced.
\end{proof}

The following equivalences are known (see for example
\cite{T} and \cite{CD}), but proofs are included for the
convenience of the reader.

\begin{proposition}
\label{prop:bent}
Let $f:GF(p)^n\to GF(p)$ be any function.
The following are equivalent.
\begin{itemize}
\item[(a)]
$f$ is bent.
\item[(b)]
The matrix $\zeta^F=(\zeta^{f(\eta(i)-\eta(j))})_{0\leq i,j\leq p^n-1}$
is Butson, where $\eta$ is as in (\ref{eqn:pary-rep-map}).
\item[(c)]
The derivative 
\[
D_bf(x) = f(x+b)-f(x),
\]
is balanced, for each $b\not= 0$.

\end{itemize}
\end{proposition}

\begin{remark}
From Proposition \ref{prop:bent}, we know
$D_bf(x)$ is balanced, for each $b\not= 0$, if and only if $f$
is bent. Therefore, we know that

\[
|\{u\in V\ |\ f(u-u_1)=f(u-u_2)\}|=p^{n-1}
\]
(which is obviously independent of $u_1,u_2$).
On the other hand, if $f$ is bent, it is not true in general
that, for each $a\in GF(p)$, 

\[
|\{u\in V\ |\ f(u-u_1)=f(u-u_2)=a\}|,
\]
is independent of $u_1,u_2\in V$.
See Example \ref{example:bent1} for a counterexample.
\end{remark}

Let 

\[
h(b) = 
(\zeta^{D_b f})^\wedge(0) = 
\sum_{x \in V}
\zeta^{f(x+b)-f(x)}.
\]

\begin{proof}
{\it (a) $\implies$ (c)}:
Note that 

\[
\begin{array}{rl}
\hat{h}(y) = &
\sum_{b \in V}
\sum_{x \in V}
\zeta^{f(x+b)-f(x)}
\zeta^{- \langle y,b\rangle}\\
&= \sum_{b \in V} \sum_{x \in V}
\zeta^{f(x+b)-f(x)- \langle y,b\rangle - \langle y,x\rangle + \langle y,x\rangle}\\
&= 
\sum_{x \in V}
\zeta^{-f(x) + \langle y,x\rangle}
\sum_{b \in V} 
\zeta^{f(x+b)- \langle y,x+b\rangle}\\
&= \hat{\zeta^f}(y)\overline{\hat{\zeta^f}(y)}
=|\hat{\zeta^f}(y)|^2=|W_f(y)|^2.
\end{array}
\]
Therefore, if $f$ is bent then $\hat{h}$ is a constant,
which means that $h$ is supported at $0$.
By Lemma \ref{lemma:balance}, $D_bf(x)$ is balanced.

{\it (c) $\implies$ (a)}:
We reverse the above argument.
Suppose $D_bf(x)$ is balanced.
By Lemma \ref{lemma:balance}, $h$ is supported at $0$,
so  $\hat{h}$ is a constant. Plug in $y=0$ and using the fact
$D_bf(x)$ is balanced, we see that the constant must 
$|V|=p^n$, Thus $=|W_f(y)|=p^{n/2}$.

{\it (c) $\implies$ (b)}:
Note that

\[
\sum_{j=0}^{p^n-1}
\zeta^{f(\eta(i)-\eta(j)) - f(\eta(j)-\eta(k))}
=
\sum_{x \in V}
\zeta^{f(x) - f(x+\eta(i)-\eta(k))}
=
\sum_{x \in V}
\zeta^{f(x+b)-f(x)},
\]
where $b=\eta(i)-\eta(k)$. If $D_bf(x)$ is balanced
then by Lemma \ref{lemma:balance}, this sum is zero
for all $b\not= 0$. These are the off-diagonal terms in the
product $\zeta^F\overline{\zeta^F}^t$. Those terms
when $i=k$ are the diagonal terms. They are 
obviously $|V|=p^n$. This implies $\zeta^F$ is Butson.

{\it (b) $\implies$ (c)}:
This follows by reversing the above argument. The details are omitted.

\end{proof}

%% circulant matrix
Recall a circulant matrix is a square matrix where each row vector 
is a cyclic shift one element to the right relative to the preceding 
row vector.
Our Fourier transform matrix $F$ is not circulant, but is ``block circulant.''
Like circulant matrices, it has the property that
$\vec{v}_a =  (\zeta^{-\langle a,x\rangle}\ |\ x\in V)$ is an
eigenvector with eigenvalue $\lambda_a = \hat{f}(-a)$ (something related to a
value of the Hadamard transform of $f$). 
Thus, the proposition below shows that it ``morally'' behaves like
a circulant matrix in some ways.
\index{circulant matrix}

\begin{proposition}
The eigenvalues $\lambda_a = \hat{f}(-a)$ of this
matrix $F$ are values of the Fourier transform of the function $f_\ccc$,

\[
\hat{f}(y) = \sum_{x\in V} f_\ccc(x)\zeta^{-\langle x,y\rangle},
\]
and the eigenvectors are the vectors of $p$-th roots
of unity,

\[
\vec{v}_a =  (\zeta^{-\langle a,x\rangle}\ |\ x\in V).
\]
\end{proposition}

\begin{proof}
In $F = (F_{i,j})$, we have $F_{i,j} = f_\ccc(\eta(i)-\eta(j))$ for 
$i,j\in \{0,1,\dots, p^n-1\}$. For each $a\in GF(p)^n$, let

\[
\vec{v}_a = (\zeta^{-\langle a,\eta(i)\rangle}\ |\ i\in \{0,1,\dots, p^n-1\})
\]
Then

\[
F\vec{v}_a = (\sum_{y\in V} f_\ccc(x-y)\zeta^{-\langle a,y\rangle}\ |\ x \in
V).
\]
The entry in the $i$th coordinate, where $x=\eta(i)$ is given by 

\[
\begin{array}{rl}
\sum_{y\in V} f_\ccc(x-y)\zeta^{-\langle a,y\rangle}
&=\sum_{y\in V} f_\ccc(-y)\zeta^{-\langle a,y+x\rangle}\\
&\  \\
&=\zeta^{-\langle a,x\rangle}\sum_{y\in V} f_\ccc(-y)\zeta^{-\langle
  a,y\rangle}\\
&\  \\
&=\zeta^{-\langle a,x\rangle}\sum_{y\in V} f_\ccc(y)\zeta^{\langle
  a,y\rangle} \\
&\  \\
&= \zeta^{-\langle a,x\rangle}\hat{f}(-a).
\end{array}
\]
Therefore, the coordinates of the vector $F\vec{v}_a$ are
the same as those of $\vec{v}_a$, up to a scalar factor.
Thus $\lambda_a = \hat{f}(-a)$ is an eigenvalue and
$\vec{v}_a = (\zeta^{-\langle a,x\rangle}\ |\ x\in V)$ is an
eigenvector.

\end{proof}

\begin{corollary}
The matrix $F$ is invertible if and
only if none of the values of the Fourier transform of 
$f_\ccc$ vanish.
\end{corollary}

\begin{corollary}
\label{corollary:spectrum}
The spectrum of the graph $\Gamma_f$ is 
precisely the set of values of the Fourier transform of $f_\ccc$.
\end{corollary}

\section{Examples of Cayley graphs}

Let $V=GF(p)^n$ and let $f:V\to GF(p)$. If we fix an ordering 
on $GF(p)^n$, then the $p^n\times p^n$ matrix

\begin{equation}
\label{eqn:adj_mat}
F = (f_\ccc(x-y)\ |\ x,y\in V),
\end{equation}
is a $\zzz$-valued matrix. Here $x$ indexes the rows and
$y$ indexes the columns.

\begin{example}
{\rm
I can be shown that Example \ref{example:17} (or an isomorphic copy)
arises via the bent function $b_8$ (see also Example \ref{example:25}). 
For this example of $b_8$, we compute the adjacency matrix associated to the 
members $R_1$ and $R_2$ of the association scheme 
$(G,R_0,R_1,R_2,R_3)$, where $G = GF(3)^2$, 

\[
R_i = \{(g,h)\in  G\times G\ |\ gh^{-1} \in D_i\},\ \ \ \ \ i=1,2,
\]
and $D_i = f^{-1}(i)$.

Consider the following Sage computation:

\vskip .15in
{\footnotesize{
\begin{Verbatim}[fontsize=\scriptsize,fontfamily=courier,fontshape=tt,frame=single,label=\sage]

sage: attach "/home/wdj/sagefiles/hadamard_transform.sage"
sage: FF = GF(3)
sage: V = FF^2
sage: Vlist = V.list()
sage: flist = [0,2,2,0,0,1,0,1,0]
sage: f = lambda x: GF(3)(flist[Vlist.index(x)])
sage: F = matrix(ZZ, [[f(x-y) for x in V] for y in V])
sage: F  ## weighted adjacency matrix
[0 2 2 0 0 1 0 1 0]
[2 0 2 1 0 0 0 0 1]
[2 2 0 0 1 0 1 0 0]
[0 1 0 0 2 2 0 0 1]
[0 0 1 2 0 2 1 0 0]
[1 0 0 2 2 0 0 1 0]
[0 0 1 0 1 0 0 2 2]
[1 0 0 0 0 1 2 0 2]
[0 1 0 1 0 0 2 2 0]
sage: eval1 = lambda x: int((x==1))
sage: eval2 = lambda x: int((x==2))
sage: F1 = matrix(ZZ, [[eval1(f(x-y)) for x in V] for y in V])
sage: F1
[0 0 0 0 0 1 0 1 0]
[0 0 0 1 0 0 0 0 1]
[0 0 0 0 1 0 1 0 0]
[0 1 0 0 0 0 0 0 1]
[0 0 1 0 0 0 1 0 0]
[1 0 0 0 0 0 0 1 0]
[0 0 1 0 1 0 0 0 0]
[1 0 0 0 0 1 0 0 0]
[0 1 0 1 0 0 0 0 0]
% sage: F1.eigenmatrix_right()
% (
% [ 2  0  0  0  0  0  0  0  0]  [ 1  0  0  1  0  0  0  0  0]
% [ 0  2  0  0  0  0  0  0  0]  [ 0  1  0  0  1  0  0  0  0]
% [ 0  0  2  0  0  0  0  0  0]  [ 0  0  1  0  0  1  0  0  0]
% [ 0  0  0 -1  0  0  0  0  0]  [ 0  1  0  0  0  0  1  0  0]
% [ 0  0  0  0 -1  0  0  0  0]  [ 0  0  1  0  0  0  0  1  0]
% [ 0  0  0  0  0 -1  0  0  0]  [ 1  0  0  0  0  0  0  0  1]
% [ 0  0  0  0  0  0 -1  0  0]  [ 0  0  1  0  0 -1  0 -1  0]
% [ 0  0  0  0  0  0  0 -1  0]  [ 1  0  0 -1  0  0  0  0 -1]
% [ 0  0  0  0  0  0  0  0 -1], [ 0  1  0  0 -1  0 -1  0  0]
sage: F2 = matrix(ZZ, [[eval2(f(x-y)) for x in V] for y in V])
sage: F2
[0 1 1 0 0 0 0 0 0]
[1 0 1 0 0 0 0 0 0]
[1 1 0 0 0 0 0 0 0]
[0 0 0 0 1 1 0 0 0]
[0 0 0 1 0 1 0 0 0]
[0 0 0 1 1 0 0 0 0]
[0 0 0 0 0 0 0 1 1]
[0 0 0 0 0 0 1 0 1]
[0 0 0 0 0 0 1 1 0]
sage: F1*F2-F2*F1 == 0
True
sage: delta = lambda x: int((x[0]==x[1]))
sage: F3 = matrix(ZZ, [[(eval0(f(x-y))+delta([x,y]))%2 for x in V] for y in V])
sage: F3
[0 0 0 1 1 0 1 0 1]
[0 0 0 0 1 1 1 1 0]
[0 0 0 1 0 1 0 1 1]
[1 0 1 0 0 0 1 1 0]
[1 1 0 0 0 0 0 1 1]
[0 1 1 0 0 0 1 0 1]
[1 1 0 1 0 1 0 0 0]
[0 1 1 1 1 0 0 0 0]
[1 0 1 0 1 1 0 0 0]
sage: F3*F2-F2*F3==0
True
sage: F3*F1-F1*F3==0
True
sage: F0 = matrix(ZZ, [[delta([x,y]) for x in V] for y in V])
sage: F0
[1 0 0 0 0 0 0 0 0]
[0 1 0 0 0 0 0 0 0]
[0 0 1 0 0 0 0 0 0]
[0 0 0 1 0 0 0 0 0]
[0 0 0 0 1 0 0 0 0]
[0 0 0 0 0 1 0 0 0]
[0 0 0 0 0 0 1 0 0]
[0 0 0 0 0 0 0 1 0]
[0 0 0 0 0 0 0 0 1]
sage: F1*F3 == 2*F2 + F3
True

\end{Verbatim}
}}

The Sage computation above tells us that the adjacency matrix of $R_1$ is 

\[
A_1 = 
\left(\begin{array}{rrrrrrrrr}
0 & 0 & 0 & 0 & 0 & 1 & 0 & 1 & 0 \\
0 & 0 & 0 & 1 & 0 & 0 & 0 & 0 & 1 \\
0 & 0 & 0 & 0 & 1 & 0 & 1 & 0 & 0 \\
0 & 1 & 0 & 0 & 0 & 0 & 0 & 0 & 1 \\
0 & 0 & 1 & 0 & 0 & 0 & 1 & 0 & 0 \\
1 & 0 & 0 & 0 & 0 & 0 & 0 & 1 & 0 \\
0 & 0 & 1 & 0 & 1 & 0 & 0 & 0 & 0 \\
1 & 0 & 0 & 0 & 0 & 1 & 0 & 0 & 0 \\
0 & 1 & 0 & 1 & 0 & 0 & 0 & 0 & 0
\end{array}\right),
\]
the adjacency matrix of $R_2$ is 

\[
A_2 = 
\left(\begin{array}{rrrrrrrrr}
0 & 1 & 1 & 0 & 0 & 0 & 0 & 0 & 0 \\
1 & 0 & 1 & 0 & 0 & 0 & 0 & 0 & 0 \\
1 & 1 & 0 & 0 & 0 & 0 & 0 & 0 & 0 \\
0 & 0 & 0 & 0 & 1 & 1 & 0 & 0 & 0 \\
0 & 0 & 0 & 1 & 0 & 1 & 0 & 0 & 0 \\
0 & 0 & 0 & 1 & 1 & 0 & 0 & 0 & 0 \\
0 & 0 & 0 & 0 & 0 & 0 & 0 & 1 & 1 \\
0 & 0 & 0 & 0 & 0 & 0 & 1 & 0 & 1 \\
0 & 0 & 0 & 0 & 0 & 0 & 1 & 1 & 0
\end{array}\right),
\]
and the adjacency matrix of $R_3$ is 

\[
A_3 = 
\left(\begin{array}{rrrrrrrrr}
0 & 0 & 0 & 1 & 1 & 0 & 1 & 0 & 1 \\
0 & 0 & 0 & 0 & 1 & 1 & 1 & 1 & 0 \\
0 & 0 & 0 & 1 & 0 & 1 & 0 & 1 & 1 \\
1 & 0 & 1 & 0 & 0 & 0 & 1 & 1 & 0 \\
1 & 1 & 0 & 0 & 0 & 0 & 0 & 1 & 1 \\
0 & 1 & 1 & 0 & 0 & 0 & 1 & 0 & 1 \\
1 & 1 & 0 & 1 & 0 & 1 & 0 & 0 & 0 \\
0 & 1 & 1 & 1 & 1 & 0 & 0 & 0 & 0 \\
1 & 0 & 1 & 0 & 1 & 1 & 0 & 0 & 0
\end{array}\right)
\]
Of course, the adjacency matrix of $R_0$ is 
the identity matrix. In the above computation,
Sage has also verified that they commute and satisfy

\[
A_1A_3 = 2A_2+A_3
\]
in the Schur ring.
}
\end{example}

\begin{example}
\label{example:2vars-2}
{\rm 
We take $V=GF(3)^2$ and consider an even function.

\vskip .15in
{\footnotesize{
\begin{Verbatim}[fontsize=\scriptsize,fontfamily=courier,fontshape=tt,frame=single,label=\sage]

sage: flist = [0,1,1,2,0,1,2,1,0]
sage: f = lambda x: GF(3)(flist[Vlist.index(x)])
sage: x = V.random_element()
sage: f(x) == f(-x)
True
sage: Gamma = boolean_cayley_graph(f, V)
sage: A = Gamma.adjacency_matrix(); A
[0 1 1 2 0 1 2 1 0]
[1 0 1 1 2 0 0 2 1]
[1 1 0 0 1 2 1 0 2]
[2 1 0 0 1 1 2 0 1]
[0 2 1 1 0 1 1 2 0]
[1 0 2 1 1 0 0 1 2]
[2 0 1 2 1 0 0 1 1]
[1 2 0 0 2 1 1 0 1]
[0 1 2 1 0 2 1 1 0]
sage: Gamma.connected_components_number()
1
\end{Verbatim}
}}

\vskip .2in
\noindent
The plot returned by 
\newline
{\small{
\verb+Graph(A).show(layout="circular", edge_labels=True, graph_border=True,dpi=150)+
}}
is shown in %Figure \ref{fig:example2} (see also 
Figure \ref{fig:example2-tikz}.

% \begin{figure}[t!]
% \begin{minipage}{\textwidth}
% \begin{center}
% \includegraphics[height=9cm,width=12cm]{images/hadamard_transform_graph-example2b}
% %\drawpng{height=9cm,width=12cm}{images/hadamard_transform_graph-example2b}
% \end{center}
% \end{minipage}
% \caption{The undirected unweighted 
% Cayley graph of an even $GF(3)$-valued function of two
%   variables from Example \ref{example:2vars-2}. (The vertices are
%   ordered as in the Example.)}
% \label{fig:example2}
% \end{figure}

\begin{figure}[t!]
\begin{minipage}{\textwidth}
\begin{center}
%This latex requires the tikz package
\definecolor{lorange}{rgb}{1.0,0.5,0}
\begin{center}
  \begin{tikzpicture}
    [font=\scriptsize,
    node/.style={shape=circle,draw=black,minimum width=0.5cm,thick,fill=lorange},
    edge1/.style={thick,red},
    edge2/.style={very thick,blue},
    edge/.style={thick}]

    \node (0) [node] at (0,4) {0};
    \node (1) [node] at (-3.5,3) {1};
    \node (2) [node] at (-6,1) {2};
    \node (3) [node] at (-5.5,-1) {3};
    \node (4) [node] at (-2.333,-3) {4};
    \node (5) [node] at (2.666,-3) {5};
    \node (6) [node] at (5.5,-1) {6};
    \node (7) [node] at (6,1) {7};
    \node (8) [node] at (3.5,3) {8};

    \draw [edge1] (0) to node[black]{$1$} (1);
    \draw [edge1] (0)  to node[black]{$1$} (2);
    \draw [edge2] (0)  to node[black]{$2$} (3);
    %\draw [edge2] (0)  to node[black]{$2$} (4);
    \draw [edge1] (0)  to node[above left,black,pos=.7]{$1$} (5);
    \draw [edge2] (0)  to node[black]{$2$} (6);
    \draw [edge1] (0)  to node[above left,black,pos=.32]{$1$} (7);
    %\draw [edge2] (0) -- node[above left,black,pos=.6]{$2$} (8);

    \draw [edge1] (1)  to node[black]{$1$} (2);
    \draw [edge1] (1)  to node[above left,black,pos=.42]{$1$} (3);
    \draw [edge2] (1)  to node[black]{$2$} (4);
    %\draw [edge2] (1)  to node[above left,black,pos=.6]{$2$} (5);
    %\draw [edge2] (1)  to node[above left,black,pos=.65]{$2$} (6);
    \draw [edge2] (1)  to node[black]{$2$} (7);
    \draw [edge1] (1)  to node[black]{$1$} (8);

    %\draw [edge2] (2)  to node[black]{$2$} (3);
    \draw [edge1] (2)  to node[black]{$1$} (4);
    \draw [edge2] (2)  to node[above left,black,pos=.65]{$2$} (5);
    \draw [edge1] (2)  to node[above left,black,pos=.38]{$1$} (6);
    %\draw [edge2] (2)  to node[black]{$2$} (7);
    \draw [edge2] (2)  to node[black]{$2$} (8);

    \draw [edge1] (3)  to node[black]{$1$} (4);
    \draw [edge1] (3)  to node[above left,black,pos=.52]{$1$} (5);
    \draw [edge2] (3)  to node[black]{$2$} (6);
    %\draw [edge2] (3)  to node[above left,black,pos=.35]{$2$} (7);
    \draw [edge1] (3)  to node[above left,black,pos=.65]{$1$} (8);

    \draw [edge1] (4)  to node[black]{$1$} (5);
    \draw [edge1] (4)  to node[above left,black,pos=.43]{$1$} (6);
    \draw [edge2] (4)  to node[above left,black,pos=.6]{$2$} (7);
    %\draw [edge2] (4)  to node[above left,black,pos=.7]{$2$} (8);

    %\draw [edge2] (5)  to node[black]{$2$} (6);
    \draw [edge1] (5)  to node[black]{$1$} (7);
    \draw [edge2] (5) to node[black]{$2$} (8);

    \draw [edge1] (6) to node[black]{$1$} (7);
    \draw [edge1] (6) to node[above left,black,pos=.6]{$1$} (8);

    \draw [edge1] (7) to node[black]{$1$} (8);

  \end{tikzpicture}
\end{center}
%\drawpng{height=9cm,width=9cm}{images/hadamard_transform_GF3-bent1-FT2}
\end{center}
\end{minipage}
\caption{The undirected unweighted 
Cayley graph of an even $GF(3)$-valued function of two
  variables from Example \ref{example:2vars-2}. (The vertices are
  ordered as in the Example.)}
\label{fig:example2-tikz}
\end{figure}
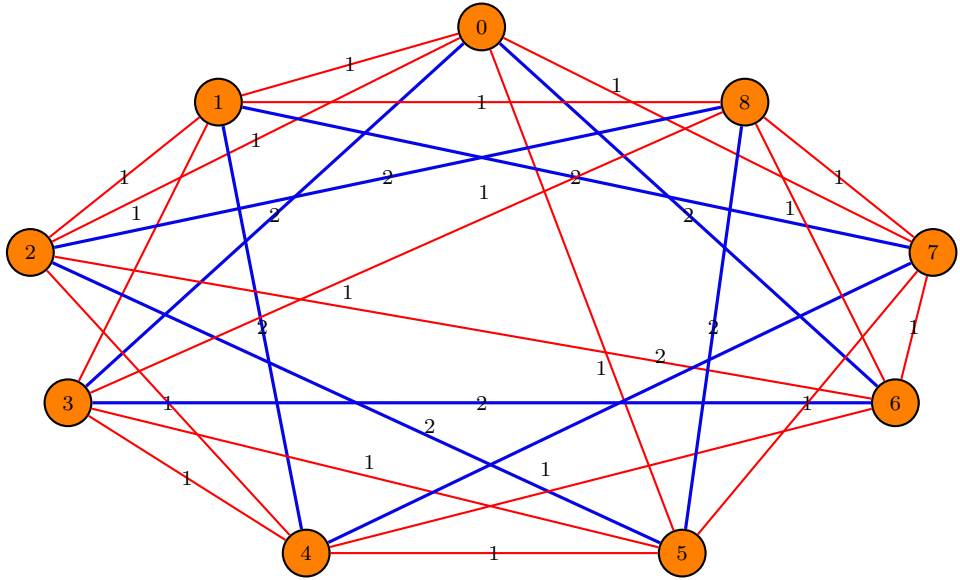

This example shall be continued below.

}
\end{example}

\begin{example}
\label{example:2vars-3}
{\rm
We take $V=GF(3)^2$ and consider an even function
whose Cayley graph has three connected components.

\vskip .15in
{\footnotesize{
\begin{Verbatim}[fontsize=\scriptsize,fontfamily=courier,fontshape=tt,frame=single,label=\sage]

sage: flist = [0,0,0,1,0,0,1,0,0]
sage: f = lambda x: GF(3)(flist[Vlist.index(x)])
sage: x = V.random_element()
sage: f(x) == f(-x)
True
sage: Gamma = boolean_cayley_graph(f, V)
sage: A = Gamma.adjacency_matrix(); A
[0 0 0 1 0 0 1 0 0]
[0 0 0 0 1 0 0 1 0]
[0 0 0 0 0 1 0 0 1]
[1 0 0 0 0 0 1 0 0]
[0 1 0 0 0 0 0 1 0]
[0 0 1 0 0 0 0 0 1]
[1 0 0 1 0 0 0 0 0]
[0 1 0 0 1 0 0 0 0]
[0 0 1 0 0 1 0 0 0]
sage: Gamma.connected_components_number()
3

\end{Verbatim}
}}

\vskip .2in
\noindent
The plot returned by {\tt Graph(A).show()} is shown in Figure \ref{fig:example3-tikz}.

\begin{figure}[t!]
\begin{minipage}{\textwidth}
\begin{center}
%This latex requires the tikz package
\definecolor{lorange}{rgb}{1.0,0.5,0}
\begin{center}
  \begin{tikzpicture}
    [font=\scriptsize,
    node/.style={shape=circle,draw=black,minimum width=0.5cm,thick,fill=lorange},
    edge1/.style={thick,red},
    edge2/.style={very thick,blue},
    edge/.style={thick}]

    \node (0) [node] at (0,0) {0};
    \node (6) [node] at (-2,0) {6};
    \node (3) [node] at (-1,2) {3};
    \node (1) [node] at (2,2) {1};
    \node (4) [node] at (1,0) {4};
    \node (7) [node] at (3,0) {7};
    \node (2) [node] at (4,0) {2};
    \node (5) [node] at (6,0) {5};
    \node (8) [node] at (5,2) {8};

    \draw [edge1] (0) to node[black]{$1$} (6);
    \draw [edge1] (0) to node[black]{$1$} (3);

    \draw [edge1] (1)  to node[black]{$1$} (4);
    \draw [edge1] (1)  to node[black]{$1$} (7);

    \draw [edge1] (2)  to node[black]{$1$} (5);
    \draw [edge1] (2)  to node[black]{$1$} (8);

    \draw [edge1] (3)  to node[black]{$1$} (6);

    \draw [edge1] (4)  to node[black]{$1$} (7);

    \draw [edge1] (5) to node[black]{$1$} (8);

  \end{tikzpicture}
\end{center}
%\drawpng{height=9cm,width=9cm}{images/hadamard_transform_GF3-bent1-FT2}
\end{center}
\end{minipage}
\caption{The undirected Cayley graph of an even $GF(3)$-valued function of two
  variables from Example \ref{example:2vars-3}. (The vertices are
  ordered as in the Example.)}
\label{fig:example3-tikz}
\end{figure}
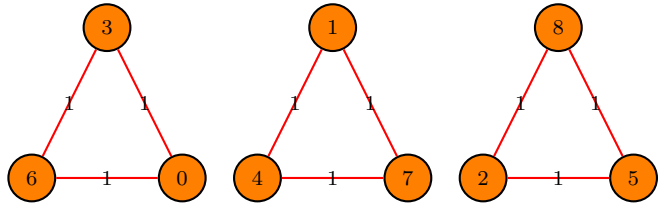

}
\end{example}

\begin{example}
{\rm
We return to the ternary function from Example \ref{example:2vars-2}.

\vskip .15in
{\footnotesize{
\begin{Verbatim}[fontsize=\scriptsize,fontfamily=courier,fontshape=tt,frame=single,label=\sage]

sage: V = GF(3)^2
sage: Vlist = V.list()   
sage: Vlist              
[(0, 0), (1, 0), (2, 0), (0, 1), (1, 1), (2, 1), (0, 2), (1, 2), (2, 2)]
sage: flist = [0,1,1,2,0,1,2,1,0]
sage: f = lambda x: GF(3)(flist[Vlist.index(x)])
sage: Gamma = boolean_cayley_graph(f, V)
sage: Gamma.adjacency_matrix()
[0 1 1 2 0 1 2 1 0]
[1 0 1 1 2 0 0 2 1]
[1 1 0 0 1 2 1 0 2]
[2 1 0 0 1 1 2 0 1]
[0 2 1 1 0 1 1 2 0]
[1 0 2 1 1 0 0 1 2]
[2 0 1 2 1 0 0 1 1]
[1 2 0 0 2 1 1 0 1]
[0 1 2 1 0 2 1 1 0]
sage: Gamma.spectrum()
[8, 2, 2, -1, -1, -1, -1, -4, -4]
sage: [CC(fourier_transform(f, a)) for a in V] 
[8.00000000000000, 2.00000000000000 - 6.66133814775094e-16*I, 
2.00000000000000 - 6.66133814775094e-16*I, 
-1.00000000000000 - 8.88178419700125e-16*I, 
-1.00000000000000 - 9.99200722162641e-16*I, 
-4.00000000000000 - 1.33226762955019e-15*I, 
-1.00000000000000 - 8.88178419700125e-16*I, 
-4.00000000000000 - 1.22124532708767e-15*I, 
-1.00000000000000 - 9.99200722162641e-16*I]

\end{Verbatim}
}}

\vskip .2in
\noindent
This shows that, in this case, the spectrum of the Cayley graph of $f$ 
agrees with the values of the
Fourier transform of $f_\ccc$.
}
\end{example}

Suppose we want to write the function $\zeta^{f(x)}$ as a linear
combination of translates of the function $f$:

\begin{equation}
\zeta^{f(x)}=\sum_{a\in V} c_a f(x-a),
\label{eqn:*}
\end{equation}
for some $c_a\in \ccc$. This may be regarded as the convolution 
of $f_\ccc$ with a function, $c$. One way to solve for the $c_a$'s is
to write this as a matrix equation,

\[
\zeta^{\vec{f}} = F\cdot \vec{c},
\]
where $\vec{c}=\vec{c}_f= (c_a\ |\ a\in V)$ and
$\zeta^{\vec{f}} = (\zeta^{f(x)}\ |\ x\in V)$. 
If $F$ is invertible, that is if the Fourier transform
of $f$ is always non-zero, then 

\[
\vec{c} = F^{-1}\zeta^{\vec{f}}.
\]

If (\ref{eqn:*}) holds then we can write the 
Walsh transform $f$,

\[
W_f(u) =
\sum_{x \in GF(p)^n}
\zeta^{f(x)- \langle u,x\rangle},
\]
as a linear combination of values of the Fourier
transform,

\[
\hat{f}(y) = \sum_{x\in V} f(x)\zeta^{-\langle x,y\rangle}.
\]
In other words, 

\begin{equation}
\label{eqn:WTvFT}
\begin{array}{rl}
W_f(u) 
&= \sum_{a\in V} c_a 
\sum_{x \in GF(p)^n}
\zeta^{-\langle u,x\rangle}
f(x-a)\\
&\  \\
&= \sum_{a\in V} c_a 
\sum_{x \in GF(p)^n}
\zeta^{-\langle u,x+a\rangle}
f(x)\\
&\  \\
&= \sum_{a\in V} c_a \zeta^{-\langle u,a\rangle}
\sum_{x \in GF(p)^n}
\zeta^{-\langle u,x\rangle}
f(x)\\
&\  \\
&= \hat{f}(u)
\sum_{a\in V} c_a \zeta^{-\langle u,a\rangle}.
\end{array}
\end{equation}
This may be regarded as the product of Fourier transforms
(that of the function $f_\ccc$ and that of the function $c$,
which depends on $f$).
In other words, there is a relationship between the Fourier transform
of a $GF(p)$-valued function and its Walsh-Hadamard transform. 
However, it is not explicit unless one knows the function 
$c$ (which depends on $f$ in a complicated way).

\subsection{$GF(3)^2\to GF(3)$}

Using Sage, we verified the following fact (originally discovered by 
the last-named author, Walsh).

\begin{proposition}
There are $18$ even bent functions
$f:GF(3)^2\to GF(3)$ such that $f(0)=0$. The group $G=GL(2,GF(3))$
acts on the set ${\mathbb{B}}$ of all such bent functions and 
there are two orbits in ${\mathbb{B}}/G$:

\[
{\mathbb{B}}/G = B_1\cup B_2,
\]
where $|B_1|=12$ and  $|B_2|=6$. 

The $18$ bent functions $b_1, b_2, \dots , b_{18}$ are given here in table form and algebraic normal form.  The orbit $B_1$ consists of the functions $b_2$, $b_3$, $b_4$, $b_5$, $b_6$, $b_7$, $b_8$, $b_{11}$, $b_{14}$, $b_{15}$, and $b_{16}$.  These functions are all regular.  The orbit $B_2$ consists of the functions $b_1$, $b_{10}$, $b_{12}$, $b_{13}$, $b_{17}$, and $b_{18}$. These functions are weakly regular (but not regular).

Each of the bent functions give rise to a weighted PDS.

\end{proposition}

\begin{example}
{\rm

Consider the even function $f:GF(3)^2\to GF(3)$  with the 
following values:

\vskip .15in
{\footnotesize{
\begin{tabular}{c|ccccccccc}
$GF(3)^2$ & (0, 0) & (1, 0) & (2, 0) & (0, 1) & (1, 1) & (2, 1) & (0,
2) & (1, 2) & (2, 2) \\ \hline
$f$ & 0  & 1  & 1  & 2  & 0  & 1  & 2  & 1  & 0 \\
\end{tabular}
}}
\vskip .15in
The Cayley graph $\Gamma$ of $f$
is given in %Figure \ref{fig:GF3-fcn} (and also 
Figure \ref{fig:GF3-fcn-tikz}.

\begin{figure}[t!]
\begin{minipage}{\textwidth}
\begin{center}
%This latex requires the tikz package
\definecolor{lorange}{rgb}{1.0,0.5,0}
\begin{center}
  \begin{tikzpicture}
    [font=\scriptsize,
    node/.style={shape=circle,draw=black,minimum width=0.5cm,thick,fill=lorange},
    edge1/.style={thick,red},
    edge2/.style={very thick,blue},
    edge/.style={thick}]

    \node (0) [node] at (0,4) {0};
    \node (1) [node] at (-3.5,3) {1};
    \node (2) [node] at (-6,1) {2};
    \node (3) [node] at (-5.5,-1) {3};
    \node (4) [node] at (-2.333,-3) {4};
    \node (5) [node] at (2.666,-3) {5};
    \node (6) [node] at (5.5,-1) {6};
    \node (7) [node] at (6,1) {7};
    \node (8) [node] at (3.5,3) {8};

    \draw [edge1] (0) to node[black]{$1$} (1);
    \draw [edge1] (0)  to node[black]{$1$} (2);
    \draw [edge2] (0)  to node[black]{$2$} (3);
    %\draw [edge2] (0)  to node[black]{$2$} (4);
    \draw [edge1] (0)  to node[above left,black,pos=.7]{$1$} (5);
    \draw [edge2] (0)  to node[black]{$2$} (6);
    \draw [edge1] (0)  to node[above left,black,pos=.32]{$1$} (7);
   % \draw [edge2] (0) -- node[above left,black,pos=.6]{$2$} (8);

    \draw [edge1] (1)  to node[black]{$1$} (2);
    \draw [edge1] (1)  to node[above left,black,pos=.42]{$1$} (3);
    \draw [edge2] (1)  to node[black]{$2$} (4);
    %\draw [edge2] (1)  to node[above left,black,pos=.6]{$2$} (5);
    %\draw [edge2] (1)  to node[above left,black,pos=.65]{$2$} (6);
    \draw [edge2] (1)  to node[black]{$2$} (7);
    \draw [edge1] (1)  to node[black]{$1$} (8);

    %\draw [edge2] (2)  to node[black]{$2$} (3);
    \draw [edge1] (2)  to node[black]{$1$} (4);
    \draw [edge2] (2)  to node[above left,black,pos=.65]{$2$} (5);
    \draw [edge1] (2)  to node[above left,black,pos=.38]{$1$} (6);
    %\draw [edge2] (2)  to node[black]{$2$} (7);
    \draw [edge2] (2)  to node[black]{$2$} (8);

    \draw [edge1] (3)  to node[black]{$1$} (4);
    \draw [edge1] (3)  to node[above left,black,pos=.52]{$1$} (5);
    \draw [edge2] (3)  to node[black]{$2$} (6);
    %\draw [edge2] (3)  to node[above left,black,pos=.35]{$2$} (7);
    \draw [edge1] (3)  to node[above left,black,pos=.65]{$1$} (8);

    \draw [edge1] (4)  to node[black]{$1$} (5);
    \draw [edge1] (4)  to node[above left,black,pos=.43]{$1$} (6);
    \draw [edge2] (4)  to node[above left,black,pos=.6]{$2$} (7);
    %\draw [edge2] (4)  to node[above left,black,pos=.7]{$2$} (8);

    %\draw [edge2] (5)  to node[black]{$2$} (6);
   \draw [edge1] (5) to node[black]{$1$} (7);
   \draw [edge2] (5) to node[black]{$2$} (8);

    \draw [edge1] (6) to node[black]{$1$} (7);
    \draw [edge1] (6) to node[above left,black,pos=.6]{$1$} (8);

    \draw [edge1] (7) to node[black]{$1$} (8);

  \end{tikzpicture}
\end{center}

\end{center}
\end{minipage}
\caption{The weighted Cayley graph of a non-bent even $GF(3)$-valued function.}
\label{fig:GF3-fcn-tikz}
\end{figure}
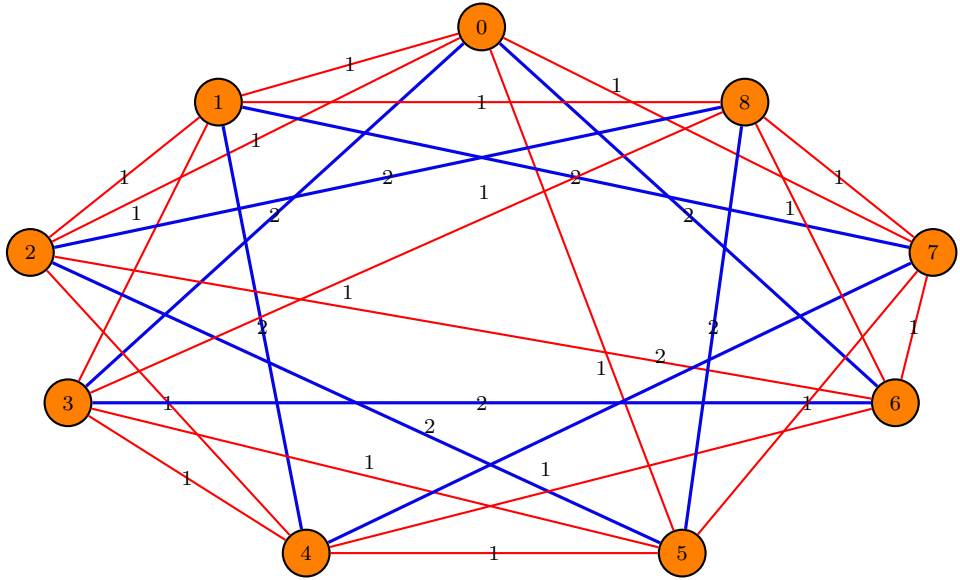

%\begin{figure}[t!]
%\begin{minipage}{\textwidth}
%\begin{center}
%\drawpng{height=9cm,width=12cm}{images/hadamard_transform_GF3-fcn1}
%\includegraphics[height=9cm,width=12cm]{images/hadamard_transform_GF3-fcn1}
%\end{center}
%\end{minipage}
%\caption{The weighted Cayley graph of a non-bent even $GF(3)$-valued function.}
%\label{fig:GF3-fcn}
%\end{figure}

The values of the Hadamard transform of $f$ are listed below
(showing that $f$ is not bent).
\vskip .15in

\begin{Verbatim}[fontsize=\scriptsize,fontfamily=courier,fontshape=tt,frame=single,label=\sage]

sage: V = GF(3)^2
sage: flist = [0,1,1,2,0,1,2,1,0]
sage: Vlist = V.list()
sage: f = lambda x: GF(3)(flist[Vlist.index(x)])
sage: [hadamard_transform(f,a) for a in V]      
[e^(2/3*I*pi) + e^(4/3*I*pi) + e^(2/3*I*pi) + e^(4/3*I*pi) + 2*e^(2/3*I*pi) + 3,
 e^(2/3*I*pi) + 5*e^(4/3*I*pi) + 3,
 3*e^(4/3*I*pi) + e^(2/3*I*pi) + 2*e^(4/3*I*pi) + 3,
 e^(2/3*I*pi) + 2*e^(4/3*I*pi) + 3*e^(2/3*I*pi) + 3,
 e^(4/3*I*pi) + 4*e^(2/3*I*pi) + e^(4/3*I*pi) + 3,
 e^(4/3*I*pi) + e^(2/3*I*pi) + e^(4/3*I*pi) + 6,
 e^(4/3*I*pi) + e^(2/3*I*pi) + e^(4/3*I*pi) + 3*e^(2/3*I*pi) + 3,
 e^(4/3*I*pi) + e^(2/3*I*pi) + e^(4/3*I*pi) + 6,
 4*e^(2/3*I*pi) + 2*e^(4/3*I*pi) + 3]
sage: [CC(hadamard_transform(f,a)) for a in V]  
[-2.22044604925031e-16 + 1.73205080756888*I,
 -2.22044604925031e-15 - 3.46410161513775*I,
 -1.99840144432528e-15 - 3.46410161513775*I,
 -2.22044604925031e-16 + 1.73205080756888*I,
 1.73205080756888*I,
 4.50000000000000 - 0.866025403784438*I,
 1.73205080756888*I,
 4.50000000000000 - 0.866025403784438*I,
 1.73205080756888*I]

\end{Verbatim}
\vskip .1in

This $f$ is not bent and has algebraic normal form
\[
2x_0^2x_1^2 + x_0^2 + x_0x_1 +2 x_1^2.
\]
In particular, it is non-homogeneous.
The weighted adjacency matrix of its Cayley graph is  

\[
\left(\begin{array}{rrrrrrrrr}
0 & 1 & 1 & 2 & 0 & 1 & 2 & 1 & 0 \\
1 & 0 & 1 & 1 & 2 & 0 & 0 & 2 & 1 \\
1 & 1 & 0 & 0 & 1 & 2 & 1 & 0 & 2 \\
2 & 1 & 0 & 0 & 1 & 1 & 2 & 0 & 1 \\
0 & 2 & 1 & 1 & 0 & 1 & 1 & 2 & 0 \\
1 & 0 & 2 & 1 & 1 & 0 & 0 & 1 & 2 \\
2 & 0 & 1 & 2 & 1 & 0 & 0 & 1 & 1 \\
1 & 2 & 0 & 0 & 2 & 1 & 1 & 0 & 1 \\
0 & 1 & 2 & 1 & 0 & 2 & 1 & 1 & 0
\end{array}\right).
\]
The matrix $N_{22}$ whose $u,v$-entry is
$|N(u,2)\cap N(v,2)|$, $u,v$ vertices of $\Gamma$, is:

\[
\left(\begin{array}{rrrrrrrrr}
2 & 0 & 0 & 1 & 0 & 0 & 1 & 0 & 0 \\
0 & 2 & 0 & 0 & 1 & 0 & 0 & 1 & 0 \\
0 & 0 & 2 & 0 & 0 & 1 & 0 & 0 & 1 \\
1 & 0 & 0 & 2 & 0 & 0 & 1 & 0 & 0 \\
0 & 1 & 0 & 0 & 2 & 0 & 0 & 1 & 0 \\
0 & 0 & 1 & 0 & 0 & 2 & 0 & 0 & 1 \\
1 & 0 & 0 & 1 & 0 & 0 & 2 & 0 & 0 \\
0 & 1 & 0 & 0 & 1 & 0 & 0 & 2 & 0 \\
0 & 0 & 1 & 0 & 0 & 1 & 0 & 0 & 2
\end{array}\right).
\]
\begin{itemize}
\item
Are all the values of $N_{22}[u,v]$ the same if $u,v$ are distinct
vertices which are not neighbors? Yes:  $N_{22}[u,v]=0$ for all such
$u,v$. 
Therefore, $\mu_{(2,2)}=0$ in the notation of (\ref{eqn:wtd-srg}).

\item
Are all the values of $N_{22}[v,v]$ the same? Yes:  
$N_{22}[v,v]=2$ for all $v$.
Therefore, $k_{(2,2)}=2$ in the notation of (\ref{eqn:wtd-srg}).
\item
Are all the values of $N_{22}[u,v]$ the same if $u,v$ are neighbors
with edge-weight $2$? Yes:  $N_{22}[u,v]=1$ for all such
$u,v$. 
Therefore, $\lambda_{(2,2,2)}=1$ in the notation of (\ref{eqn:wtd-srg}).
\item
Are all the values of $N_{22}[u,v]$ the same if $u,v$ are neighbors
with edge-weight $1$? Yes:  $N_{22}[u,v]=0$ for all such
$u,v$. 
Therefore, $\lambda_{(2,2,1)}=0$ in the notation of (\ref{eqn:wtd-srg}).
\end{itemize}

The matrix $N_{12}$, whose $u,v$-entry is
$|N(u,2)\cap N(v,1)|=|N(u,1)\cap N(v,2)|$, is:

\[
\left(\begin{array}{rrrrrrrrr}
0 & 1 & 1 & 0 & 2 & 1 & 0 & 1 & 2 \\
1 & 0 & 1 & 1 & 0 & 2 & 2 & 0 & 1 \\
1 & 1 & 0 & 2 & 1 & 0 & 1 & 2 & 0 \\
0 & 1 & 2 & 0 & 1 & 1 & 0 & 2 & 1 \\
2 & 0 & 1 & 1 & 0 & 1 & 1 & 0 & 2 \\
1 & 2 & 0 & 1 & 1 & 0 & 2 & 1 & 0 \\
0 & 2 & 1 & 0 & 1 & 2 & 0 & 1 & 1 \\
1 & 0 & 2 & 2 & 0 & 1 & 1 & 0 & 1 \\
2 & 1 & 0 & 1 & 2 & 0 & 1 & 1 & 0
\end{array}\right).
\]
Define $N_{21}$ similarly. Sage verifies that $N_{12}=N_{21}$.

\begin{itemize}
\item
Are all the values of $N_{21}[u,v]$ the same if $u,v$ are distinct
vertices which are not neighbors? Yes:  $N_{21}[u,v]=2$ for all such
$u,v$. 
Therefore, $\mu_{(2,1)}=2$ in the notation of (\ref{eqn:wtd-srg}).
\item
Are all the values of $N_{21}[v,v]$ the same? Yes:  
$N_{21}[v,v]=0$ for all $v$.
Therefore, $k_{(2,1)}=0$ in the notation of (\ref{eqn:wtd-srg}).
\item
Are all the values of $N_{21}[u,v]$ the same if $u,v$ are neighbors
with edge-weight $2$? Yes:  $N_{21}[u,v]=0$ for all such
$u,v$. 
Therefore, $\lambda_{(2,1,2)}=0$ in the notation of (\ref{eqn:wtd-srg}).
\item
Are all the values of $N_{21}[u,v]$ the same if $u,v$ are neighbors
with edge-weight $1$? Yes:  $N_{21}[u,v]=1$ for all such
$u,v$. 
Therefore, $\lambda_{(2,1,1)}=1$ in the notation of (\ref{eqn:wtd-srg}).
\end{itemize}

The matrix $N_{11}$, whose $u,v$-entry is
$|N(u,1)\cap N(v,1)|$, is:

\[
\left(\begin{array}{rrrrrrrrr}
4 & 1 & 1 & 2 & 2 & 1 & 2 & 1 & 2 \\
1 & 4 & 1 & 1 & 2 & 2 & 2 & 2 & 1 \\
1 & 1 & 4 & 2 & 1 & 2 & 1 & 2 & 2 \\
2 & 1 & 2 & 4 & 1 & 1 & 2 & 2 & 1 \\
2 & 2 & 1 & 1 & 4 & 1 & 1 & 2 & 2 \\
1 & 2 & 2 & 1 & 1 & 4 & 2 & 1 & 2 \\
2 & 2 & 1 & 2 & 1 & 2 & 4 & 1 & 1 \\
1 & 2 & 2 & 2 & 2 & 1 & 1 & 4 & 1 \\
2 & 1 & 2 & 1 & 2 & 2 & 1 & 1 & 4
\end{array}\right).
\]
\begin{itemize}
\item
Are all the values of $N_{11}[u,v]$ the same if $u,v$ are distinct
vertices which are not neighbors? Yes:  $N_{11}[u,v]=2$ for all such
$u,v$. 
Therefore, $\mu_{(1,1)}=2$ in the notation of (\ref{eqn:wtd-srg}).
\item
Are all the values of $N_{11}[v,v]$ the same? Yes:  
$N_{11}[v,v]=4$ for all $v$.
Therefore, $k_{(1,1)}=4$ in the notation of (\ref{eqn:wtd-srg}).
\item
Are all the values of $N_{11}[u,v]$ the same if $u,v$ are neighbors
with edge-weight $2$? Yes:  $N_{11}[u,v]=2$ for all such $u,v$. 
Therefore, $\lambda_{(1,1,2)}=2$ in the notation of (\ref{eqn:wtd-srg}).
\item
Are all the values of $N_{11}[u,v]$ the same if $u,v$ are neighbors
with edge-weight $1$? Yes:  $N_{11}[u,v]=1$ for all such $u,v$. 
Therefore, $\lambda_{(1,1,1)}=1$ in the notation of (\ref{eqn:wtd-srg}).
\end{itemize}

In summary, we have 

\[
\mu_{(1,1)} = 2,\ k_{(1,1)} = 4, \ \lambda_{(1,1,1)} = 1, \ \lambda_{(1,1,2)} = 2, \ 
\]
\[
\mu_{(1,2)} = 2,\ k_{(1,2)} = 0, \ \lambda_{(1,2,1)} = 1, \ \lambda_{(1,2,2)} = 0, \ 
\]
\[
\mu_{(2,2)} = 0,\ k_{(2,2)} = 2, \ \lambda_{(2,2,1)} = 0, \ \lambda_{(2,2,2)} = 1.
\]

This verifies the statements in the conclusion of
Analog \ref{analog:BC-SRG} for
this function. In other words, the associated edge-weighted Cayley graph
is strongly regular. (However, $f$ is not bent.)
}
\end{example}

\begin{example}
{\rm

Consider the even function $f:GF(3)^2\to GF(3)$  with the 
following values:

\vskip .15in
{\footnotesize{
\begin{tabular}{c|ccccccccc}
$GF(3)^2$ & (0, 0) & (1, 0) & (2, 0) & (0, 1) & (1, 1) & (2, 1) & (0,
2) & (1, 2) & (2, 2) \\ \hline
$f$ & 0  & 2  & 2  & 2  & 0  & 1  & 2  & 1  & 0 \\
\end{tabular}
}}

\vskip .15in

This $f$ has algebraic normal form
\[
x_0^2x_1^2 +2 x_0^2 + x_0x_1 +2 x_1^2,
\]
and is not bent and non-homogeneous.
The weighted adjacency matrix of its Cayley graph is  

\[
\left(\begin{array}{rrrrrrrrr}
0 & 2 & 2 & 2 & 0 & 1 & 2 & 1 & 0 \\
2 & 0 & 2 & 1 & 2 & 0 & 0 & 2 & 1 \\
2 & 2 & 0 & 0 & 1 & 2 & 1 & 0 & 2 \\
2 & 1 & 0 & 0 & 2 & 2 & 2 & 0 & 1 \\
0 & 2 & 1 & 2 & 0 & 2 & 1 & 2 & 0 \\
1 & 0 & 2 & 2 & 2 & 0 & 0 & 1 & 2 \\
2 & 0 & 1 & 2 & 1 & 0 & 0 & 2 & 2 \\
1 & 2 & 0 & 0 & 2 & 1 & 2 & 0 & 2 \\
0 & 1 & 2 & 1 & 0 & 2 & 2 & 2 & 0
\end{array}\right).
\]
We have 

\[
\mu_{(1,1)} = 0,\ k_{(1,1)} = 2, \ \lambda_{(1,1,1)} = 1, \ \lambda_{(1,1,2)} = 0, \ 
\]
\[
\mu_{(1,2)} = 2,\ k_{(1,2)} = 0, \ \lambda_{(1,2,1)} = 0, \ \lambda_{(1,2,2)} = 1, \ 
\]
\[
\mu_{(2,2)} = 2,\ k_{(2,2)} = 4, \ \lambda_{(2,2,1)} = 2, \ \lambda_{(2,2,2)} = 1.
\]
This verifies the statements in the conclusion of
Conjecture \ref{conj:weighted-cayley} for
this function. (Again, $f$ is not bent.)
}
\end{example}

\begin{example}
{\rm

Consider the even function $f:GF(3)^2\to GF(3)$  with the 
following values:

\vskip .15in
{\footnotesize{
\begin{tabular}{c|ccccccccc}
$GF(3)^2$ & (0, 0) & (1, 0) & (2, 0) & (0, 1) & (1, 1) & (2, 1) & (0,
2) & (1, 2) & (2, 2) \\ \hline
$f$ & 0  & 0  & 0  & 2  & 0  & 1  & 2  & 1  & 0 \\
\end{tabular}
}}

\vskip .15in
%The Cayley graph $\Gamma$ of $f$
%is given in Figure \ref{fig:GF3-bent9}.

%\begin{figure}[t!]
%\begin{minipage}{\textwidth}
%\begin{center}
%\drawpng{height=9cm,width=12cm}{images/hadamard_transform_GF3-bent9}
%\end{center}
%\end{minipage}
%\caption{The weighted Cayley graph of a bent even $GF(3)$-valued function, $b_9$.}
%\label{fig:GF3-bent9}
%\end{figure}

This $f$ has algebraic normal form
\[
x_0x_1 +2 x_1^2,
\]
and is bent and homogeneous.
The weighted adjacency matrix of its Cayley graph is  

\[
\left(\begin{array}{rrrrrrrrr}
0 & 0 & 0 & 2 & 0 & 1 & 2 & 1 & 0 \\
0 & 0 & 0 & 1 & 2 & 0 & 0 & 2 & 1 \\
0 & 0 & 0 & 0 & 1 & 2 & 1 & 0 & 2 \\
2 & 1 & 0 & 0 & 0 & 0 & 2 & 0 & 1 \\
0 & 2 & 1 & 0 & 0 & 0 & 1 & 2 & 0 \\
1 & 0 & 2 & 0 & 0 & 0 & 0 & 1 & 2 \\
2 & 0 & 1 & 2 & 1 & 0 & 0 & 0 & 0 \\
1 & 2 & 0 & 0 & 2 & 1 & 0 & 0 & 0 \\
0 & 1 & 2 & 1 & 0 & 2 & 0 & 0 & 0
\end{array}\right).
\]
We have 

\[
\mu_{(1,1)} = 0,\ k_{(1,1)} = 2, \ \lambda_{(1,1,1)} = 1, \ \lambda_{(1,1,2)} = 0, \ 
\]
\[
\mu_{(1,2)} = 1,\ k_{(1,2)} = 0, \ \lambda_{(1,2,1)} = 0, \ \lambda_{(1,2,2)} = 0, \ 
\]
\[
\mu_{(2,2)} = 0,\ k_{(2,2)} = 2, \ \lambda_{(2,2,1)} = 0, \ \lambda_{(2,2,2)} = 1.
\]
This verifies the statements in the conclusion of
Conjecture \ref{conj:weighted-cayley} for
this function, $f=b_9$.
}
\end{example}

The last-named author (SW) has made the following observation.

\begin{proposition}
\label{prop:SW}
Let $f:GF(3)^2\to GF(3)$ be an even bent function with $f(0)=0$.
If the level curves of $f$,

\[
D_i = \{v\in GF(3)^2\ |\ f(v)=i\},
\]
yield a weighted PDS with
intersection numbers $p_{ij}^k$ then one of the following occurs.

\begin{enumerate}

\item
We have $|D_1|=|D_2|=2$, and the intersection numbers $p_{ij}^k$
are given as follows:

\[
\begin{array}{cc}
\begin{array}{c|cccc}
p_{ij}^0 & 0 & 1 & 2 & 3 \\ \hline
0         & 1 & 0 & 0 & 0 \\
1         & 0 & 2 & 0 & 0 \\
2         & 0 & 0 & 2 & 0 \\
3         & 0 & 0 & 0 & 4 \\
\end{array}
 &
\begin{array}{c|cccc}
p_{ij}^1 & 0 & 1 & 2 & 3 \\ \hline
0         & 0 & 1 & 0 & 0 \\
1         & 1 & 1 & 0 & 0 \\
2         & 0 & 0 & 0 & 2 \\
3         & 0 & 0 & 2 & 2 \\
\end{array} \\
 & \\
\begin{array}{c|cccc}
p_{ij}^2 & 0 & 1 & 2 & 3 \\ \hline
0         & 0 & 0 & 1 & 0 \\
1         & 0 & 0 & 0 & 2 \\
2         & 1 & 0 & 1 & 0 \\
3         & 0 & 2 & 0 & 2 \\
\end{array}
 &
\begin{array}{c|cccc}
p_{ij}^3 & 0 & 1 & 2 & 3 \\ \hline
0         & 0 & 0 & 0 & 1 \\
1         & 0 & 0 & 1 & 1 \\
2         & 0 & 1 & 0 & 1 \\
3         & 1 & 1 & 1 & 1 \\
\end{array} \\
\end{array}
\]
Furthermore, $D = D_1 \cup D_2$ is a  
$(9,4,1,2)$-PDS of 
Latin square type ($N=3$ and $R=2$) and 
negative Latin square type ($N=-3$ and $R=-1$).
\item
We have $|D_1|=|D_2|=4$, $D_3=\emptyset$,
and the intersection numbers $p_{ij}^k$
are given as follows:

\[
\begin{array}{cc}
\begin{array}{c|ccc}
p_{ij}^0 & 0 & 1 & 2  \\ \hline
0         & 1 & 0 & 0  \\
1         & 0 & 4 & 0  \\
2         & 0 & 0 & 4  \\
\end{array}
 &
\begin{array}{c|cccc}
p_{ij}^1 & 0 & 1 & 2  \\ \hline
0         & 0 & 1 & 0  \\
1         & 1 & 1 & 2  \\
2         & 0 & 2 & 2  \\
\end{array} \\
 & \\
\begin{array}{c|cccc}
p_{ij}^2 & 0 & 1 & 2  \\ \hline
0         & 0 & 0 & 1  \\
1         & 0 & 2 & 2  \\
2         & 1 & 2 & 1  \\
\end{array}
 &
{\rm no} \  p_{ij}^3 \\
\end{array}
\]

\end{enumerate}

\end{proposition}

This is verified using a case-by-case analysis.

\subsection{$GF(3)^3\to GF(3)$}

We can classify some bent functions on $GF(3)^3$
in terms of the corresponding combinatorial structure of their
level curves. Unlike the $GF(3)^2$ case, not all such bent 
functions have ``combinatorial'' level curves.

\begin{proposition}
There are $2340$ even bent functions
$f:GF(3)^3\to GF(3)$ such that $f(0)=0$. The group $G=GL(3,GF(3))$
acts on the set ${\mathbb{B}}$ of all such bent functions and 
there are $4$ orbits in ${\mathbb{B}}/G$:

\[
{\mathbb{B}}/G = B_1\cup B_2\cup B_3\cup B_4,
\]
where $|B_1|=234$,  $|B_2|=936$, $|B_3|=234$,
and $|B_4|=936$.

\begin{table}
\centering
\setlength{\extrarowheight}{3pt}
\begin{tabular}{|l|l|}
\hline 
$B_1$ & $f_1(x_0,x_1,x_2)=x_0^2 + x_1^2+x_2^2$ \\
 \hline
$B_2$ & $f_2(x_0,x_1,x_2)=x_0x_2 + 2x_1^2+2x_0^2 x_1^2$ \\
\hline
$B_3$ & $f_3(x_0,x_1,x_2)=-x_0^2 - x_1^2-x_2^2$ \\
\hline
$B_4$ & $f_4(x_0,x_1,x_2)=-x_0x_2 - 2x_1^2-2x_0^2 x_1^2$ \\
\hline
\end{tabular}
\caption{Representatives of orbits in $\mathbb B/G$}
\label{tab:tab33}
\end{table}

The bent functions which give rise to a weighted PDS\footnote{Note, 
the weighted PDSs are given in the examples below.} are
those in orbits $B_1$ and $B_3$. The 
other bent functions do not.

The functions in orbits $B_1$ and $B_3$ are weakly regular but not regular. 
The functions in orbits $B_2$ and $B_4$ are not weakly regular.
\end{proposition}

\begin{remark}
The result above agrees with the results of Pott et al \cite{PTFL},
where they overlap.

\end{remark}

\begin{example}
{\rm

Consider the example of the even function $f:GF(3)^3\to GF(3)$  given 
in \S \ref{sec:bent-gf3**3} below.
The adjacency matrix of its edge-weighted Cayley graph $\Gamma = (V,E)$ is given below.

\vskip .15in

{\scriptsize{
\begin{Verbatim}[fontsize=\scriptsize,fontfamily=courier,fontshape=tt,frame=single,label=\sage]

sage: FF = GF(3)
sage: V = FF^3
sage: Vlist = V.list()
sage: Vlist
[(0, 0, 0), (1, 0, 0), (2, 0, 0), (0, 1, 0), (1, 1, 0), (2, 1, 0), (0, 2, 0), 
 (1, 2, 0), (2, 2, 0), (0, 0, 1), (1, 0, 1), (2, 0, 1), (0, 1, 1), (1, 1, 1), 
 (2, 1, 1), (0, 2, 1), (1, 2, 1), (2, 2, 1), (0, 0, 2), (1, 0, 2), (2, 0, 2), 
 (0, 1, 2), (1, 1, 2), (2, 1, 2), (0, 2, 2), (1, 2, 2), (2, 2, 2)]
sage: flist = [0,2,2,1,1,1,1,1,1,2,0,1,1,2,0,1,0,0,2,1,0,1,0,0,1,0,2]
sage: f = lambda x: GF(3)(flist[Vlist.index(x)])   
sage: [CC(hadamard_transform(f,a)).abs() for a in V]
[5.19615242270663, 5.19615242270663, 5.19615242270664, 
5.19615242270663, 5.19615242270663, 5.19615242270663,
 5.19615242270663, 5.19615242270663, 5.19615242270663,
 5.19615242270663, 5.19615242270663, 5.19615242270663,
 5.19615242270663, 5.19615242270664, 5.19615242270663,
 5.19615242270663, 5.19615242270663, 5.19615242270663,
 5.19615242270664, 5.19615242270663, 5.19615242270663,
 5.19615242270663, 5.19615242270663, 5.19615242270663,
 5.19615242270663, 5.19615242270663, 5.19615242270664]
sage: Gamma = boolean_cayley_graph(f, V)
sage: Gamma.spectrum()
[18, 3, 3, 3, 3, 3, 3, 0, 0, 0, 0, 0, 0, 0, 0, -3, -3, -3, -3, -3, -3, -3, 
 -3, -3, -3, -3, -3]

\end{Verbatim}
}}
\vskip .1in

The algebraic normal form of $f$ is

\[
x_0x_2 +2x_1^2 + 2 x_0^2x_1^2 \, ,
\]
which is non-homogeneous but bent. It is not regular, nor merely weakly regular.
The weighted adjacency matrix is 
%sage: Alist = [[ZZ(f(x-y)) for x in V] for y in V]
%sage: A = matrix(GF(3), Alist)

{\scriptsize{
\[
\left(\begin{array}{rrrrrrrrrrrrrrrrrrrrrrrrrrr}
0 & 2 & 2 & 1 & 1 & 1 & 1 & 1 & 1 & 2 & 0 & 1 & 1 & 2 & 0 & 1 & 0 & 0 & 2 & 1 & 0 & 1 & 0 & 0 & 1 & 0 & 2 \\
2 & 0 & 2 & 1 & 1 & 1 & 1 & 1 & 1 & 1 & 2 & 0 & 0 & 1 & 2 & 0 & 1 & 0 & 0 & 2 & 1 & 0 & 1 & 0 & 2 & 1 & 0 \\
2 & 2 & 0 & 1 & 1 & 1 & 1 & 1 & 1 & 0 & 1 & 2 & 2 & 0 & 1 & 0 & 0 & 1 & 1 & 0 & 2 & 0 & 0 & 1 & 0 & 2 & 1 \\
1 & 1 & 1 & 0 & 2 & 2 & 1 & 1 & 1 & 1 & 0 & 0 & 2 & 0 & 1 & 1 & 2 & 0 & 1 & 0 & 2 & 2 & 1 & 0 & 1 & 0 & 0 \\
1 & 1 & 1 & 2 & 0 & 2 & 1 & 1 & 1 & 0 & 1 & 0 & 1 & 2 & 0 & 0 & 1 & 2 & 2 & 1 & 0 & 0 & 2 & 1 & 0 & 1 & 0 \\
1 & 1 & 1 & 2 & 2 & 0 & 1 & 1 & 1 & 0 & 0 & 1 & 0 & 1 & 2 & 2 & 0 & 1 & 0 & 2 & 1 & 1 & 0 & 2 & 0 & 0 & 1 \\
1 & 1 & 1 & 1 & 1 & 1 & 0 & 2 & 2 & 1 & 2 & 0 & 1 & 0 & 0 & 2 & 0 & 1 & 1 & 0 & 0 & 1 & 0 & 2 & 2 & 1 & 0 \\
1 & 1 & 1 & 1 & 1 & 1 & 2 & 0 & 2 & 0 & 1 & 2 & 0 & 1 & 0 & 1 & 2 & 0 & 0 & 1 & 0 & 2 & 1 & 0 & 0 & 2 & 1 \\
1 & 1 & 1 & 1 & 1 & 1 & 2 & 2 & 0 & 2 & 0 & 1 & 0 & 0 & 1 & 0 & 1 & 2 & 0 & 0 & 1 & 0 & 2 & 1 & 1 & 0 & 2 \\
2 & 1 & 0 & 1 & 0 & 0 & 1 & 0 & 2 & 0 & 2 & 2 & 1 & 1 & 1 & 1 & 1 & 1 & 2 & 0 & 1 & 1 & 2 & 0 & 1 & 0 & 0 \\
0 & 2 & 1 & 0 & 1 & 0 & 2 & 1 & 0 & 2 & 0 & 2 & 1 & 1 & 1 & 1 & 1 & 1 & 1 & 2 & 0 & 0 & 1 & 2 & 0 & 1 & 0 \\
1 & 0 & 2 & 0 & 0 & 1 & 0 & 2 & 1 & 2 & 2 & 0 & 1 & 1 & 1 & 1 & 1 & 1 & 0 & 1 & 2 & 2 & 0 & 1 & 0 & 0 & 1 \\
1 & 0 & 2 & 2 & 1 & 0 & 1 & 0 & 0 & 1 & 1 & 1 & 0 & 2 & 2 & 1 & 1 & 1 & 1 & 0 & 0 & 2 & 0 & 1 & 1 & 2 & 0 \\
2 & 1 & 0 & 0 & 2 & 1 & 0 & 1 & 0 & 1 & 1 & 1 & 2 & 0 & 2 & 1 & 1 & 1 & 0 & 1 & 0 & 1 & 2 & 0 & 0 & 1 & 2 \\
0 & 2 & 1 & 1 & 0 & 2 & 0 & 0 & 1 & 1 & 1 & 1 & 2 & 2 & 0 & 1 & 1 & 1 & 0 & 0 & 1 & 0 & 1 & 2 & 2 & 0 & 1 \\
1 & 0 & 0 & 1 & 0 & 2 & 2 & 1 & 0 & 1 & 1 & 1 & 1 & 1 & 1 & 0 & 2 & 2 & 1 & 2 & 0 & 1 & 0 & 0 & 2 & 0 & 1 \\
0 & 1 & 0 & 2 & 1 & 0 & 0 & 2 & 1 & 1 & 1 & 1 & 1 & 1 & 1 & 2 & 0 & 2 & 0 & 1 & 2 & 0 & 1 & 0 & 1 & 2 & 0 \\
0 & 0 & 1 & 0 & 2 & 1 & 1 & 0 & 2 & 1 & 1 & 1 & 1 & 1 & 1 & 2 & 2 & 0 & 2 & 0 & 1 & 0 & 0 & 1 & 0 & 1 & 2 \\
2 & 0 & 1 & 1 & 2 & 0 & 1 & 0 & 0 & 2 & 1 & 0 & 1 & 0 & 0 & 1 & 0 & 2 & 0 & 2 & 2 & 1 & 1 & 1 & 1 & 1 & 1 \\
1 & 2 & 0 & 0 & 1 & 2 & 0 & 1 & 0 & 0 & 2 & 1 & 0 & 1 & 0 & 2 & 1 & 0 & 2 & 0 & 2 & 1 & 1 & 1 & 1 & 1 & 1 \\
0 & 1 & 2 & 2 & 0 & 1 & 0 & 0 & 1 & 1 & 0 & 2 & 0 & 0 & 1 & 0 & 2 & 1 & 2 & 2 & 0 & 1 & 1 & 1 & 1 & 1 & 1 \\
1 & 0 & 0 & 2 & 0 & 1 & 1 & 2 & 0 & 1 & 0 & 2 & 2 & 1 & 0 & 1 & 0 & 0 & 1 & 1 & 1 & 0 & 2 & 2 & 1 & 1 & 1 \\
0 & 1 & 0 & 1 & 2 & 0 & 0 & 1 & 2 & 2 & 1 & 0 & 0 & 2 & 1 & 0 & 1 & 0 & 1 & 1 & 1 & 2 & 0 & 2 & 1 & 1 & 1 \\
0 & 0 & 1 & 0 & 1 & 2 & 2 & 0 & 1 & 0 & 2 & 1 & 1 & 0 & 2 & 0 & 0 & 1 & 1 & 1 & 1 & 2 & 2 & 0 & 1 & 1 & 1 \\
1 & 2 & 0 & 1 & 0 & 0 & 2 & 0 & 1 & 1 & 0 & 0 & 1 & 0 & 2 & 2 & 1 & 0 & 1 & 1 & 1 & 1 & 1 & 1 & 0 & 2 & 2 \\
0 & 1 & 2 & 0 & 1 & 0 & 1 & 2 & 0 & 0 & 1 & 0 & 2 & 1 & 0 & 0 & 2 & 1 & 1 & 1 & 1 & 1 & 1 & 1 & 2 & 0 & 2 \\
2 & 0 & 1 & 0 & 0 & 1 & 0 & 1 & 2 & 0 & 0 & 1 & 0 & 2 & 1 & 1 & 0 & 2 & 1 & 1 & 1 & 1 & 1 & 1 & 2 & 2 & 0
\end{array}\right)
\]
}}

We have 

\[
\mu_{(1,1)} = \{4, 6\},\ k_{(1,1)} = 12, \ \lambda_{(1,1,1)} = 5, \ \lambda_{(1,1,2)} = 6, \ 
\]
\[
\mu_{(1,2)} = 3,\ k_{(1,2)} = 0, \ \lambda_{(1,2,1)} = \{2,4\}, \ \lambda_{(1,2,2)} = 2, \ 
\]
\[
\mu_{(2,2)} = \{0,2\},\ k_{(2,2)} = 6, \ \lambda_{(2,2,1)} = \{0,2\}, \ \lambda_{(2,2,2)} = 1.
\]

In this example, Analog \ref{conj:weighted-cayley} is false.

}
\end{example}

\begin{example}
{\rm

Consider the example of the bent even function $f:GF(3)^3\to GF(3)$  given 
by

\[
f(x_0,x_1)=x_0x_1+x_2^2 \, ,
\]
which is homogeneous but bent. It is not weakly regular.
The adjacency matrix of the associated edge-weighted Cayley graph is 

{\scriptsize{
\[
\left(\begin{array}{rrrrrrrrrrrrrrrrrrrrrrrrrrr}
0 & 0 & 0 & 0 & 1 & 2 & 0 & 2 & 1 & 1 & 1 & 1 & 1 & 2 & 0 & 1 & 0 & 2 & 1 & 1 & 1 & 1 & 2 & 0 & 1 & 0 & 2 \\
0 & 0 & 0 & 2 & 0 & 1 & 1 & 0 & 2 & 1 & 1 & 1 & 0 & 1 & 2 & 2 & 1 & 0 & 1 & 1 & 1 & 0 & 1 & 2 & 2 & 1 & 0 \\
0 & 0 & 0 & 1 & 2 & 0 & 2 & 1 & 0 & 1 & 1 & 1 & 2 & 0 & 1 & 0 & 2 & 1 & 1 & 1 & 1 & 2 & 0 & 1 & 0 & 2 & 1 \\
0 & 2 & 1 & 0 & 0 & 0 & 0 & 1 & 2 & 1 & 0 & 2 & 1 & 1 & 1 & 1 & 2 & 0 & 1 & 0 & 2 & 1 & 1 & 1 & 1 & 2 & 0 \\
1 & 0 & 2 & 0 & 0 & 0 & 2 & 0 & 1 & 2 & 1 & 0 & 1 & 1 & 1 & 0 & 1 & 2 & 2 & 1 & 0 & 1 & 1 & 1 & 0 & 1 & 2 \\
2 & 1 & 0 & 0 & 0 & 0 & 1 & 2 & 0 & 0 & 2 & 1 & 1 & 1 & 1 & 2 & 0 & 1 & 0 & 2 & 1 & 1 & 1 & 1 & 2 & 0 & 1 \\
0 & 1 & 2 & 0 & 2 & 1 & 0 & 0 & 0 & 1 & 2 & 0 & 1 & 0 & 2 & 1 & 1 & 1 & 1 & 2 & 0 & 1 & 0 & 2 & 1 & 1 & 1 \\
2 & 0 & 1 & 1 & 0 & 2 & 0 & 0 & 0 & 0 & 1 & 2 & 2 & 1 & 0 & 1 & 1 & 1 & 0 & 1 & 2 & 2 & 1 & 0 & 1 & 1 & 1 \\
1 & 2 & 0 & 2 & 1 & 0 & 0 & 0 & 0 & 2 & 0 & 1 & 0 & 2 & 1 & 1 & 1 & 1 & 2 & 0 & 1 & 0 & 2 & 1 & 1 & 1 & 1 \\
1 & 1 & 1 & 1 & 2 & 0 & 1 & 0 & 2 & 0 & 0 & 0 & 0 & 1 & 2 & 0 & 2 & 1 & 1 & 1 & 1 & 1 & 2 & 0 & 1 & 0 & 2 \\
1 & 1 & 1 & 0 & 1 & 2 & 2 & 1 & 0 & 0 & 0 & 0 & 2 & 0 & 1 & 1 & 0 & 2 & 1 & 1 & 1 & 0 & 1 & 2 & 2 & 1 & 0 \\
1 & 1 & 1 & 2 & 0 & 1 & 0 & 2 & 1 & 0 & 0 & 0 & 1 & 2 & 0 & 2 & 1 & 0 & 1 & 1 & 1 & 2 & 0 & 1 & 0 & 2 & 1 \\
1 & 0 & 2 & 1 & 1 & 1 & 1 & 2 & 0 & 0 & 2 & 1 & 0 & 0 & 0 & 0 & 1 & 2 & 1 & 0 & 2 & 1 & 1 & 1 & 1 & 2 & 0 \\
2 & 1 & 0 & 1 & 1 & 1 & 0 & 1 & 2 & 1 & 0 & 2 & 0 & 0 & 0 & 2 & 0 & 1 & 2 & 1 & 0 & 1 & 1 & 1 & 0 & 1 & 2 \\
0 & 2 & 1 & 1 & 1 & 1 & 2 & 0 & 1 & 2 & 1 & 0 & 0 & 0 & 0 & 1 & 2 & 0 & 0 & 2 & 1 & 1 & 1 & 1 & 2 & 0 & 1 \\
1 & 2 & 0 & 1 & 0 & 2 & 1 & 1 & 1 & 0 & 1 & 2 & 0 & 2 & 1 & 0 & 0 & 0 & 1 & 2 & 0 & 1 & 0 & 2 & 1 & 1 & 1 \\
0 & 1 & 2 & 2 & 1 & 0 & 1 & 1 & 1 & 2 & 0 & 1 & 1 & 0 & 2 & 0 & 0 & 0 & 0 & 1 & 2 & 2 & 1 & 0 & 1 & 1 & 1 \\
2 & 0 & 1 & 0 & 2 & 1 & 1 & 1 & 1 & 1 & 2 & 0 & 2 & 1 & 0 & 0 & 0 & 0 & 2 & 0 & 1 & 0 & 2 & 1 & 1 & 1 & 1 \\
1 & 1 & 1 & 1 & 2 & 0 & 1 & 0 & 2 & 1 & 1 & 1 & 1 & 2 & 0 & 1 & 0 & 2 & 0 & 0 & 0 & 0 & 1 & 2 & 0 & 2 & 1 \\
1 & 1 & 1 & 0 & 1 & 2 & 2 & 1 & 0 & 1 & 1 & 1 & 0 & 1 & 2 & 2 & 1 & 0 & 0 & 0 & 0 & 2 & 0 & 1 & 1 & 0 & 2 \\
1 & 1 & 1 & 2 & 0 & 1 & 0 & 2 & 1 & 1 & 1 & 1 & 2 & 0 & 1 & 0 & 2 & 1 & 0 & 0 & 0 & 1 & 2 & 0 & 2 & 1 & 0 \\
1 & 0 & 2 & 1 & 1 & 1 & 1 & 2 & 0 & 1 & 0 & 2 & 1 & 1 & 1 & 1 & 2 & 0 & 0 & 2 & 1 & 0 & 0 & 0 & 0 & 1 & 2 \\
2 & 1 & 0 & 1 & 1 & 1 & 0 & 1 & 2 & 2 & 1 & 0 & 1 & 1 & 1 & 0 & 1 & 2 & 1 & 0 & 2 & 0 & 0 & 0 & 2 & 0 & 1 \\
0 & 2 & 1 & 1 & 1 & 1 & 2 & 0 & 1 & 0 & 2 & 1 & 1 & 1 & 1 & 2 & 0 & 1 & 2 & 1 & 0 & 0 & 0 & 0 & 1 & 2 & 0 \\
1 & 2 & 0 & 1 & 0 & 2 & 1 & 1 & 1 & 1 & 2 & 0 & 1 & 0 & 2 & 1 & 1 & 1 & 0 & 1 & 2 & 0 & 2 & 1 & 0 & 0 & 0 \\
0 & 1 & 2 & 2 & 1 & 0 & 1 & 1 & 1 & 0 & 1 & 2 & 2 & 1 & 0 & 1 & 1 & 1 & 2 & 0 & 1 & 1 & 0 & 2 & 0 & 0 & 0 \\
2 & 0 & 1 & 0 & 2 & 1 & 1 & 1 & 1 & 2 & 0 & 1 & 0 & 2 & 1 & 1 & 1 & 1 & 1 & 2 & 0 & 2 & 1 & 0 & 0 & 0 & 0
\end{array}\right).
\]
}}

\begin{Verbatim}[fontsize=\scriptsize,fontfamily=courier,fontshape=tt,frame=single,label=\sage]

sage: p = 3; n = 3
sage: FF = GF(p)
sage:  V = GF(p)**n
sage: f = lambda x: FF(x[0]*x[1]+x[2]^2)   
sage: flist = [f(v) for v in V]
sage: flist
[0, 0, 0, 0, 1, 2, 0, 2, 1, 1, 1, 1, 1, 2, 0, 1, 0, 2, 1, 1, 1, 1, 2, 0, 1, 0, 2]
sage: 
sage: 
sage: [CC(hadamard_transform(f,a)).abs() for a in V]
[5.19615242270663, 5.19615242270663, 5.19615242270664, 
5.19615242270663, 5.19615242270663, 5.19615242270663,
 5.19615242270663, 5.19615242270663, 5.19615242270663,
 5.19615242270663, 5.19615242270663, 5.19615242270663,
 5.19615242270663, 5.19615242270664, 5.19615242270663,
 5.19615242270663, 5.19615242270663, 5.19615242270663,
 5.19615242270664, 5.19615242270663, 5.19615242270663,
 5.19615242270663, 5.19615242270663, 5.19615242270663,
 5.19615242270663, 5.19615242270663, 5.19615242270664]
sage: Gamma = boolean_cayley_graph(f, V)
sage: Gamma.spectrum()
[18, 3, 3, 3, 3, 3, 3, 0, 0, 0, 0, 0, 0, 0, 0, -3, -3, -3, -3, -3, -3, -3, 
 -3, -3, -3, -3, -3]
sage: Gamma.is_strongly_regular()
False

\end{Verbatim}

The unweighed Cayley graph of $f$ is regular but has four distinct eigenvalues, so is
not strongly regular (as Sage indicates above). However, a Sage computation shows
$|W_f(a)|=3^{3/2}$ for all $a \in GF(3)^3$, $f$ is bent.
Since $W_f(0)/3^{3/2}$ is a $6$-th root of unity but not a cube root, $f$ is not weakly
regular.

By a Sage computation, we have 

\[
\mu_{(1,1)} = 6,\ k_{(1,1)} = 12, \ \lambda_{(1,1,1)} = 5, \ \lambda_{(1,1,2)} = 4, \ 
\]
\[
\mu_{(1,2)} = 3,\ k_{(1,2)} = 0, \ \lambda_{(1,2,1)} = 2, \ \lambda_{(1,2,2)} = 4, \ 
\]
\[
\mu_{(2,2)} = 0,\ k_{(2,2)} = 6, \ \lambda_{(2,2,1)} = 2, \ \lambda_{(2,2,2)} = 1.
\]

% sage: r = (hadamard_transform(f,a))/3**(3/2)
% sage: r
% 1/9*(4*e^(4/3*I*pi) + e^(2/3*I*pi) + e^(4/3*I*pi) + 5*e^(2/3*I*pi) + e^(4/3*I*pi) + e^(2/3*I*pi) + 5*e^(4/3*I*pi) + e^(2/3*I*pi) + e^(4/3*I*pi) + e^(2/3*I*pi) + 6)*sqrt(3)
% sage: r^3
% 1/243*(4*e^(4/3*I*pi) + e^(2/3*I*pi) + e^(4/3*I*pi) + 5*e^(2/3*I*pi) + e^(4/3*I*pi) + e^(2/3*I*pi) + 5*e^(4/3*I*pi) + e^(2/3*I*pi) + e^(4/3*I*pi) + e^(2/3*I*pi) + 6)^3*sqrt(3)
% sage: CC(r^3)
% -3.24133957342678e-15 - 1.00000000000000*I
% sage: N,e,c,d = find_SRG_parameters(flist, 3, 3)
% sage: c; d; e
% [[0, 0, [0]], [0, 1, [0]], [0, 2, [0]], [1, 0, [0]], [1, 1, [12]], [1, 2, [0]], [2, 0, [0]], [2, 1, [0]], [2, 2, [6]]]
% [[0, 0, 0, [0]], [0, 0, 1, [0]], [0, 0, 2, [0]], [0, 1, 0, [0]], [0, 1, 1, [0]], [0, 1, 2, [0]], [0, 2, 0, [0]], [0, 2, 1, [0]], [0, 2, 2, [0]], [1, 0, 0, [0, 3]], [1, 0, 1, [5]], [1, 0, 2, [4]], [1, 1, 0, [12, 6]], [1, 1, 1, [5]], [1, 1, 2, [4]], [1, 2, 0, [0, 3]], [1, 2, 1, [2]], [1, 2, 2, [4]], [2, 0, 0, [0, 3]], [2, 0, 1, [2]], [2, 0, 2, [1]], [2, 1, 0, [0, 3]], [2, 1, 1, [2]], [2, 1, 2, [4]], [2, 2, 0, [0, 6]], [2, 2, 1, [2]], [2, 2, 2, [1]]]
% [[0, 0, [0]], [0, 1, [0]], [0, 2, [0]], [1, 0, [3]], [1, 1, [6]], [1, 2, [3]], [2, 0, [3]], [2, 1, [3]], [2, 2, [0]]]

In this example, Analog \ref{conj:weighted-cayley} is true.

}
\end{example}

Let $f:GF(3)^3\to GF(3)$ be an even bent function with $f(0)=0$,
let
\[
D_i = \{v\in GF(3)^3\ |\ f(v)=i\},\ \ \ i=1,2,
\]
let $D_0=\{0\}$ and $D_3=GF(3)^3\setminus (D_0\cup D_1\cup D_2)$.

The next result extends Proposition \ref{prop:SW}.

\begin{proposition}
\label{prop:gf33bent}
Let $f:GF(3)^3\to GF(3)$ be an even bent function with $f(0)=0$.
If the level curves of $f$, $D_i$, yield a weighted PDS with
intersection numbers $p_{ij}^k$ then one of the following occurs.

\begin{enumerate}

\item
We have $|D_1|=6$, $|D_2|=12$, and the intersection numbers $p_{ij}^k$
are given as follows:

\[
\begin{array}{cc}
\begin{array}{c|cccc}
p_{ij}^0 & 0 & 1 & 2 & 3 \\ \hline
0         & 1 & 0 & 0 & 0 \\
1         & 0 & 6 & 0 & 0 \\
2         & 0 & 0 & 12 & 0 \\
3         & 0 & 0 & 0 & 8 \\
\end{array}
 &
\begin{array}{c|cccc}
p_{ij}^1 & 0 & 1 & 2 & 3 \\ \hline
0         & 0 & 1 & 0 & 0 \\
1         & 1 & 1 & 4 & 0 \\
2         & 0 & 4 & 4 & 4 \\
3         & 0 & 0 & 4 & 4 \\
\end{array} \\
 & \\
\begin{array}{c|cccc}
p_{ij}^2 & 0 & 1 & 2 & 3 \\ \hline
0         & 0 & 0 & 1 & 0 \\
1         & 0 & 2 & 2 & 2 \\
2         & 1 & 2 & 5 & 4 \\
3         & 0 & 2 & 4 & 2 \\
\end{array}
 &
\begin{array}{c|cccc}
p_{ij}^3 & 0 & 1 & 2 & 3 \\ \hline
0         & 0 & 0 & 0 & 1 \\
1         & 0 & 0 & 3 & 3 \\
2         & 0 & 3 & 6 & 3 \\
3         & 1 & 3 & 3 & 1 \\
\end{array} \\
\end{array}
\]
\item
We have $|D_1|=12$, $|D_2|=6$, and the intersection numbers $p_{ij}^k$
are given as follows:

\[
\begin{array}{cc}
\begin{array}{c|cccc}
p_{ij}^0 & 0 & 1 & 2 & 3 \\ \hline
0         & 1 & 0 & 0 & 0 \\
1         & 0 & 12 & 0 & 0 \\
2         & 0 & 0 & 6 & 0 \\
3         & 0 & 0 & 0 & 8 \\
\end{array}
 &
\begin{array}{c|cccc}
p_{ij}^1 & 0 & 1 & 2 & 3 \\ \hline
0         & 0 & 1 & 0 & 0 \\
1         & 1 & 5 & 2 & 4 \\
2         & 0 & 2 & 2 & 2 \\
3         & 0 & 4 & 2 & 2 \\
\end{array} \\
 & \\
\begin{array}{c|cccc}
p_{ij}^2 & 0 & 1 & 2 & 3 \\ \hline
0         & 0 & 0 & 1 & 0 \\
1         & 0 & 4 & 4 & 4 \\
2         & 1 & 4 & 1 & 0 \\
3         & 0 & 4 & 0 & 4 \\
\end{array}
 &
\begin{array}{c|cccc}
p_{ij}^3 & 0 & 1 & 2 & 3 \\ \hline
0         & 0 & 0 & 0 & 1 \\
1         & 0 & 6 & 3 & 3 \\
2         & 0 & 3 & 0 & 3 \\
3         & 1 & 3 & 3 & 1 \\
\end{array} \\
\end{array}
\]

\end{enumerate}

\end{proposition}

In this case, 
if $f:GF(3)^3\to GF(3)$ satisfies the hypothesis of the above
proposition then $f$ is necessarily quadratic\footnote{However, this may very
likely have more to do with the fact that $p$ and $n$ are so small.}.

One way to investigate this question is to partition the set of even functions into 
equivalence classes with respect to the group action of 
$GL(3,GF(3))$, then pick a representative from each class and test for
bentness. Once we know which orbits under $GL(3,GF(3))$ are bent, we can
check the conjecture and the question for a representative from each orbit.

What are these orbits?

Consider the set $\mathbb{E}$ of all functions $f:GF(3)^3 \rightarrow GF(3)$ such that 
\begin{itemize}
\item $f$ is even,
\item $f(0)=0$, and
\item the degree of the algebraic normal form of $f$ is at most $4$.
\end{itemize}

The algebraic normal form of such a function must be of the form 
\[
f(x_0,x_1,x_2)=a_1 x_0^2 + a_2 x_0 x_1 + a_3 x_0 x_2 + a_4 x_1^2 +a_5 x_1 x_2 + a_6 x_2^2 \
\]
\[
+ b_1 x_0^2 x_1^2 +b_2 x_0^2 x_1 x_2 + b_3 x_0^2 x_2^2 + b_4 x_0 x_1^2 x_2 
+ b_5 x_0 x_1 x_2^2 + b_6 x_1^2 x_2^2
\]
where $a_1, ... , a_6, b_1, ... , b_6$ are in $GF(3)^3$.
Thus there are $3^{12} = 531,441$ such functions.  Recall the {\it signature} of $f$ is
the sequence of cardinalities of the level curves

\[
D_i = \{x\in GF(3)^3\ |\ f(x)=i\}.
\]

Let $G = GL(3,GF(3))$ be the set of nondegenerate linear transformations 
$\phi:GF(3)^3 \rightarrow GF(3)^3$.  This group acts on ${\mathbb{E}}$ 
in a natural way and we say $f\in {\mathbb{E}}$ is {\it equivalent} to
$g\in {\mathbb{E}}$ if and only if $f$ is sent to $g$ under some
element of $G$. An equivalence class is simply an orbit
in ${\mathbb{E}}$ under this action of $G$. 
Mathematica was used to calculate that $|G|=11232$. 
However, since $f(\phi(x))=f(-\phi(x))$ for all $\phi$ in $G$ and $x$ in $GF(3)^3$, 
there are at most $5616$ functions in the equivalence class of any nonzero element of 
$\mathbb{E}$.  

If $f$ is bent, then so is $f \circ \phi$, for $\phi$ in $G$.  Therefore, one way 
to find all bent functions in $\mathbb{E}$ is to partition $\mathbb{B}$ into equivalence 
classes under the action of $G$ and test an element of each equivalence class 
to see if it is bent.  However, the computational time for attacking this problem directly was prohibitive.  

We next note that the size of the level curves $f^{-1}(1)$ and $f^{-1}(2)$ is preserved 
under the action of elements of $G$, i.e., the signature of $f$ is the same 
for all functions in each equivalence class.  Mathematica was used to partition 
$\mathbb{E}$ into sets with the same signature.  There are $120120$ elements of 
$\mathbb{E}$ or signature $(6,12)$ or $(12,6)$. There are $35$ signatures that 
occur. The sizes of the signature equivalence classes range from $0$ (for the zero function) 
to $90090$ for $|D_1|=|D_2|=8$.  

Mathematica was then used to find all equivalence classes of functions in 
$\mathbb{E}$ under  transformations in $G$ for each of the $35$ signature 
equivalence classes.  There are a total of $281$ equivalence classes of functions in 
$\mathbb{E}$ under the action of $GL(3,GF(3))$.
Of these, $4$ classes consist of bent functions.  In other words, if 
$\mathbb{B}$ denotes the subset of $\mathbb{E}$ consisting of bent
functions then $G$ acts on $\mathbb{B}$ and the number of orbits is $4$.

There were two equivalence classes of bent functions of
type $|D_1|=6$ and $|D_2|=12$.   The other two 
bent classes were of type $|D_1|=12$ and $|D_2|=6$ and consisted of the negatives 
of the functions in the first two classes.  We will call the classes $B_1$, $B_2$, $B_3$, 
and $B_4$:

\[
{\mathbb{B}}/G = B_1\cup B_2\cup B_3\cup B_4.
\]
Note the $(6,12)$ classes are negatives of the $(12,6)$
classes, so after a possible reindexing, we have $B_3=-B_1$ and $B_4=-B_2$.  

A representative of $B_1$ is

\[
x_1^2+x_2^2+x_3^2.
\]
There are $234$ bent functions in its equivalence class under nondegenerate linear 
transformations. Note that the algebraic normal form of each function in this class 
is quadratic.

A representative of $B_2$ is 

\[
x_0 x_2+2x_1^2+2x_0^2 x_1^2.
\]
There are $936$ bent functions in its equivalence class under nondegenerate linear transformations.

Thus there are a total of $2340$ bent functions in $\mathbb{B}$.

We know that if $W_f(0)$ is rational then the level curves
$f^{-1}(i)$, $i\not= 0$, have the same cardinality\footnote{In fact, 
this is true any time $n$ is even (see \cite{CTZ}).}.
The following Sage computation shows that, in this case, $W_f(0)$ is not a
rational number for the representatives of $B_1,B_2$ displayed above.

\vskip .15in
{\footnotesize{
\begin{Verbatim}[fontsize=\scriptsize,fontfamily=courier,fontshape=tt,frame=single,label=\sage]

sage: PR.<x0,x1,x2> = PolynomialRing(FF, 3, "x0,x1,x2")
sage: f = x0^2 + x1^2 + x2^2
sage: V = GF(3)^3
sage: Vlist = V.list()
sage: flist = [f(x[0],x[1],x[2]) for x in Vlist]
sage: f = lambda x: FF(flist[Vlist.index(x)])
sage: hadamard_transform(f,V(0))
12*e^(4/3*I*pi) + 6*e^(2/3*I*pi) + 9
sage: CC(hadamard_transform(f,V(0)))
-3.99680288865056e-15 - 5.19615242270663*I
sage: f = x0*x2 + 2*x1^2 + 2*x0^2*x2^2
sage: flist = [f(x[0],x[1],x[2]) for x in Vlist]
sage: f = lambda x: FF(flist[Vlist.index(x)])
sage: hadamard_transform(f,V(0))
14*e^(4/3*I*pi) + 2*e^(2/3*I*pi) + 11
sage: CC(hadamard_transform(f,V(0)))
2.99999999999999 - 10.3923048454133*I
\end{Verbatim}
}}

\subsection{$GF(5)^2\to GF(5)$}

Using Sage, we give examples of bent functions of
$2$ variables over $GF(5)$ and study their signatures (\ref{eqn:support-i}).

\begin{proposition}
There are $1420$ even bent functions
$f:GF(5)^2\to GF(5)$ such that $f(0)=0$. The group $G=GL(2,GF(5))$
acts on the set ${\mathbb{B}}$ of all such bent functions and 
there are $11$ orbits in ${\mathbb{B}}/G$:

\[
{\mathbb{B}}/G = B_1\cup B_2\cup B_3\cup B_4\cup B_5\cup 
B_6\cup B_7\cup B_8\cup B_9\cup B_{10}\cup B_{11},
\]
where $|B_1|=40$,  $|B_2|=60$, $|B_3|=\dots =|B_9|=120$,
and $|B_{10}| =|B_{11}|=240$.
The bent functions which give rise to a weighted PDS\footnote{Note, 
the weighted PDSs are given in the examples below.} are
$f_1$, $f_2$, $f_5$, $f_6$, $f_9$ in Table \ref{tab:tab1}. The 
other $f_i$'s do not.
\end{proposition}

\begin{remark}
The result above agrees with the results of Pott et al \cite{PTFL},
where they overlap.

\end{remark}

\begin{table}
\centering
\setlength{\extrarowheight}{3pt}
\begin{tabular}{|l|l|c|}
\hline 
$B_1$ & $f_1(x_0,x_1)=-x_0^2 + 2x_1^2$ & weakly regular\\
 \hline
$B_2$ & $f_2(x_0,x_1)=-x_0x_1 + x_1^2$ & regular\\
\hline
$B_3$ & $f_3(x_0,x_1)=-2x_0^4 + 2x_0^2 + 2x_0x_1$  & regular\\
\hline
$B_4$ & $f_4(x_0,x_1)=-x_1^4 + x_0x_1 - 2x_1^2$ & regular \\
\hline
$B_5$ & $f_5(x_0,x_1)= x_0^3x_1 + 2x_1^4$ & regular \\
\hline
$B_6$ & $f_6(x_0,x_1)=-x_0x_1^3 + x_1^4$  & regular\\
\hline
$B_7$ & $f_7(x_0,x_1)=x_1^4 - x_0x_1$ & regular \\
\hline
$B_8$ & $f_8(x_0,x_1)= 2x_1^4 - 2x_0x_1 + 2x_1^2$ & regular \\
\hline
$B_9$ & $f_9(x_0,x_1)= -x_0^3x_1 + x_1^4$  & regular\\
\hline
$B_{10}$ & $f_{10}(x_0,x_1)=2x_0x_1^3 + x_1^4 - x_1^2$ & regular \\
\hline
$B_{11}$ & $f_{11}(x_0,x_1)=x_0x_1^3 - x_1^4 - 2x_1^2$  & regular\\
\hline
\end{tabular}
\caption{Representatives of orbits in $\mathbb B/G$}
\label{tab:tab1}
\end{table}

\begin{example}
\label{example:x02+x0x1}
{\rm

Consider the example of the even function $f:GF(5)^2\to GF(5)$  given 
in \S \ref{sec:bent-gf5**2} below.

\vskip .15in

\begin{Verbatim}[fontsize=\scriptsize,fontfamily=courier,fontshape=tt,frame=single,label=\sage]

sage: FF = GF(5)
sage: V = FF^2
sage: Vlist = V.list()
sage: R.<x0,x1> = PolynomialRing(FF,2,"x0,x1")
sage: ff = x0^2+x0*x1
sage: flist = [ff(x0=v[0],x1=v[1]) for v in V]
sage: f = lambda x: FF(flist[Vlist.index(x)])
sage: Gamma = boolean_cayley_graph(f, V)
sage: Gamma.connected_components_number()
1
sage: Gamma.spectrum()
[16, 1, 1, 1, 1, 1, 1, 1, 1, 1, 1, 1, 1, 1, 1, 1, 1, -4, -4, -4, -4, -4, -4, -4, -4]

\end{Verbatim}
\vskip .1in

This $f$ is homogeneous, bent and regular (hence also weakly regular).
Its edge-weighted Cayley graph $\Gamma = (V,E)$ has
weighted adjacency matrix given by
%sage: Alist = [[ZZ(f(x-y)) for x in V] for y in V]
%sage: A = matrix(GF(5), Alist)

{\scriptsize{
\[
\left(\begin{array}{rrrrrrrrrrrrrrrrrrrrrrrrr}
0 & 1 & 4 & 4 & 1 & 0 & 2 & 1 & 2 & 0 & 0 & 3 & 3 & 0 & 4 & 0 & 4 & 0 & 3 & 3 & 0 & 0 & 2 & 1 & 2 \\
1 & 0 & 1 & 4 & 4 & 0 & 0 & 2 & 1 & 2 & 4 & 0 & 3 & 3 & 0 & 3 & 0 & 4 & 0 & 3 & 2 & 0 & 0 & 2 & 1 \\
4 & 1 & 0 & 1 & 4 & 2 & 0 & 0 & 2 & 1 & 0 & 4 & 0 & 3 & 3 & 3 & 3 & 0 & 4 & 0 & 1 & 2 & 0 & 0 & 2 \\
4 & 4 & 1 & 0 & 1 & 1 & 2 & 0 & 0 & 2 & 3 & 0 & 4 & 0 & 3 & 0 & 3 & 3 & 0 & 4 & 2 & 1 & 2 & 0 & 0 \\
1 & 4 & 4 & 1 & 0 & 2 & 1 & 2 & 0 & 0 & 3 & 3 & 0 & 4 & 0 & 4 & 0 & 3 & 3 & 0 & 0 & 2 & 1 & 2 & 0 \\
0 & 0 & 2 & 1 & 2 & 0 & 1 & 4 & 4 & 1 & 0 & 2 & 1 & 2 & 0 & 0 & 3 & 3 & 0 & 4 & 0 & 4 & 0 & 3 & 3 \\
2 & 0 & 0 & 2 & 1 & 1 & 0 & 1 & 4 & 4 & 0 & 0 & 2 & 1 & 2 & 4 & 0 & 3 & 3 & 0 & 3 & 0 & 4 & 0 & 3 \\
1 & 2 & 0 & 0 & 2 & 4 & 1 & 0 & 1 & 4 & 2 & 0 & 0 & 2 & 1 & 0 & 4 & 0 & 3 & 3 & 3 & 3 & 0 & 4 & 0 \\
2 & 1 & 2 & 0 & 0 & 4 & 4 & 1 & 0 & 1 & 1 & 2 & 0 & 0 & 2 & 3 & 0 & 4 & 0 & 3 & 0 & 3 & 3 & 0 & 4 \\
0 & 2 & 1 & 2 & 0 & 1 & 4 & 4 & 1 & 0 & 2 & 1 & 2 & 0 & 0 & 3 & 3 & 0 & 4 & 0 & 4 & 0 & 3 & 3 & 0 \\
0 & 4 & 0 & 3 & 3 & 0 & 0 & 2 & 1 & 2 & 0 & 1 & 4 & 4 & 1 & 0 & 2 & 1 & 2 & 0 & 0 & 3 & 3 & 0 & 4 \\
3 & 0 & 4 & 0 & 3 & 2 & 0 & 0 & 2 & 1 & 1 & 0 & 1 & 4 & 4 & 0 & 0 & 2 & 1 & 2 & 4 & 0 & 3 & 3 & 0 \\
3 & 3 & 0 & 4 & 0 & 1 & 2 & 0 & 0 & 2 & 4 & 1 & 0 & 1 & 4 & 2 & 0 & 0 & 2 & 1 & 0 & 4 & 0 & 3 & 3 \\
0 & 3 & 3 & 0 & 4 & 2 & 1 & 2 & 0 & 0 & 4 & 4 & 1 & 0 & 1 & 1 & 2 & 0 & 0 & 2 & 3 & 0 & 4 & 0 & 3 \\
4 & 0 & 3 & 3 & 0 & 0 & 2 & 1 & 2 & 0 & 1 & 4 & 4 & 1 & 0 & 2 & 1 & 2 & 0 & 0 & 3 & 3 & 0 & 4 & 0 \\
0 & 3 & 3 & 0 & 4 & 0 & 4 & 0 & 3 & 3 & 0 & 0 & 2 & 1 & 2 & 0 & 1 & 4 & 4 & 1 & 0 & 2 & 1 & 2 & 0 \\
4 & 0 & 3 & 3 & 0 & 3 & 0 & 4 & 0 & 3 & 2 & 0 & 0 & 2 & 1 & 1 & 0 & 1 & 4 & 4 & 0 & 0 & 2 & 1 & 2 \\
0 & 4 & 0 & 3 & 3 & 3 & 3 & 0 & 4 & 0 & 1 & 2 & 0 & 0 & 2 & 4 & 1 & 0 & 1 & 4 & 2 & 0 & 0 & 2 & 1 \\
3 & 0 & 4 & 0 & 3 & 0 & 3 & 3 & 0 & 4 & 2 & 1 & 2 & 0 & 0 & 4 & 4 & 1 & 0 & 1 & 1 & 2 & 0 & 0 & 2 \\
3 & 3 & 0 & 4 & 0 & 4 & 0 & 3 & 3 & 0 & 0 & 2 & 1 & 2 & 0 & 1 & 4 & 4 & 1 & 0 & 2 & 1 & 2 & 0 & 0 \\
0 & 2 & 1 & 2 & 0 & 0 & 3 & 3 & 0 & 4 & 0 & 4 & 0 & 3 & 3 & 0 & 0 & 2 & 1 & 2 & 0 & 1 & 4 & 4 & 1 \\
0 & 0 & 2 & 1 & 2 & 4 & 0 & 3 & 3 & 0 & 3 & 0 & 4 & 0 & 3 & 2 & 0 & 0 & 2 & 1 & 1 & 0 & 1 & 4 & 4 \\
2 & 0 & 0 & 2 & 1 & 0 & 4 & 0 & 3 & 3 & 3 & 3 & 0 & 4 & 0 & 1 & 2 & 0 & 0 & 2 & 4 & 1 & 0 & 1 & 4 \\
1 & 2 & 0 & 0 & 2 & 3 & 0 & 4 & 0 & 3 & 0 & 3 & 3 & 0 & 4 & 2 & 1 & 2 & 0 & 0 & 4 & 4 & 1 & 0 & 1 \\
2 & 1 & 2 & 0 & 0 & 3 & 3 & 0 & 4 & 0 & 4 & 0 & 3 & 3 & 0 & 0 & 2 & 1 & 2 & 0 & 1 & 4 & 4 & 1 & 0
\end{array}\right)
\]
}}

Using Sage, we have 

\[
\mu_{(1,1)}=0,\ k_{(1,1)}=4,\ \lambda_{(1,1,1)}=0,\ \lambda_{(1,1,2)}=2, \ \lambda_{(1,1,3)}=0,\ \lambda_{(1,1,4)}=1,
\]
\[
\mu_{(1,2)}=1,\ k_{(1,2)}=0,\ \lambda_{(1,2,1)}=2,\ \lambda_{(1,2,2)}=0, \ \lambda_{(1,2,3)}=0,\ \lambda_{(1,2,4)}=0,
\]
\[
\mu_{(1,3)}= 1,\ k_{(1,3)}= 0,\ \lambda_{(1,3,1)}= 0,\ \lambda_{(1,3,2)}= 0, \ \lambda_{(1,3,3)}= 2,\ \lambda_{(1,3,4)}= 0,
\]
\[
\mu_{(1,4)}= 1,\ k_{(1,4)}= 0,\ \lambda_{(1,4,1)}= 1,\ \lambda_{(1,4,2)}= 0, \ \lambda_{(1,4,3)}= 0,\ \lambda_{(1,4,4)}= 1,
\]
\[
\mu_{(2,2)}=0,\ k_{(2,2)}=4,\ \lambda_{(2,2,1)}=0,\ \lambda_{(2,2,2)}=0, \ \lambda_{(2,2,3)}=1,\ \lambda_{(2,2,4)}=2,
\]
\[
\mu_{(2,3)}= 1,\ k_{(2,3)}= 0,\ \lambda_{(2,3,1)}= 0,\ \lambda_{(2,3,2)}= 1, \ \lambda_{(2,3,3)}= 1,\ \lambda_{(2,3,4)}= 0,
\]
\[
\mu_{(2,4)}= 1,\ k_{(2,4)}= 0,\ \lambda_{(2,4,1)}= 0,\ \lambda_{(2,4,2)}= 2, \ \lambda_{(2,4,3)}= 0,\ \lambda_{(2,4,4)}= 0,
\]
\[
\mu_{(3,3)}= 0,\ k_{(3,3)}= 4,\ \lambda_{(3,3,1)}= 2,\ \lambda_{(3,3,2)}= 1, \ \lambda_{(3,3,3)}= 0,\ \lambda_{(3,3,4)}= 0,
\]
\[
\mu_{(3,4)}= 1,\ k_{(3,4)}= 0,\ \lambda_{(3,4,1)}= 0,\ \lambda_{(3,4,2)}= 0, \ \lambda_{(3,4,3)}= 0,\ \lambda_{(3,4,4)}= 2,
\]
\[
\mu_{(4,4)}= 0,\ k_{(4,4)}= 4,\ \lambda_{(4,4,1)}= 1,\ \lambda_{(4,4,2)}= 0, \ \lambda_{(4,4,3)}= 2,\ \lambda_{(4,4,4)}= 0\, .
\]

In this example, Analog \ref{conj:weighted-cayley} is true.

The parameters as an unweighted strongly regular graph are
$(25,16,9,12)$.
}
\end{example}

\vskip .15in

\begin{example}
{\rm

Consider the example of the even function $f:GF(5)^2\to GF(5)$  given 
by 

\[
f(x_0,x_1)= x_0^4+2x_0x_1.
\]
This is non-homogeneous, but bent and regular.

\vskip .15in

\begin{Verbatim}[fontsize=\scriptsize,fontfamily=courier,fontshape=tt,frame=single,label=\sage]

sage: p = 5; n = 2
sage: FF = GF(p)
sage: V = GF(p)**n
sage: Vlist = V.list()
sage: f = lambda x: FF(x[0]^4+2*x[0]*x[1])
sage: [CC(hadamard_transform(f,a)).abs() for a in V]
[5.00000000000000, 5.00000000000000, 5.00000000000000, 5.00000000000000,
 5.00000000000000, 5.00000000000000, 5.00000000000000, 5.00000000000000,
 5.00000000000000, 5.00000000000000, 5.00000000000000, 5.00000000000000,
 5.00000000000000, 5.00000000000000, 5.00000000000000, 5.00000000000000,
 5.00000000000000, 5.00000000000000, 5.00000000000000, 5.00000000000000,
 5.00000000000000, 5.00000000000000, 5.00000000000000, 5.00000000000000,
 5.00000000000000]
sage: Gamma = boolean_cayley_graph(f, V)
sage: Gamma.spectrum()
[16, 3.236067977499790?, 3.236067977499790?, 3.236067977499790?, 3.236067977499790?,
 1, 1, 1, 1, -0.3819660112501051?, -0.3819660112501051?, -0.3819660112501051?,
 -0.3819660112501051?, -1.236067977499790?, -1.236067977499790?, -1.236067977499790?,
 -1.236067977499790?, -2.618033988749895?, -2.618033988749895?, -2.618033988749895?,
 -2.618033988749895?, -4, -4, -4, -4]
sage: Gamma.is_strongly_regular()
False

\end{Verbatim}
\vskip .1in

Its edge-weighted Cayley graph $\Gamma = (V,E)$ has
weighted adjacency matrix given by

{\scriptsize{
\[
\left(\begin{array}{rrrrrrrrrrrrrrrrrrrrrrrrr}
0 & 1 & 1 & 1 & 1 & 0 & 3 & 0 & 2 & 4 & 0 & 0 & 4 & 3 & 2 & 0 & 2 & 3 & 4 & 0 & 0 & 4 & 2 & 0 & 3 \\
1 & 0 & 1 & 1 & 1 & 4 & 0 & 3 & 0 & 2 & 2 & 0 & 0 & 4 & 3 & 0 & 0 & 2 & 3 & 4 & 3 & 0 & 4 & 2 & 0 \\
1 & 1 & 0 & 1 & 1 & 2 & 4 & 0 & 3 & 0 & 3 & 2 & 0 & 0 & 4 & 4 & 0 & 0 & 2 & 3 & 0 & 3 & 0 & 4 & 2 \\
1 & 1 & 1 & 0 & 1 & 0 & 2 & 4 & 0 & 3 & 4 & 3 & 2 & 0 & 0 & 3 & 4 & 0 & 0 & 2 & 2 & 0 & 3 & 0 & 4 \\
1 & 1 & 1 & 1 & 0 & 3 & 0 & 2 & 4 & 0 & 0 & 4 & 3 & 2 & 0 & 2 & 3 & 4 & 0 & 0 & 4 & 2 & 0 & 3 & 0 \\
0 & 4 & 2 & 0 & 3 & 0 & 1 & 1 & 1 & 1 & 0 & 3 & 0 & 2 & 4 & 0 & 0 & 4 & 3 & 2 & 0 & 2 & 3 & 4 & 0 \\
3 & 0 & 4 & 2 & 0 & 1 & 0 & 1 & 1 & 1 & 4 & 0 & 3 & 0 & 2 & 2 & 0 & 0 & 4 & 3 & 0 & 0 & 2 & 3 & 4 \\
0 & 3 & 0 & 4 & 2 & 1 & 1 & 0 & 1 & 1 & 2 & 4 & 0 & 3 & 0 & 3 & 2 & 0 & 0 & 4 & 4 & 0 & 0 & 2 & 3 \\
2 & 0 & 3 & 0 & 4 & 1 & 1 & 1 & 0 & 1 & 0 & 2 & 4 & 0 & 3 & 4 & 3 & 2 & 0 & 0 & 3 & 4 & 0 & 0 & 2 \\
4 & 2 & 0 & 3 & 0 & 1 & 1 & 1 & 1 & 0 & 3 & 0 & 2 & 4 & 0 & 0 & 4 & 3 & 2 & 0 & 2 & 3 & 4 & 0 & 0 \\
0 & 2 & 3 & 4 & 0 & 0 & 4 & 2 & 0 & 3 & 0 & 1 & 1 & 1 & 1 & 0 & 3 & 0 & 2 & 4 & 0 & 0 & 4 & 3 & 2 \\
0 & 0 & 2 & 3 & 4 & 3 & 0 & 4 & 2 & 0 & 1 & 0 & 1 & 1 & 1 & 4 & 0 & 3 & 0 & 2 & 2 & 0 & 0 & 4 & 3 \\
4 & 0 & 0 & 2 & 3 & 0 & 3 & 0 & 4 & 2 & 1 & 1 & 0 & 1 & 1 & 2 & 4 & 0 & 3 & 0 & 3 & 2 & 0 & 0 & 4 \\
3 & 4 & 0 & 0 & 2 & 2 & 0 & 3 & 0 & 4 & 1 & 1 & 1 & 0 & 1 & 0 & 2 & 4 & 0 & 3 & 4 & 3 & 2 & 0 & 0 \\
2 & 3 & 4 & 0 & 0 & 4 & 2 & 0 & 3 & 0 & 1 & 1 & 1 & 1 & 0 & 3 & 0 & 2 & 4 & 0 & 0 & 4 & 3 & 2 & 0 \\
0 & 0 & 4 & 3 & 2 & 0 & 2 & 3 & 4 & 0 & 0 & 4 & 2 & 0 & 3 & 0 & 1 & 1 & 1 & 1 & 0 & 3 & 0 & 2 & 4 \\
2 & 0 & 0 & 4 & 3 & 0 & 0 & 2 & 3 & 4 & 3 & 0 & 4 & 2 & 0 & 1 & 0 & 1 & 1 & 1 & 4 & 0 & 3 & 0 & 2 \\
3 & 2 & 0 & 0 & 4 & 4 & 0 & 0 & 2 & 3 & 0 & 3 & 0 & 4 & 2 & 1 & 1 & 0 & 1 & 1 & 2 & 4 & 0 & 3 & 0 \\
4 & 3 & 2 & 0 & 0 & 3 & 4 & 0 & 0 & 2 & 2 & 0 & 3 & 0 & 4 & 1 & 1 & 1 & 0 & 1 & 0 & 2 & 4 & 0 & 3 \\
0 & 4 & 3 & 2 & 0 & 2 & 3 & 4 & 0 & 0 & 4 & 2 & 0 & 3 & 0 & 1 & 1 & 1 & 1 & 0 & 3 & 0 & 2 & 4 & 0 \\
0 & 3 & 0 & 2 & 4 & 0 & 0 & 4 & 3 & 2 & 0 & 2 & 3 & 4 & 0 & 0 & 4 & 2 & 0 & 3 & 0 & 1 & 1 & 1 & 1 \\
4 & 0 & 3 & 0 & 2 & 2 & 0 & 0 & 4 & 3 & 0 & 0 & 2 & 3 & 4 & 3 & 0 & 4 & 2 & 0 & 1 & 0 & 1 & 1 & 1 \\
2 & 4 & 0 & 3 & 0 & 3 & 2 & 0 & 0 & 4 & 4 & 0 & 0 & 2 & 3 & 0 & 3 & 0 & 4 & 2 & 1 & 1 & 0 & 1 & 1 \\
0 & 2 & 4 & 0 & 3 & 4 & 3 & 2 & 0 & 0 & 3 & 4 & 0 & 0 & 2 & 2 & 0 & 3 & 0 & 4 & 1 & 1 & 1 & 0 & 1 \\
3 & 0 & 2 & 4 & 0 & 0 & 4 & 3 & 2 & 0 & 2 & 3 & 4 & 0 & 0 & 4 & 2 & 0 & 3 & 0 & 1 & 1 & 1 & 1 & 0
\end{array}\right)
\]
}}

\vskip .15in

Using Sage, we have 

\[
\mu_{(1,1)}=0,\ k_{(1,1)}=4,\ \lambda_{(1,1,1)}=3,\ \lambda_{(1,1,2)}=0, \ \lambda_{(1,1,3)}=0,\ \lambda_{(1,1,4)}=0,
\]
\[
\mu_{(1,2)}=1,\ k_{(1,2)}=0,\ \lambda_{(1,2,1)}=0,\ \lambda_{(1,2,2)}=0, \ \lambda_{(1,2,3)}=1,\ \lambda_{(1,2,4)}=1,
\]
\[
\mu_{(1,3)}= 1,\ k_{(1,3)}= 0,\ \lambda_{(1,3,1)}= 0,\ \lambda_{(1,3,2)}= 1, \ \lambda_{(1,3,3)}= 0,\ \lambda_{(1,3,4)}= 1,
\]
\[
\mu_{(1,4)}= 1,\ k_{(1,4)}= 0,\ \lambda_{(1,4,1)}= 0,\ \lambda_{(1,4,2)}= 1, \ \lambda_{(1,4,3)}= 1,\ \lambda_{(1,4,4)}= 0,
\]
\[
\mu_{(2,2)}=\{0,1\},\ k_{(2,2)}=4,\ \lambda_{(2,2,1)}=0,\ \lambda_{(2,2,2)}=0, \ \lambda_{(2,2,3)}=2,\ \lambda_{(2,2,4)}=0,
\]
\[
\mu_{(2,3)}= \{0,1\},\ k_{(2,3)}= 0,\ \lambda_{(2,3,1)}= 1,\ \lambda_{(2,3,2)}= 2, \ \lambda_{(2,3,3)}= 0,\ \lambda_{(2,3,4)}= 0,
\]
\[
\mu_{(2,4)}= \{0,1\},\ k_{(2,4)}= 0,\ \lambda_{(2,4,1)}= 1,\ \lambda_{(2,4,2)}= 0, \ \lambda_{(2,4,3)}= 0,\ \lambda_{(2,4,4)}= 2,
\]
\[
\mu_{(3,3)}= \{0,2\},\ k_{(3,3)}= 4,\ \lambda_{(3,3,1)}= 0,\ \lambda_{(3,3,2)}= 0, \ \lambda_{(3,3,3)}= 0,\ \lambda_{(3,3,4)}= 1,
\]
\[
\mu_{(3,4)}= \{0,1\},\ k_{(3,4)}= 0,\ \lambda_{(3,4,1)}= 1,\ \lambda_{(3,4,2)}= 0, \ \lambda_{(3,4,3)}= 1,\ \lambda_{(3,4,4)}= 1,
\]
\[
\mu_{(4,4)}= 0,\ k_{(4,4)}= 4,\ \lambda_{(4,4,1)}= 0,\ \lambda_{(4,4,2)}= 2, \ \lambda_{(4,4,3)}= 1,\ \lambda_{(4,4,4)}= 0 \, .
\]

In this example, Analog  \ref{conj:weighted-cayley} is false.
}
\end{example}

The number of even (polynomial) functions $f$ of degree
less than or equal to $4$ is $5^8=390625$. The number of such 
functions having signature $(4,4,4,4)$ is $10740$ and 
the number of such functions having signature $(6,6,6,6)$ is $2920$.

Using Sage, we discovered there are $1420$ even bent functions
$f:GF(5)^2\to GF(5)$ such that $f(0)=0$. The group $G=GL(2,GF(5))$
acts on the set ${\mathbb{B}}$ of all such bent functions and 
there are $11$ orbits in ${\mathbb{B}}/G$:

\[
{\mathbb{B}}/G = B_1\cup B_2\cup B_3\cup B_4\cup B_5\cup 
B_6\cup B_7\cup B_8\cup B_9\cup B_{10}\cup B_{11},
\]
where $|B_1|=40$,  $|B_2|=60$, $|B_3|=\dots =|B_9|=120$,
and $|B_{10}| =|B_{11}|=240$.

A representative of $B_1$ is

\[
f_1(x_0,x_1) =  -x_0^2 + 2x_1^2.
\]
A representative of $B_2$ is

\[
f_2(x_0,x_1) =  -x_0x_1 + x_1^2.
\]
A representative of $B_3$ is

\[
f_3(x_0,x_1) = -2x_0^4 + 2x_0^2 + 2x_0x_1.
\]
A representative of $B_4$ is

\[
f_4(x_0,x_1) = -x_1^4 + x_0x_1 - 2x_1^2.
\]
A representative of $B_5$ is

\[
f_5(x_0,x_1) =  x_0^3x_1 + 2x_1^4.
\]
A representative of $B_6$ is

\[
f_6(x_0,x_1) = -x_0x_1^3 + x_1^4
\]
A representative of $B_7$ is

\[
f_7(x_0,x_1) = x_1^4 - x_0x_1
\]
A representative of $B_8$ is

\[
f_8(x_0,x_1) =   2x_1^4 - 2x_0x_1 + 2x_1^2.
\]
A representative of $B_9$ is

\[
f_9(x_0,x_1) =  -x_0^3x_1 + x_1^4
\]
A representative of $B_{10}$ is

\[
f_{10}(x_0,x_1) =  2x_0x_1^3 + x_1^4 - x_1^2.
\]
A representative of $B_{11}$ is

\[
f_{11}(x_0,x_1) =  x_0x_1^3 - x_1^4 - 2x_1^2.
\]

These $11$ bent functions form a complete set of
representatives of the $G$-equivalence classes
of ${\mathbb{B}}$. We write $f\sim g$ if and only if
$f=g\circ\phi$, for some $\phi\in G$.
The group $GF(5)^\times$ also acts on ${\mathbb{B}}$.

\begin{itemize}
\item
for $i\in \{1,2,6\}$,
the functions $af_i$, for $a \in GF(5)^\times$, are all $G$-equivalent,
\item
$f_3\sim 2f_4\sim 3f_7\sim 4f_8$,
\item
$f_4\sim 3f_3\sim 4f_7\sim 2f_8$,
\item
$f_5\sim 4f_5\sim 2f_9\sim 3f_9$,
\item
$f_7\sim 2f_3\sim 4f_4\sim 3f_8$,
\item
$f_8\sim 4f_3\sim 3f_4\sim 2f_7$,
\item
$f_9\sim 2f_5\sim 3f_5\sim 4f_9$,
\item
$f_{10}\sim 4f_{10}\sim 2f_{11}\sim 3f_{11}$,
\item
$f_{11}\sim 2f_{10}\sim 3f_{10}\sim 4f_{11}$.
\end{itemize}
It follows that $f_3$, $f_4$, $f_7$ and $f_8$ all must have the same
signature. Similarly,  $f_5$ and $f_9$ must have the same
signature, and $f_{10}$ and $f_{11}$ must have the same
signature.

Note $f_5$ and $f_6$ are not $GL(2,GF(5))$-equivalent but 
they both corresponding to weighted PDSs with the same intersection numbers.
In particular, the adjacency ring corresponding to $f_5$ is
isomorphic to the adjacency ring corresponding to $f_6$.

\begin{example}
{\rm
The example of $f_1$ above can be used to construct an edge-weighted 
strongly regular Cayley graph, hence also a weighted PDS attached to its
level curves.

Define the level curve $D_i$ ($i=1,2,3,4$) as above, the let
$D_0=\{0\}$ and $D_5=GF(5)^2\setminus \cup_{i=0}^4 D_i$. We can interprete
$p_{ij}^k$ to be the number of times each element of $D_k$ occurs in
$D_jD_i^{-1}$. By computing these numbers directly using Sage, we obtain
the intersection numbers $p_{ij}^k$:

\[
\begin{array}{cc}
\begin{array}{c|cccccc}
p_{ij}^0 & 0 & 1 & 2 & 3 & 4 & 5\\ \hline
0         & 1 & 0 & 0 & 0 & 0 & 0 \\
1         & 0 & 6 & 0 & 0  & 0 & 0\\
2         & 0 & 0 & 6 & 0  & 0 & 0\\
3         & 0 & 0 & 0 & 6 & 0 & 0\\
4         & 0 & 0 & 0 & 0 & 6 & 0\\
5         & 0 & 0 & 0 & 0 & 0 & 0\\
\end{array}
 &
\begin{array}{c|cccccc}
p_{ij}^1 & 0 & 1 & 2 & 3 & 4 & 5\\ \hline
0         & 0 & 1 & 0 & 0 & 0 & 0 \\
1         & 1 & 2 & 0 & 2 & 1 & 0 \\
2         & 0 & 0 & 2 & 2 & 2 & 0 \\
3         & 0 & 2 & 2 & 0 & 2 & 0 \\
4         & 0 & 1 & 2 & 2 & 1 & 0 \\
5         & 0 & 0 & 0 & 0 & 0 & 0 \\
\end{array} \\
 & \\
\begin{array}{c|cccccc}
p_{ij}^2 & 0 & 1 & 2 & 3 & 4 & 5\\ \hline
0         & 0 & 0 & 1 & 0 & 0 & 0 \\
1         & 0 & 0 & 2 & 2  & 2 & 0\\
2         & 1 & 2 & 2 & 1  & 0 & 0\\
3         & 0 & 2 & 1 & 1 & 0 & 0\\
4         & 0 & 2 & 0 & 2 & 2 & 0\\
5         & 0 & 0 & 0 & 0 & 0 & 0\\
\end{array}
 &
\begin{array}{c|cccccc}
p_{ij}^3   & 0 & 1 & 2 & 3 & 4 & 5\\ \hline
0         & 0 & 0 & 0 & 1 & 0 & 0 \\
1         & 0 & 2 & 2 & 0 & 2 & 0\\
2         & 0 & 2 & 1 & 1 & 2 & 0\\
3         & 1 & 0 & 1 & 2 & 2 & 0\\
4         & 0 & 2 & 2 & 2 & 0 & 0\\
5         & 0 & 0 & 0 & 0 & 0 & 0\\
\end{array} \\
 & \\
\begin{array}{c|cccccc}
p_{ij}^4 & 0 & 1 & 2 & 3 & 4 & 5\\ \hline
0         & 0 & 0 & 0 & 0 & 1 & 0 \\
1         & 0 & 1 & 2 & 2  & 1 & 0\\
2         & 0 & 2 & 0 & 2  & 2 & 0\\
3         & 0 & 2 & 2 & 2 & 0 & 0\\
4         & 1 & 1 & 2 & 0 & 2 & 0\\
5         & 0 & 0 & 0 & 0 & 0 & 0\\
\end{array}
 &
\begin{array}{c|cccccc}
p_{ij}^5 & 0 & 1 & 2 & 3 & 4 & 5\\ \hline
0         & 0 & 0 & 0 & 0 & 0 & 0 \\
1         & 0 & 0 & 0 & 0  & 0 & 0\\
2         & 0 & 0 & 0 & 0  & 0 & 0\\
3         & 0 & 0 & 0 & 0 & 0 & 0\\
4         & 0 & 0 & 0 & 0 & 0 & 0\\
5         & 0 & 0 & 0 & 0 & 0 & 0\\
\end{array} \\
\end{array}
\]

}
\end{example}

\begin{example}
{\rm
The example of $f_2$ above can be used to construct an edge-weighted 
strongly regular Cayley graph, hence also a weighted PDS attached to its
level curves.

Define the level curve $D_i$ ($i=1,2,3,4$) as above, the let
$D_0=\{0\}$ and $D_5=GF(5)^2\setminus \cup_{i=0}^4 D_i$. We can interprete
$p_{ij}^k$ to be the number of times each element of $D_k$ occurs in
$D_jD_i^{-1}$. By computing these numbers directly using Sage, we obtain
the intersection numbers $p_{ij}^k$:

\[
\begin{array}{cc}
\begin{array}{c|cccccc}
p_{ij}^0 & 0 & 1 & 2 & 3 & 4 & 5\\ \hline
0         & 1 & 0 & 0 & 0 & 0 & 0 \\
1         & 0 & 4 & 0 & 0  & 0 & 0\\
2         & 0 & 0 & 4 & 0  & 0 & 0\\
3         & 0 & 0 & 0 & 4 & 0 & 0\\
4         & 0 & 0 & 0 & 0 & 4 & 0\\
5         & 0 & 0 & 0 & 0 & 0 & 8\\
\end{array}
 &
\begin{array}{c|cccccc}
p_{ij}^1 & 0 & 1 & 2 & 3 & 4 & 5\\ \hline
0         & 0 & 1 & 0 & 0 & 0 & 0 \\
1         & 1 & 0 & 2 & 0 & 1 & 0 \\
2         & 0 & 2 & 0 & 0 & 0 & 2 \\
3         & 0 & 0 & 0 & 2 & 0 & 2 \\
4         & 0 & 1 & 0 & 0 & 0 & 3 \\
5         & 0 & 0 & 2 & 2 & 3 & 1 \\
\end{array} \\
& \\
\begin{array}{c|cccccc}
p_{ij}^2 & 0 & 1 & 2 & 3 & 4 & 5\\ \hline
0         & 0 & 0 & 1 & 0 & 0 & 0 \\
1         & 0 & 2 & 0 & 0 & 0 & 2\\
2         & 1 & 0 & 0 & 1  & 2 & 0\\
3         & 0 & 0 & 1 & 1 & 0 & 2\\
4         & 0 & 0 & 2 & 0 & 0 & 2\\
5         & 0 & 2 & 0 & 2 & 2 & 2\\
\end{array}
 &
\begin{array}{c|cccccc}
p_{ij}^3   & 0 & 1 & 2 & 3 & 4 & 5\\ \hline
0         & 0 & 0 & 0 & 1 & 0 & 0 \\
1         & 0 & 0 & 0 & 2 & 0 & 2\\
2         & 0 & 0 & 1 & 1 & 0 & 2\\
3         & 1 & 2 & 1 & 0 & 0 & 0\\
4         & 0 & 0 & 0 & 0 & 2 & 2\\
5         & 0 & 2 & 2 & 0 & 2 & 2\\
\end{array} \\
& \\
\begin{array}{c|cccccc}
p_{ij}^4 & 0 & 1 & 2 & 3 & 4 & 5\\ \hline
0         & 0 & 0 & 0 & 0 & 1 & 0 \\
1         & 0 & 1 & 0 & 0 & 1 & 2\\
2         & 0 & 0 & 2 & 0 & 0 & 2\\
3         & 0 & 0 & 0 & 0 & 2 & 2\\
4         & 1 & 1 & 0 & 2 & 0 & 0\\
5         & 0 & 2 & 2 & 2 & 0 & 2\\
\end{array}
 &
\begin{array}{c|cccccc}
p_{ij}^5 & 0 & 1 & 2 & 3 & 4 & 5\\ \hline
0         & 0 & 0 & 0 & 0 & 0 & 1 \\
1         & 0 & 0 & 1 & 1 & 1 & 1\\
2         & 0 & 1 & 0 & 1 & 1 & 1\\
3         & 0 & 1 & 1 & 0 & 1 &  1\\
4         & 0 & 1 & 1 & 1 & 0 & 1 \\
5         & 1 & 1 & 1 & 1 & 1 & 3\\
\end{array} \\
\end{array}
\]

}
\end{example}

\begin{example}
{\rm
The level curves of $f_3$ above do not give rise to a weighted PDS.

On the other hand, we can define the adjacency matrix $A_i$ attached 
to the level curve $D_i$ ($i=1,2,3,4$) as the $25\times 25$ matrix
obtained by taking the weighted adjacency matrix $A$ of
the corresponding Cayley graph and putting a $1$ in every entry where
the corresponding entry of $A$ is equal to $i$, and a $0$ otherwise.
In this case,

\[
D_1 = \{(2, 0), (3, 0), (4, 2), (1, 3)\},
\]
\[
D_2 = \{ (1, 1), (3, 1), (2, 4), (4, 4)\},
\]
\[
D_3 = \{ (4, 1), (3, 2), (2, 3), (1, 4)\},
\]
\[
D_4 = \{ (1, 2), (2, 2), (3, 3), (4, 3)\}.
\]
The ``adjacency matrix'' $A_0$ is the $25\times 25$ identity matrix
and the ``adjacency matrix'' $A_5$ is the $25\times 25$ matrix 
which has the property that $A_0+A_1+A_2+A_3+A_4+A_5$ is
the all $1$'s matrix.

If the Cayley graph {\it were} a strongly regular
edge-weighted graph then, according to \cite{CvL}, equation (17.13) (proven in Theorem
\ref{thrm:pijk-formula} above),
the intersection numbers $p_{ij}^k$ could be computed using

\[
\trace(A_iA_jA_k)=|GF(5)^2||D_k|p_{ij}^k.
\]
Using \sage, we compute

\[
\trace(A_1^3)=0,
\trace(A_1^2A_2)=0,
\trace(A_1A_2^2)=200,
\trace(A_2^3)=0,
\]
\[
\trace(A_1^2A_3)=0,
\trace(A_1A_3^2)=0,
\trace(A_3^3)=300,
\]
\[
\trace(A_1^2A_4)=200,
\trace(A_1A_4^2)=0,
\trace(A_4^3)=0,
\]
\[
\trace(A_1^2A_5)=100,
\trace(A_1A_5^2)=300,
\trace(A_5^3)=300,
\]
\[
\trace(A_2^2A_3)=0,
\trace(A_2^2A_4)=100,
\trace(A_2A_3^2)=0,
\]
\[
\trace(A_2A_4^2)=100,
\trace(A_3^2A_4)=0,
\trace(A_3A_4^2)=0,
\]
\[
\trace(A_2^2A_5)=0,
\trace(A_2A_5^2)=400,
\]
\[
\trace(A_3^2A_5)=0,
\trace(A_3A_5^2)=200,
\]
\[
\trace(A_4^2A_5)=200,
\trace(A_4A_5^2)=200,
\]
\[
\trace(A_1A_2A_3)=100,
\trace(A_1A_2A_4)=0,
\]
\[
\trace(A_1A_2A_5)=100,
\trace(A_1A_3A_5)=200,
\]
\[
\trace(A_1A_4A_5)=100,
\trace(A_2A_4A_5)=100,
\]
\[
\trace(A_2A_3A_5)=200,
\trace(A_3A_4A_5)=200,
\]
\[
\trace(A_1A_3A_4)=100,
\trace(A_2A_3A_4)=100.
\]
Using these, we can compute the $p_{ij}^k$'s.

We have $|D_1|=|D_2|=|D_3|=|D_4|=4$, and the intersection numbers $p_{ij}^k$
are given as follows:

\[
\begin{array}{cc}
\begin{array}{c|cccccc}
p_{ij}^0 & 0 & 1 & 2 & 3 & 4 & 5\\ \hline
0         & 1 & 0 & 0 & 0 & 0 & 0 \\
1         & 0 & 4 & 0 & 0  & 0 & 0\\
2         & 0 & 0 & 4 & 0  & 0 & 0\\
3         & 0 & 0 & 0 & 4 & 0 & 0\\
4         & 0 & 0 & 0 & 0 & 4 & 0\\
5         & 0 & 0 & 0 & 0 & 0 & 8\\
\end{array}
 &
\begin{array}{c|cccccc}
p_{ij}^1 & 0 & 1 & 2 & 3 & 4 & 5\\ \hline
0         & 0 & 1 & 0 & 0 & 0 & 0 \\
1         & 1 & 0 & 0 & 0  & 2 & 1\\
2         & 0 & 0 & 2 & 1  & 0 & 1\\
3         & 0 & 0 & 1 & 0 & 1 & 2\\
4         & 0 & 2 & 0 & 1 & 0 & 1\\
5         & 0 & 1 & 1 & 2 & 1 & 3\\
\end{array} \\
 & \\
\begin{array}{c|cccccc}
p_{ij}^2 & 0 & 1 & 2 & 3 & 4 & 5\\ \hline
0         & 0 & 0 & 1 & 0 & 0 & 0 \\
1         & 0 & 0 & 2 & 1  & 0 & 1\\
2         & 1 & 2 & 0 & 0  & 1 & 0\\
3         & 0 & 1 & 0 & 0 & 1 & 2\\
4         & 0 & 0 & 1 & 1 & 1 & 1\\
5         & 0 & 1 & 0 & 2 & 1 & 4\\
\end{array}
 &
\begin{array}{c|cccccc}
p_{ij}^3   & 0 & 1 & 2 & 3 & 4 & 5\\ \hline
0         & 0 & 0 & 0 & 1 & 0 & 0 \\
1         & 0 & 0 & 1 & 0  & 1 & 2\\
2         & 0 & 1 & 0 & 0  & 1 & 2\\
3         & 1 & 0 & 0 & 3 & 0 & 0\\
4         & 0 & 1 & 1 & 0 & 0 & 2\\
5         & 0 & 2 & 2 & 0 & 2 & 2\\
\end{array} \\
 & \\
\begin{array}{c|cccccc}
p_{ij}^4 & 0 & 1 & 2 & 3 & 4 & 5\\ \hline
0         & 0 & 0 & 0 & 0 & 1 & 0 \\
1         & 0 & 2 & 0 & 1  & 0 & 1\\
2         & 0 & 0 & 1 & 1  & 1 & 1\\
3         & 0 & 1 & 1 & 0 & 0 & 2\\
4         & 1 & 0 & 1 & 0 & 0 & 2\\
5         & 0 & 1 & 1 & 2 & 2 & 2\\
\end{array}
 &
\begin{array}{c|cccccc}
p_{ij}^5 & 0 & 1 & 2 & 3 & 4 & 5\\ \hline
0         & 0 & 0 & 0 & 0 & 0 & 1 \\
1         & 0 & 1/2 & 1/2 & 1  & 1/2 & 1/2\\
2         & 0 & 1/2 & 0 & 1  & 1/2 & 2\\
3         & 0 & 1 & 1 & 0 & 1 & 1\\
4         & 0 & 1/ & 1/2 & 1 & 1 & 1\\
5         & 1 & 3/2 & 2 & 1 & 1 & 3/2\\
\end{array} \\
\end{array}
\]
In other words, they are not integers, so cannot correspond to an
edge-weighted strongly regular graph.
}
\end{example}

\begin{example}
{\rm
Consider the bent function
\[
f_4(x_0,x_1)=-x_1^4-2x_1^2+x_0x_1.
\]
This function represents a $GL(2,GF(5))$ orbit of size $120$.
The level curves of this function do not give rise to a weighted PDS.
By the way, if we try a computation of all the $p_{ij}^k$'s as in the
above example, we do not get integers.
Similarly, the  level curves of $f_7$, $f_8$, $f_{10}$, and $f_{11}$ do not
give rise to a weighted PDS, since the $p_{ij}^k$'s are not always integers.

}
\end{example}

\begin{example}
{\rm
The example of $f_5$ above can be used to construct an edge-weighted 
strongly regular Cayley graph, hence also a weighted PDS attached to its
level curves.

The intersection numbers $p_{ij}^k$ are given by:

\[
\begin{array}{cc}
\begin{array}{c|cccccc}
p_{ij}^0 & 0 & 1 & 2 & 3 & 4 & 5\\ \hline
0         & 1 & 0 & 0 & 0 & 0 & 0 \\
1         & 0 & 4 & 0 & 0  & 0 & 0\\
2         & 0 & 0 & 4 & 0  & 0 & 0\\
3         & 0 & 0 & 0 & 4 & 0 & 0\\
4         & 0 & 0 & 0 & 0 & 4 & 0\\
5         & 0 & 0 & 0 & 0 & 0 & 8\\
\end{array}
 &
\begin{array}{c|cccccc}
p_{ij}^1 & 0 & 1 & 2 & 3 & 4 & 5\\ \hline
0         & 0 & 1 & 0 & 0 & 0 & 0 \\
1         &1 & 3 & 0 & 0 & 0 & 0\\
2         & 0 & 0 & 0 & 1 & 1 & 2\\
3         & 0 & 0 & 1 & 0 & 1 & 2\\
4         & 0 & 0 & 1 & 1 & 0 & 2\\
5         & 0 & 0 & 2 & 2 & 2 & 2\\
\end{array} \\
& \\
\begin{array}{c|cccccc}
p_{ij}^2 & 0 & 1 & 2 & 3 & 4 & 5\\ \hline
0         & 0 & 0 & 1 & 0 & 0 & 0 \\
1         & 0 & 0 & 0 & 1 & 1 & 2\\
2         & 1 & 0 & 3 & 0 & 0 & 0\\
3         & 0 & 1 & 0 & 0 & 1 & 2\\
4         & 0 & 1 & 0 & 1 & 0 & 2\\
5         & 0 & 2 & 0 & 2 & 2 & 2\\
\end{array}
 &
\begin{array}{c|cccccc}
p_{ij}^3   & 0 & 1 & 2 & 3 & 4 & 5\\ \hline
0         & 0 & 0 & 0 & 1 & 0 & 0 \\
1         & 0 & 0 & 1 & 0 & 1 & 2\\
2         & 0 & 1 & 0 & 0 & 1 & 2\\
3         & 1 & 0 & 0 & 3 & 0 & 0\\
4         & 0 & 1 & 1 & 0 & 0 & 2\\
5         & 0 & 2 & 2 & 0 & 2 & 2\\
\end{array} \\
& \\
\begin{array}{c|cccccc}
p_{ij}^4 & 0 & 1 & 2 & 3 & 4 & 5\\ \hline
0         & 0 & 0 & 0 & 0 & 1 & 0 \\
1         & 0 & 0 & 1 & 1 & 0 & 2\\
2         & 0 & 1 & 0 & 1 & 0 & 2\\
3         & 0 & 1 & 1 & 0 & 0 & 2\\
4         & 1 & 0 & 0 & 0 & 3 & 0\\
5         & 0 & 2 & 2 & 2 & 0 & 2\\
\end{array}
 &
\begin{array}{c|cccccc}
p_{ij}^5 & 0 & 1 & 2 & 3 & 4 & 5\\ \hline
0         & 0 & 0 & 0 & 0 & 0 & 1 \\
1         & 0 & 0 & 1 & 1 & 1 & 1\\
2         & 0 & 1 & 0 & 1 & 1 & 1\\
3         & 0 & 1 & 1 & 0 & 1 &  1\\
4         & 0 & 1 & 1 & 1 & 0 & 1 \\
5         & 1 & 1 & 1 & 1 & 1 & 3\\
\end{array} \\
\end{array}
\]
The examples of $f_6$ and $f_9$ above have the same $p_{ij}^k$'s.

Note $f_5$ and $f_6$ are not multiples. Therefore, the $p_{ij}^k$'s do not
determine the equivalence class of the bent function nor even
the (larger) equivalence class ``up to a scalar factor'' .

}
\end{example}

\section{Examples of bent functions}

\subsection{Algebraic Normal Form}

Similar to how Carlet \cite{C} shows that every Boolean function can be 
written in
algebraic normal form, we can show that each $GF(p)$-valued function over
$GF(p)^n$ can be written in algebraic normal form as well. 

An {\it atomic $p$-ary function} is a function
$GF(p)^n \rightarrow GF(p)$ supported at a single point. For 
$v\in GF(p)^n$, the atomic function supported at $v$ is
the function $f_v: GF(p)^n \rightarrow GF(p)$
such that $f_v(v)=1$ and for every $w\in GF(p)^n$ 
such that $w\not=v$ $f_v(w)=0$.
We begin by
showing how to write the algebraic normal form of the atomic $p$-ary functions,
where 

\begin{theorem}\label{thm:atomic-ANF}
  Let $f_v$ be an atomic $p$-ary function. Then
  \begin{equation}\label{eqn:atomic-anf}
  f_v(x) = \prod_{i=0}^{n-1}{\left(\frac{1}{(p-1)!}\prod_{j=1}^{p-1}{(j+v_i-x_i)}\right)}.
  \end{equation}
\end{theorem}

\begin{proof}
  First, we start by showing that $f_v(v)=1$. We can do this by plugging $v$
  directly into (\ref{eqn:atomic-anf}).
  \begin{align*}
  f_v(v) &= \prod_{i=0}^{n-1}{\left(\frac{1}{(p-1)!}\prod_{j=1}^{p-1}{(j+v_i-v_i)}\right)} \\
         &= \prod_{i=0}^{n-1}{\left(\frac{1}{(p-1)!}\prod_{j=1}^{p-1}{j}\right)} \\
         &= \prod_{i=0}^{n-1}{\left(\frac{(p-1)!}{(p-1)!}\right)} \\
         &= 1.
  \end{align*}

  Second, we show that $f_v(w)=0$ for every $w\not=v$. Let $w\not=v$. Then,
  pick $k$ such that $w_k\not=v_k$. So there exists $j \in \{1,\dots,n-1\}$ such
  that $j+v_k-w_k = 0 \pmod p$. Thus, the inside product of (\ref{eqn:atomic-anf}) is
  $0$ for $i=k$ and the whole equation is $0$. So $f_v(w) = 0$.
\end{proof}

It easily follows that every $GF(p)$-valued function over $GF(p)^n$ can be
written in algebraic normal form.

\begin{corollary}
  Let $g:GF(p)^n \rightarrow GF(p)$. Then
  \begin{equation}\label{eqn:anf}
    g(x) = \sum_{v\in GF(p)^n}{g(v)f_v(x)}.
  \end{equation}
\end{corollary}

\begin{example}
{\rm
Sage can easily list all the atomic functions
over $GF(3)$ having $2$ variables:

\begin{Verbatim}[fontsize=\scriptsize,fontfamily=courier,fontshape=tt,frame=single,label=\sage]

sage: V = GF(3)^2
sage: x0,x1 = var("x0,x1")               
sage: xx = [x0,x1]
sage: [expand(prod([2*prod([GF(3)(j)+v[i]-xx[i] for j in range(1,3)]) 
       for i in range(2)])) for v in V]
[x0^2*x1^2 + 2*x0^2 + 2*x1^2 + 1,
 x0^2*x1^2 + x0*x1^2 + 2*x0^2 + 2*x0,
 x0^2*x1^2 + 2*x0*x1^2 + 2*x0^2 + x0,
 x0^2*x1^2 + x0^2*x1 + 2*x1^2 + 2*x1,
 x0^2*x1^2 + x0^2*x1 + x0*x1^2 + x0*x1,
 x0^2*x1^2 + x0^2*x1 + 2*x0*x1^2 + 2*x0*x1,
 x0^2*x1^2 + 2*x0^2*x1 + 2*x1^2 + x1,
 x0^2*x1^2 + 2*x0^2*x1 + x0*x1^2 + 2*x0*x1,
 x0^2*x1^2 + 2*x0^2*x1 + 2*x0*x1^2 + x0*x1]
sage: f = x0^2*x1^2 + x0^2*x1 + x0*x1^2 + x0*x1             
sage: [f(x0=v[0],x1=v[1]) for v in V]             
[0, 0, 0, 0, 1, 0, 0, 0, 0]

\end{Verbatim}
}
\end{example}

\begin{proposition} (Hou)
{\rm
The degree of any bent function 
$f:GF(p)^n\to GF(p)$, when represented in ANF, satisfies

\[
\deg(f)\leq \frac{n(p-1)}{2} +1.
\]
The degree of any weakly regular bent function 
$f:GF(p)^n\to GF(p)$, when represented in ANF, satisfies

\[
\deg(f)\leq \frac{n(p-1)}{2}.
\]
}
\end{proposition}

For a proof of these results, see Hou \cite{H} (and see also
\cite{CM} for further details).

\subsection{Bent functions $GF(3)^2\to GF(3)$}
\label{sec:bent-gf3**2}

We focus on examples of even functions
$GF(3)^2\to GF(3)$ sending $0$ to $0$. There are exactly $3^4=81$ 
such functions.

\begin{example}
\label{example:bent1}
{\rm
Here is an example of a bent function $f$ of two variables over $GF(3)$.
This function $f$ is defined by the following table of values:

\vskip .15in
{\scriptsize{
\begin{tabular}{c|ccccccccc}
$GF(3)^2$ & (0, 0) & (1, 0) & (2, 0) & (0, 1) & (1, 1) & (2, 1) & (0,
2) & (1, 2) & (2, 2) \\ \hline
$f$ & 0  & 1  & 1  & 1  & 2  & 2  & 1  & 2  & 2 \\
\end{tabular}
}}
\vskip .15in

\begin{Verbatim}[fontsize=\scriptsize,fontfamily=courier,fontshape=tt,frame=single,label=\sage]

sage: V = GF(3)^2  
sage: Vlist = V.list()    
sage: Vlist
[(0, 0), (1, 0), (2, 0), (0, 1), (1, 1), (2, 1), (0, 2), (1, 2), (2, 2)]
sage: f00 = 0; f10 = 1; f01 = 1; f11 = 2; f12 = 2
sage: flist = [f00,f10,f10,f01,f11,f12,f01,f12,f11]
sage: f = lambda x: GF(3)(flist[Vlist.index(x)])
sage: [CC(hadamard_transform(f,a)).abs() for a in V] 
[3.00000000000000, 3.00000000000000, 3.00000000000000, 
3.00000000000000, 3.00000000000000, 3.00000000000000, 
3.00000000000000, 3.00000000000000, 3.00000000000000]
sage: pts = [CC(hadamard_transform(f, a)) for a in V]
sage: t = var('t')                                                                                     
sage: P1 = points([(x.real(), x.imag()) for x in pts], 
          pointsize=40, xmin=-12,xmax=12,ymin=-12,ymax=12)
sage: P2 = parametric_plot([(3)*cos(t),(3)*sin(t)], (t,0,2*pi), linestyle = "--")                      
sage: (P1+P2).show() 

\end{Verbatim}

\vskip .2in
\noindent
The plot of the values of the Hadamard transform (\ref{eqn:FT}) of $f$ is in Figure
\ref{fig:bent1-FT}. 

\begin{figure}[t!]
\begin{minipage}{\textwidth}
\begin{center}
\includegraphics[height=9cm,width=9cm]{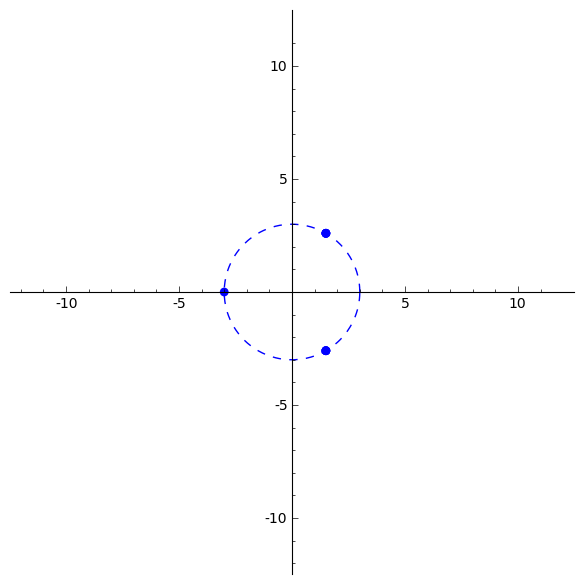}
\end{center}
\end{minipage}
\caption{The plot of the values of the Hadamard transform of 
$f$ in the complex plane of the even $GF(3)$-valued bent function of two
  variables from Example \ref{example:bent1}. 
(The vertices are ordered as in Example \ref{example:bent1}.)
}
\label{fig:bent1-FT}
\end{figure}

The set of such functions has some amusing combinatorial properties we shall 
discuss below.
}
\end{example}

There are exactly $18$ such bent functions.

\vskip .15in
{\scriptsize{
\begin{tabular}{c|ccccccccc}
$GF(3)^2$ & (0, 0) & (1, 0) & (2, 0) & (0, 1) & (1, 1) & (2, 1) & (0,
2) & (1, 2) & (2, 2) \\ \hline
$b_1$ & 0  & 1  & 1  & 1  & 2  & 2  & 1  & 2  & 2 \\
$b_2$ & 0  & 2  & 2  & 1  & 0  & 0  & 1  & 0  & 0 \\
$b_3$ & 0  & 1  & 1  & 2  & 0  & 0  & 2  & 0  & 0 \\
$b_4$ & 0  & 2  & 2  & 0  & 1  & 0  & 0  & 0  & 1 \\
$b_5$ & 0  & 0  & 0  & 2  & 1  & 0  & 2  & 0  & 1 \\
$b_6$ & 0  & 1  & 1  & 0  & 2  & 0  & 0  & 0  & 2 \\
$b_7$ & 0  & 0  & 0  & 1  & 2  & 0  & 1  & 0  & 2 \\
$b_8$ & 0  & 2  & 2  & 0  & 0  & 1  & 0  & 1  & 0 \\
$b_9$ & 0  & 0  & 0  & 2  & 0  & 1  & 2  & 1  & 0 \\
$b_{10}$ & 0  & 2  & 2  & 2  & 1  & 1  & 2  & 1  & 1 \\
$b_{11}$ & 0  & 0  & 0  & 0  & 2  & 1  & 0  & 1  & 2 \\
$b_{12}$ & 0  & 2  & 2  & 1  & 2  & 1  & 1  & 1  & 2 \\
$b_{13}$ & 0  & 1  & 1  & 2  & 2  & 1  & 2  & 1  & 2 \\
$b_{14}$ & 0  & 1  & 1  & 0  & 0  & 2  & 0  & 2  & 0 \\
$b_{15}$ & 0  & 0  & 0  & 1  & 0  & 2  & 1  & 2  & 0 \\
$b_{16}$ & 0  & 0  & 0  & 0  & 1  & 2  & 0  & 2  & 1 \\
$b_{17}$ & 0  & 2  & 2  & 1  & 1  & 2  & 1  & 2  & 1 \\
$b_{18}$ & 0  & 1  & 1  & 2  & 1  & 2  & 2  & 2  & 1 \\ \hline
\end{tabular}
}}

\vskip .1in
The {\it unweighted} Cayley graph of $b_2$
(as well as $b_3$, $b_4$, $b_5$, $b_6$, $b_7$, $b_8$, 
$b_9$, $b_{11}$, $b_{14}$, $b_{15}$ and $b_{16}$)
is a strongly regular graph having 
parameters $SRG(\nu, k, \lambda, \mu)$ where
$\nu =9$, $k = 4$, $\lambda = 1$, $\mu=2$.
We say that these bent functions are of 
{\it type} $(9,4,1,2)$. The other $6$ bent functions are of
{\it type} $(9,8,7,0)$.
Up to isomorphism, there is only one (unweighted) strongly
regular graph having parameters $SRG(9,4,1,2)$ 
\cite{Br}, \cite{Sp}. We shall see later that the edge-weighted
Cayley graphs arising from these $12$ bent functions of
type $(9,4,1,2)$ are also isomorphic\footnote{
We say edge-weighted graphs are {\it isomorphic} if there
is a bijection of the vertices which preserves the weight
of each edge.
} 
as weighted 
(strongly regular) graphs.
Likewise, these $6$ bent functions of
type $(9,8,7,0)$ are also isomorphic as weighted 
(strongly regular) graphs.

\begin{example}
\label{example:bentsGF3}
{\rm
Let $b_1,\dots, b_{18}$ denote the bent functions defined in
\S \ref{sec:bent-gf3**2}.
The following example shows that the dual of $b_1$ is $b_{10}$
and the dual of $b_2$ is $b_3$, but in one case we must 
pre-multiply by $-1$ and in the other case we don't.

\vskip .15in

\begin{Verbatim}[fontsize=\footnotesize,fontfamily=courier,fontshape=tt,frame=single,label=\sage]

sage: FF = GF(3)
sage: V = FF^2
sage: Vlist = V.list()
sage: flist = [0, 1, 1, 1, 2, 2, 1, 2, 2]
## this is b1
sage: f = lambda x: GF(3)(flist[Vlist.index(x)])
sage: [CC(hadamard_transform(f,a)).abs() for a in V]
[3.00000000000000, 3.00000000000000, 3.00000000000000, 
 3.00000000000000, 3.00000000000000, 3.00000000000000, 
 3.00000000000000, 3.00000000000000, 3.00000000000000]
sage: L = [CC(hadamard_transform(f,a)) for a in V]; L
[-3.00000000000000 + 1.33226762955019e-15*I, 
 1.50000000000000 + 2.59807621135332*I, 
 1.50000000000000 + 2.59807621135332*I, 
 1.50000000000000 + 2.59807621135332*I, 
 1.50000000000000 - 2.59807621135331*I, 
 1.50000000000000 - 2.59807621135331*I, 
 1.50000000000000 + 2.59807621135332*I, 
 1.50000000000000 - 2.59807621135331*I, 
 1.50000000000000 - 2.59807621135331*I]
sage: [crude_CC_log(-z/3, 3) for z in L]
[0, 2, 2, 2, 1, 1, 2, 1, 1]
## this is b10
\end{Verbatim}

\vskip .2in
\noindent
Note the pre-multiplication by $-1$. This 
bent function $f=b_1:GF(3)^2\to GF(3)$ is weakly regular,
and the weakly regular dual of $b_1$ is $b_{10}$.

\vskip .15in

\begin{Verbatim}[fontsize=\scriptsize,fontfamily=courier,fontshape=tt,frame=single,label=\sage]

sage: flist = [0, 2, 2, 1, 0, 0, 1, 0, 0]
## this is b2
sage: f = lambda x: GF(3)(flist[Vlist.index(x)])
sage: L = [CC(hadamard_transform(f,a)) for a in V]; L
[3.00000000000000 + 6.66133814775094e-16*I, 
 -1.50000000000000 + 2.59807621135332*I, 
 -1.50000000000000 + 2.59807621135332*I, 
 -1.50000000000000 - 2.59807621135331*I, 
 3.00000000000000 + 6.66133814775094e-16*I, 
 3.00000000000000 + 6.66133814775094e-16*I, 
 -1.50000000000000 - 2.59807621135331*I, 
 3.00000000000000 + 6.66133814775094e-16*I, 
 3.00000000000000 + 6.66133814775094e-16*I]
sage: [crude_CC_log(z/3, 3) for z in L]
[0, 1, 1, 2, 0, 0, 2, 0, 0]
## this is b3

\end{Verbatim}

\vskip .2in
\noindent
This bent function $f=b_2:GF(3)^2\to GF(3)$ is regular,
and the regular dual of $b_2$ is $b_{3}$.

We similar computations verify the following:
\begin{itemize}
\item
$b_1$ and $b_{10}$ are both weakly regular and 
$-1$-dual to each other,

\item
$b_2$ and $b_3$ are regular and dual to each other,
\item
$b_4$ and $b_9$ are regular and dual to each other,
\item
$b_5$ and $b_8$ are regular and dual to each other,
\item
$b_6$ and $b_{15}$ are regular and dual to each other,
\item
$b_7$ and $b_{14}$ are regular and dual to each other,
\item
$b_{11}$ and $b_{16}$ are regular and dual to each other,
\item
$b_{12}$, $b_{13}$, $b_{17}$ and $b_{18}$ are all weakly regular
and are each $-1$-self-dual.
\end{itemize}
}
\end{example}

\vskip .3in
Relationships:

\[
b_1=-b_{10}, \ \ 
b_2=-b_3,  \ \ 
b_4 = -b_6,  \ \ 
b_5 = -b_7,  \ \ 
b_8=-b_{14}, 
\]
\[
b_9=-b_{15},  \ \ 
b_{11}=-b_{16}, \ \ 
b_{12}=-b_{18}, \ \ 
b_{13}=-b_{17},
\]
\[
b_1 = b_7+b_{14} = b_6+b_{15}, \ \ 
b_{10} = b_{4}+b_{9} = b_{5}+b_{8}, \ \ 
b_{12} = b_{2}+b_{11} = b_{7}+b_{8},
\]
\[
b_{13} = b_{3}+b_{11} = b_{6}+b_{9}, \ \ 
b_{17} = b_{2}+b_{16} = b_{4}+b_{15,}\ \ 
b_{18} = b_{3}+b_{16} = b_{5}+b_{14}.
\]

Recall from Example \ref{example:bentsGF3}, the following
are regular

\[
b_2^* = b_3, \ \ 
b_{4}^* =b_{9},\ b_5^* = b_{8},\  b_{6}^* = b_{15}, \ \ 
b_{7}^* = b_{14}, b_{11}^* = b_{16},\ 
\]
whereas
\[
b_1^* = - b_{10}
\]
are weakly regular and $(-1)$-dual to each other (but not
regular), and the others are all $(-1)$-self-dual and weakly regular
(but not regular),

\[
b_{12}^* = - b_{12},\ 
b_{13}^* = - b_{13},\ 
b_{17}^* = - b_{17},\ 
b_{18}^* = - b_{18}.
\]

Supports:

\[
\{1,2,3,6\} = \supp(b_2)=\supp(b_3), \ \ 
\{1,2,4,8\} = \supp(b_4)=\supp(b_6),
\]
\[
\{1,2,5,7\} = \supp(b_8)=\supp(b_{14}), \ \ 
\{3,5,6,7\} = \supp(b_9)=\supp(b_{15}),
\]
\[
\{3,4,6,8\} = \supp(b_{5})=\supp(b_{7}),\ \
\{4,5,7,8\} = \supp(b_{11})=\supp(b_{16}),
\]
and
%{\small{
\[
\supp(b_{1})=\supp(b_{10})
= \supp(b_{12})=\supp(b_{13})
= \supp(b_{17})=\supp(b_{18})
\]
%%}}
are all equal to $\{1,2,3,4,5,6,7,8\}$.
Note that 
\begin{itemize}
\item
All these functions are weight $4$ or weight $8$.
\item
If you pick any two support sets of weight $4$,
$S_1$ and $S_2$ say, then they satisfy either

\[
S_1\cap S_2=\emptyset\ \ \ {\rm or}\ \ \
|S_1\cap S_2|=2.
\]
\item
The $12$ which are regular, but not $\mu$-regular for some $\mu\not= 1$)
can all be obtained from $f(x_0,x_1)=x_0^2+x_0x_1$ by linear 
transformations of the coordinates, i.e.
$(x_0,x_1)\mapsto (ax_0+bx_1,cx_0+dx_1)$ where $ad-bc\not= 0$.
Each such isomorphism of $GF(3)^2$ induces an isomorphism of the associated 
edge-weighted Cayley graphs.
\item
Similarly, the $6$ which are weakly regular can all be obtained from 
$x_0^2+x_1^2$ by linear transformations of the coordinates.
\end{itemize}

In fact, if you consider the set

\[
S = \{ \emptyset \} \cup \{ {\rm supp}(f)\ |\ f:GF(3)^2\to GF(3), \ f(0)=0 \},
\]
then $S$ forms a group under the symmetric difference operator
$\Delta$. In fact, $S\cong GF(2)^3$.

\vskip .2in

\begin{question}
To what extent is it true that if $f_1$, $f_2$ are bent 
functions on $GF(p)^n$ with $f_1(0)=f_2(0)=0$, then

\[
{\rm supp}(f_1)\, \Delta\, {\rm supp}(f_3) = {\rm supp}(f_3)
\]
for some bent function $f_3$ satisfying $f_3(0)=0$?
\end{question}

More general version:

\begin{question} Over $GF(p)$, $p\not= 2$, does the set of supports

\[
\{\emptyset\} \cup \{ {\rm supp}(f)\ |\ f \mbox{ is bent}, f(0)=0, f \mbox{ is even} \}
\]
form a group (under $\Delta$)?
\end{question}

This does not seem to hold in the binary case\footnote{What is true in the
binary case is an oddly similar result: the vectors in the support of a bent function
form a Hadamard difference set in the additive group $GF(2)^n$.}.

The following was verified with direct (computer-aided) computations.

\begin{lemma}
\label{lemma:34}
Assume $p=3$, $n=2$.
\begin{itemize}
\item[(a)]
The edge-weighted Cayley graph of $b_i$ is strongly regular and not complete as a
simple (unweighted) graph if and only if $b_i$ is regular if and only if
$i\in \{2,3,4,5,6,7,8,9,11,14,15,16\}.$
\item[(b)]
The edge-weighted Cayley graph of $b_i$ is strongly regular and complete as a
simple (unweighted) graph if and only if $b_i$ is weakly regular if and only if
$i\notin \{2,3,4,5,6,7,8,9,11,14,15,16\}.$
\end{itemize}
\end{lemma}

\begin{example}
\label{example:25}
{\rm
This example is intended to illustrate the 
bent function $b_8$ listed in the table above,
and to provide more detail on Example \ref{example:17}.

Consider the finite field

\[
GF(9) = GF(3)[x]/(x^2+1) = \{0,1,2,x,x+1,x+2,2x,2x+1,2x+2\}.
\]
The set of non-zero quadratic residues is given by

\[
D = \{1,2,x,2x\}.
\]
Let $\Gamma$ be the graph whose vertices are $GF(9)$ and whose edges
$e=(a,b)$ are those pairs for which $a-b\in D$.

The graph looks like the Cayley graph for $b_8$ in Figure \ref{fig:example6}, 
except

\[
8\to 0, 0\to 2x+2, 1\to 2x+1, 2\to 2x, 
\]
\[
3\to x+2, 4\to x+1, 5\to x, 6\to 2,
7\to 1, 8\to 0.
\] 
This is a strongly regular graph with parameters 
$(9,4,1,2)$.
\vskip .1in
{\scriptsize{
\begin{tabular}{c|ccccccccc}
$v$       & 0       & 1       & 2       & 3       & 4       & 5       & 6       & 7 & 8 \\ \hline
$N(v,0)$  & 3,4,6,8 & 4,5,6,7 & 3,5,7,8 & 0,2,6,7 & 0,1,7,8 & 1,2,6,8 & 0,1,3,5 & 1,2,3,4 & 0,2,4,5 \\   
$N(v,1)$ & 5,7 & 3,8 & 4,6 & 1,8 & 2,6 & 0,7 & 2,4 & 0,5 & 1,3 \\
$N(v,2)$ & 1,2 & 0,2 & 0,1 & 4,5 & 3,5 & 3,4 & 7,8 & 6,8 & 6,7 \\ 
\end{tabular}
}}
\vskip .1in

\noindent
The axioms of an edge-weighted strongly regular graph can 
be directly verified using this table.

\begin{figure}[t!]
\begin{minipage}{\textwidth}
\begin{center}
\includegraphics[height=4cm,width=4cm]{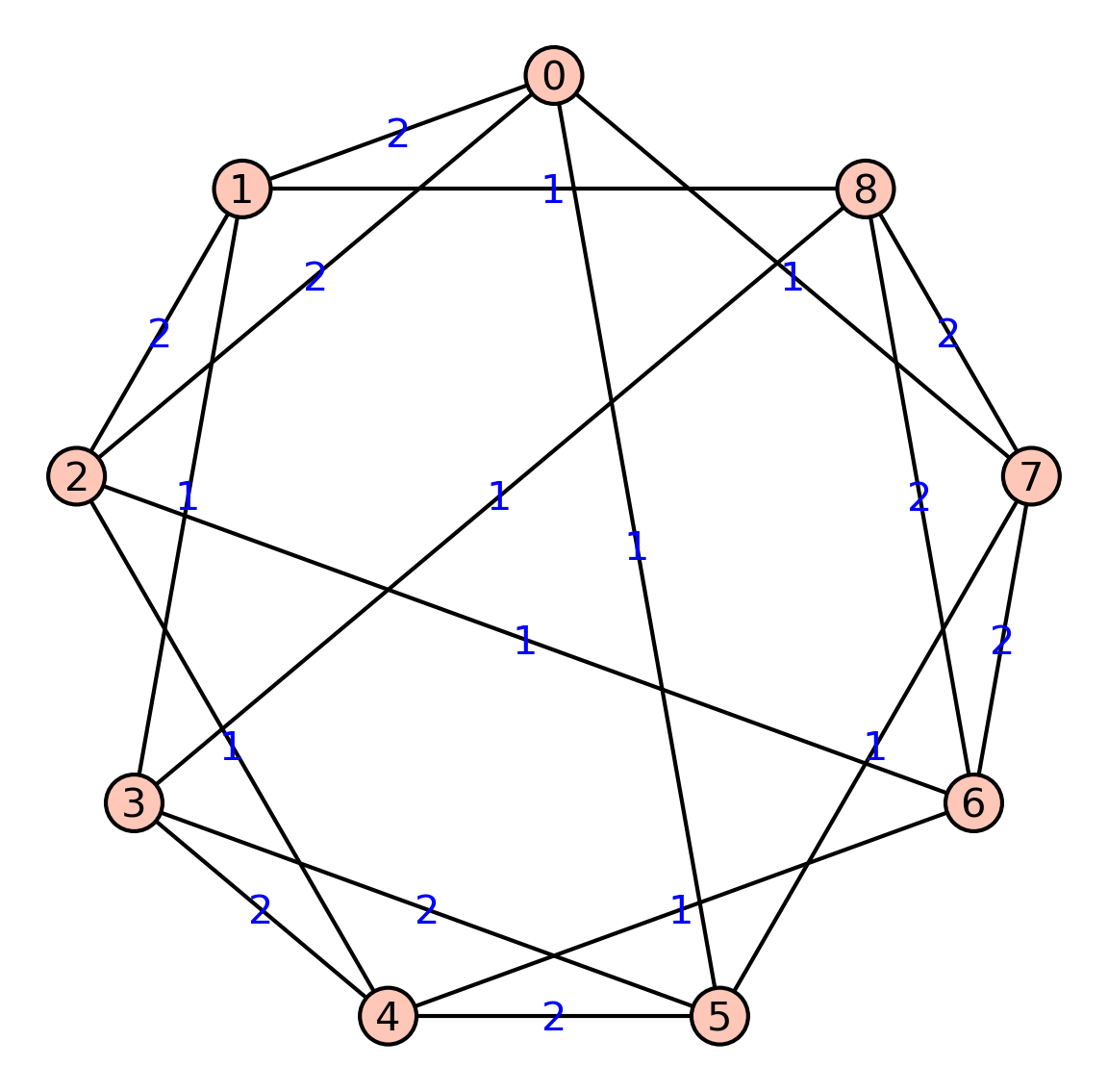} 
\end{center}
\end{minipage}
\caption{The Cayley graphs for $b_8$}
\label{fig:example6}
\end{figure}

}
\end{example}

\subsection{Bent functions $GF(3)^3\to GF(3)$}
\label{sec:bent-gf3**3}

Let $f:GF(3)^3\to GF(3)$ be an even bent function with $f(0)=0$,
let
\[
D_i = \{v\in GF(3)^3\ |\ f(v)=i\},\ \ \ i=1,2,
\]
let $D_0=\{0\}$ and $D_3=GF(3)^3-(D_0\cup D_1\cup D_2)$.

There are a total of $2340$ even bent functions on $GF(3)^3$.

We know that if $W_f(0)$ is rational then the level curves
$f^{-1}(i)$, $i\not= 0$, have the same cardinality\footnote{In fact, 
this is true any time $n$ is even (see \cite{CTZ}).}.
The following Sage computation shows that, in this case, $W_f(0)$ is not a
rational number for therepresentatives of $B_1,B_2$ displayed above.

\vskip .15in
{\footnotesize{
\begin{Verbatim}[fontsize=\scriptsize,fontfamily=courier,fontshape=tt,frame=single,label=\sage]

sage: PR.<x0,x1,x2> = PolynomialRing(FF, 3, "x0,x1,x2")
sage: f = x0^2 + x1^2 + x2^2
sage: V = GF(3)^3
sage: Vlist = V.list()
sage: flist = [f(x[0],x[1],x[2]) for x in Vlist]
sage: f = lambda x: FF(flist[Vlist.index(x)])
sage: hadamard_transform(f,V(0))
12*e^(4/3*I*pi) + 6*e^(2/3*I*pi) + 9
sage: CC(hadamard_transform(f,V(0)))
-3.99680288865056e-15 - 5.19615242270663*I
sage: f = x0*x2 + 2*x1^2 + 2*x0^2*x2^2
sage: flist = [f(x[0],x[1],x[2]) for x in Vlist]
sage: f = lambda x: FF(flist[Vlist.index(x)])
sage: hadamard_transform(f,V(0))
14*e^(4/3*I*pi) + 2*e^(2/3*I*pi) + 11
sage: CC(hadamard_transform(f,V(0)))
2.99999999999999 - 10.3923048454133*I
\end{Verbatim}
}}

Using Sage, we give some examples of an even bent function of
$3$ variables over $GF(3)$.

\vskip .15in
{\footnotesize{
\begin{Verbatim}[fontsize=\scriptsize,fontfamily=courier,fontshape=tt,frame=single,label=\sage]

sage: FF = GF(3)
sage: V = FF^3
sage: Vlist = V.list()
sage: Vlist
[(0, 0, 0), (1, 0, 0), (2, 0, 0), (0, 1, 0), (1, 1, 0), (2, 1, 0), 
(0, 2, 0), (1, 2, 0), (2, 2, 0), (0, 0, 1), (1, 0, 1), (2, 0, 1), 
(0, 1, 1), (1, 1, 1), (2, 1, 1), (0, 2, 1), (1, 2, 1), (2, 2, 1), 
(0, 0, 2), (1, 0, 2), (2, 0, 2), (0, 1, 2), (1, 1, 2), (2, 1, 2), 
(0, 2, 2), (1, 2, 2), (2, 2, 2)]
sage: flist = [0,2,2,1,1,1,1,1,1,2,0,1,1,2,0,1,0,0,2,1,0,1,0,0,1,0,2]
sage: f = lambda x: GF(3)(flist[Vlist.index(x)])   
sage: [f(a)- f(-a) for a in V]
[0, 0, 0, 0, 0, 0, 0, 0, 0, 0, 0, 0, 0, 0, 0, 0, 0, 0, 0, 0, 0, 0, 0, 0, 0, 0, 0]
sage: [CC(hadamard_transform(f,a)).abs() for a in V]
[5.19615242270663, 5.19615242270663, 5.19615242270664, 5.19615242270663,
 5.19615242270663, 5.19615242270663, 5.19615242270663, 5.19615242270663,
 5.19615242270663, 5.19615242270663, 5.19615242270663, 5.19615242270663,
 5.19615242270663, 5.19615242270664, 5.19615242270663, 5.19615242270663,
 5.19615242270663, 5.19615242270663, 5.19615242270664, 5.19615242270663,
 5.19615242270663, 5.19615242270663, 5.19615242270663, 5.19615242270663,
 5.19615242270663, 5.19615242270663, 5.19615242270664]
sage: supp_f = [Vlist.index(x) for x in V if f(x)<>0]; supp_f
[1, 2, 3, 4, 5, 6, 7, 8, 9, 11, 12, 13, 15, 18, 19, 21, 24, 26]
\end{Verbatim}
}}

Since $W_f(0)/3^{3/2}$ is a $6$-th root of unity, but not a cube root of unity, it
follows from Lemma \ref{lemma:weaklyregular} that $f$ is not weakly regular.

\vskip .15in
{\footnotesize{
\begin{Verbatim}[fontsize=\scriptsize,fontfamily=courier,fontshape=tt,frame=single,label=\sage]

sage: a = 2; b = 4; c = 4
sage: flist = [0,a,a,c,b,1,c,1,b,2,0,1,1,2,0,1,0,0,2,1,0,1,0,0,1,0,2]
sage: f = lambda x: GF(3)(flist[Vlist.index(x)])   
sage: [f(a)- f(-a) for a in V]
[0, 0, 0, 0, 0, 0, 0, 0, 0, 0, 0, 0, 0, 0, 0, 0, 0, 0, 0, 0, 0, 0, 0, 0, 0, 0, 0]
sage: [CC(hadamard_transform(f,a)).abs() for a in V]
[5.19615242270663, 5.19615242270663, 5.19615242270664, 5.19615242270663,
 5.19615242270663, 5.19615242270663, 5.19615242270663, 5.19615242270663,
 5.19615242270663, 5.19615242270663, 5.19615242270663, 5.19615242270663,
 5.19615242270663, 5.19615242270664, 5.19615242270663, 5.19615242270663,
 5.19615242270663, 5.19615242270663, 5.19615242270664, 5.19615242270663,
 5.19615242270663, 5.19615242270663, 5.19615242270663, 5.19615242270663,
 5.19615242270663, 5.19615242270663, 5.19615242270664]
sage: supp_f = [Vlist.index(x) for x in V if f(x)<>0]; supp_f
[1, 2, 3, 4, 5, 6, 7, 8, 9, 11, 12, 13, 15, 18, 19, 21, 24, 26]
\end{Verbatim}
}}

Here are some algebraic examples:

\vskip .15in
{\footnotesize{
\begin{Verbatim}[fontsize=\scriptsize,fontfamily=courier,fontshape=tt,frame=single,label=\sage]

sage: R.<x0,x1,x2> = PolynomialRing(FF,3,"x0,x1,x2")
sage: V = GF(3)^3
sage: ff = x1^2+x0*x2
sage: flist = [ff(x0=v[0],x1=v[1],x2=v[2]) for v in V]
sage: Vlist = V.list()
sage: f = lambda x: GF(3)(flist[Vlist.index(x)])
sage: [CC(hadamard_transform(f,a)).abs() for a in V]
[5.19615242270663, 5.19615242270663, 5.19615242270663,
5.19615242270663, 5.19615242270663, 5.19615242270663,
5.19615242270663, 5.19615242270663, 5.19615242270663,
5.19615242270663, 5.19615242270663, 5.19615242270663,
5.19615242270663, 5.19615242270663, 5.19615242270663,
5.19615242270663, 5.19615242270663, 5.19615242270663,
5.19615242270663, 5.19615242270663, 5.19615242270663,
5.19615242270663, 5.19615242270663, 5.19615242270663,
5.19615242270663, 5.19615242270663, 5.19615242270663]
sage:
sage: ff = x1^2+x0*x2+x1+2*x0
sage: flist = [ff(x0=v[0],x1=v[1],x2=v[2]) for v in V]
sage: f = lambda x: GF(3)(flist[Vlist.index(x)])
sage: [CC(hadamard_transform(f,a)).abs() for a in V]
[5.19615242270663, 5.19615242270663, 5.19615242270663,
5.19615242270663, 5.19615242270663, 5.19615242270663,
5.19615242270663, 5.19615242270663, 5.19615242270663,
5.19615242270663, 5.19615242270663, 5.19615242270663,
5.19615242270663, 5.19615242270663, 5.19615242270663,
5.19615242270663, 5.19615242270663, 5.19615242270663,
5.19615242270663, 5.19615242270663, 5.19615242270663,
5.19615242270663, 5.19615242270663, 5.19615242270663,
5.19615242270663, 5.19615242270663, 5.19615242270663]
\end{Verbatim}
}}

Note the first example is even but the second one is not.

\subsection{Bent functions $GF(5)^2\to GF(5)$}
\label{sec:bent-gf5**2}

Using Sage, we give examples of bent functions of
$2$ variables over $GF(5)$ and study their signatures (\ref{eqn:support-i}).

Do the ``level curves'' of a bent function $GF(5)^2\to GF(5)$
give rise to a PDS? An association scheme (see below for a definition)?

\begin{example}
{\rm
Let $G = GF(25) = GF(5)[x]/(x^2+2)$,

\[
D_1 = \{1, 4, x+2, 4x+3\},
D_2 = \{x+1, x+3, 4x+2, 4x+4\},
\]
\[
D_3 = \{2x+1, 2x+2, 3x+3, 3x+4\},
D_4 = \{2, 3, 2x+4, 3x+1\},
\]
and $D = D_1\cup  D_2 \cup D_3 \cup D_4$.
If $f(x_0,x_1) = x_0^2+x_0x_1$ then
each subset $D_i$ ($i=1,2,3,4$) is the image
of the level curve $f^{-1}(i)$
under the $GF(5)$-vector space isomorphism

\[
\begin{array}{ccc}
\phi\, :\, GF(5)^2 & \to & GF(25),\\
(a,b) & \longmapsto & bx+a \, ,\\
\end{array}
\]
$D_i=\phi(f^{-1}(i))$, $i=1,2,3,4$.
As in the previous example, we can compute the
$k_{i,j}$'s, $\mu_{i,j}$'s, and $\lambda_{i,j}^k$'s
(see Example \ref{example:x02+x0x1} for the details).
This $f$ is homogeneous, bent and regular (hence also weakly regular).

The weighted adjacency matrix $A$ of the 
edge-weighted Cayley graph associated to $f$,
$\Gamma_f$, is given below:

\[
\left(\begin{array}{rrrrrrrrrrrrrrrrrrrrrrrrr}
0 & 1 & 4 & 4 & 1 & 0 & 2 & 1 & 2 & 0 & 0 & 3 & 3 & 0 & 4 & 0 & 4 & 0 & 3 & 3 & 0 & 0 & 2 & 1 & 2 \\
1 & 0 & 1 & 4 & 4 & 0 & 0 & 2 & 1 & 2 & 4 & 0 & 3 & 3 & 0 & 3 & 0 & 4 & 0 & 3 & 2 & 0 & 0 & 2 & 1 \\
4 & 1 & 0 & 1 & 4 & 2 & 0 & 0 & 2 & 1 & 0 & 4 & 0 & 3 & 3 & 3 & 3 & 0 & 4 & 0 & 1 & 2 & 0 & 0 & 2 \\
4 & 4 & 1 & 0 & 1 & 1 & 2 & 0 & 0 & 2 & 3 & 0 & 4 & 0 & 3 & 0 & 3 & 3 & 0 & 4 & 2 & 1 & 2 & 0 & 0 \\
1 & 4 & 4 & 1 & 0 & 2 & 1 & 2 & 0 & 0 & 3 & 3 & 0 & 4 & 0 & 4 & 0 & 3 & 3 & 0 & 0 & 2 & 1 & 2 & 0 \\
0 & 0 & 2 & 1 & 2 & 0 & 1 & 4 & 4 & 1 & 0 & 2 & 1 & 2 & 0 & 0 & 3 & 3 & 0 & 4 & 0 & 4 & 0 & 3 & 3 \\
2 & 0 & 0 & 2 & 1 & 1 & 0 & 1 & 4 & 4 & 0 & 0 & 2 & 1 & 2 & 4 & 0 & 3 & 3 & 0 & 3 & 0 & 4 & 0 & 3 \\
1 & 2 & 0 & 0 & 2 & 4 & 1 & 0 & 1 & 4 & 2 & 0 & 0 & 2 & 1 & 0 & 4 & 0 & 3 & 3 & 3 & 3 & 0 & 4 & 0 \\
2 & 1 & 2 & 0 & 0 & 4 & 4 & 1 & 0 & 1 & 1 & 2 & 0 & 0 & 2 & 3 & 0 & 4 & 0 & 3 & 0 & 3 & 3 & 0 & 4 \\
0 & 2 & 1 & 2 & 0 & 1 & 4 & 4 & 1 & 0 & 2 & 1 & 2 & 0 & 0 & 3 & 3 & 0 & 4 & 0 & 4 & 0 & 3 & 3 & 0 \\
0 & 4 & 0 & 3 & 3 & 0 & 0 & 2 & 1 & 2 & 0 & 1 & 4 & 4 & 1 & 0 & 2 & 1 & 2 & 0 & 0 & 3 & 3 & 0 & 4 \\
3 & 0 & 4 & 0 & 3 & 2 & 0 & 0 & 2 & 1 & 1 & 0 & 1 & 4 & 4 & 0 & 0 & 2 & 1 & 2 & 4 & 0 & 3 & 3 & 0 \\
3 & 3 & 0 & 4 & 0 & 1 & 2 & 0 & 0 & 2 & 4 & 1 & 0 & 1 & 4 & 2 & 0 & 0 & 2 & 1 & 0 & 4 & 0 & 3 & 3 \\
0 & 3 & 3 & 0 & 4 & 2 & 1 & 2 & 0 & 0 & 4 & 4 & 1 & 0 & 1 & 1 & 2 & 0 & 0 & 2 & 3 & 0 & 4 & 0 & 3 \\
4 & 0 & 3 & 3 & 0 & 0 & 2 & 1 & 2 & 0 & 1 & 4 & 4 & 1 & 0 & 2 & 1 & 2 & 0 & 0 & 3 & 3 & 0 & 4 & 0 \\
0 & 3 & 3 & 0 & 4 & 0 & 4 & 0 & 3 & 3 & 0 & 0 & 2 & 1 & 2 & 0 & 1 & 4 & 4 & 1 & 0 & 2 & 1 & 2 & 0 \\
4 & 0 & 3 & 3 & 0 & 3 & 0 & 4 & 0 & 3 & 2 & 0 & 0 & 2 & 1 & 1 & 0 & 1 & 4 & 4 & 0 & 0 & 2 & 1 & 2 \\
0 & 4 & 0 & 3 & 3 & 3 & 3 & 0 & 4 & 0 & 1 & 2 & 0 & 0 & 2 & 4 & 1 & 0 & 1 & 4 & 2 & 0 & 0 & 2 & 1 \\
3 & 0 & 4 & 0 & 3 & 0 & 3 & 3 & 0 & 4 & 2 & 1 & 2 & 0 & 0 & 4 & 4 & 1 & 0 & 1 & 1 & 2 & 0 & 0 & 2 \\
3 & 3 & 0 & 4 & 0 & 4 & 0 & 3 & 3 & 0 & 0 & 2 & 1 & 2 & 0 & 1 & 4 & 4 & 1 & 0 & 2 & 1 & 2 & 0 & 0 \\
0 & 2 & 1 & 2 & 0 & 0 & 3 & 3 & 0 & 4 & 0 & 4 & 0 & 3 & 3 & 0 & 0 & 2 & 1 & 2 & 0 & 1 & 4 & 4 & 1 \\
0 & 0 & 2 & 1 & 2 & 4 & 0 & 3 & 3 & 0 & 3 & 0 & 4 & 0 & 3 & 2 & 0 & 0 & 2 & 1 & 1 & 0 & 1 & 4 & 4 \\
2 & 0 & 0 & 2 & 1 & 0 & 4 & 0 & 3 & 3 & 3 & 3 & 0 & 4 & 0 & 1 & 2 & 0 & 0 & 2 & 4 & 1 & 0 & 1 & 4 \\
1 & 2 & 0 & 0 & 2 & 3 & 0 & 4 & 0 & 3 & 0 & 3 & 3 & 0 & 4 & 2 & 1 & 2 & 0 & 0 & 4 & 4 & 1 & 0 & 1 \\
2 & 1 & 2 & 0 & 0 & 3 & 3 & 0 & 4 & 0 & 4 & 0 & 3 & 3 & 0 & 0 & 2 & 1 & 2 & 0 & 1 & 4 & 4 & 1 & 0
\end{array}\right)
\]

This is verified using the following Sage commands:

\vskip .15in
{\footnotesize{
\begin{Verbatim}[fontsize=\scriptsize,fontfamily=courier,fontshape=tt,frame=single,label=\sage]

sage: attach "/home/wdj/sagefiles/hadamard_transform.sage"
sage: FF = GF(5)
sage: V = FF^2
sage: R.<x0,x1> = PolynomialRing(FF,2,"x0,x1")
sage: ff = x0^2+x0*x1
sage: flist = [ff(x0=v[0],x1=v[1]) for v in V]
sage: Vlist = V.list()
sage: f = lambda x: FF(flist[Vlist.index(x)])
sage: Gamma = boolean_cayley_graph(f, V)
sage: A = Gamma.weighted_adjacency_matrix(); A
25 x 25 sparse matrix over Finite Field of size 5 

\end{Verbatim}
}}

}
\end{example}

First, we consider the bent regular function $f(x_0,x_1)=x_0^2+x_0x_1$.

\vskip .15in
{\footnotesize{
\begin{Verbatim}[fontsize=\scriptsize,fontfamily=courier,fontshape=tt,frame=single,label=\sage]

sage: FF = GF(5)
sage: V = FF^2
sage: R.<x0,x1> = PolynomialRing(FF,2,"x0,x1")
sage: ff = x0^2+x0*x1
sage: flist = [ff(x0=v[0],x1=v[1]) for v in V]
sage: Vlist = V.list()
sage: f = lambda x: FF(flist[Vlist.index(x)])
sage: [CC(hadamard_transform(f,a)).abs() for a in V]
[5.00000000000000, 5.00000000000000, 5.00000000000000,
5.00000000000000, 5.00000000000000, 5.00000000000000,
5.00000000000000, 5.00000000000000, 5.00000000000000,
5.00000000000000, 5.00000000000000, 5.00000000000000,
5.00000000000000, 5.00000000000000, 5.00000000000000,
5.00000000000000, 5.00000000000000, 5.00000000000000,
5.00000000000000, 5.00000000000000, 5.00000000000000,
5.00000000000000, 5.00000000000000, 5.00000000000000,
5.00000000000000]
sage: Gamma = boolean_cayley_graph(f, V)
sage: Gamma.connected_components_number()
1
sage: Gamma.spectrum()
[16, 1, 1, 1, 1, 1, 1, 1, 1, 1, 1, 1, 1, 1, 1, 1, 1, -4, -4, -4, -4,
-4, -4, -4, -4]

\end{Verbatim}
}}

Next, we consider $f(x_0,x_1)=x_0^2+x_1^2$, having signature
$[9, 4, 4, 4, 4]$. This is also regular bent.

\vskip .15in
{\footnotesize{
\begin{Verbatim}[fontsize=\scriptsize,fontfamily=courier,fontshape=tt,frame=single,label=\sage]

sage: ff = x0^2+x1^2                                
sage: flist = [FF(x) for x in flist]
sage: f = lambda x: FF(flist[Vlist.index(x)])
sage: ff = x0^2+x1^2                         
sage: flist = [ff(x0=v[0],x1=v[1]) for v in V]
sage: f = lambda x: FF(flist[Vlist.index(x)])
sage: flist = [FF(x) for x in flist]          
sage: [CC(hadamard_transform(f,a)).abs() for a in V]
[5.00000000000000, 5.00000000000000, 5.00000000000000, 
 5.00000000000000, 5.00000000000000, 5.00000000000000,
 5.00000000000000, 5.00000000000000, 5.00000000000000,
 5.00000000000000, 5.00000000000000, 5.00000000000000,
 5.00000000000000, 5.00000000000000, 5.00000000000000,
 5.00000000000000, 5.00000000000000, 5.00000000000000,
 5.00000000000000, 5.00000000000000, 5.00000000000000,
 5.00000000000000, 5.00000000000000, 5.00000000000000,
 5.00000000000000]
sage: flist                                   
[0, 1, 4, 4, 1, 1, 2, 0, 0, 2, 4, 0, 3, 3, 0, 4, 0, 3, 3, 0, 1, 2, 0, 0, 2]
sage: fcount = [flist.count(x) for x in FF]; fcount
[9, 4, 4, 4, 4]

\end{Verbatim}
}}

\noindent
Not that this is even and the signature agrees with 
Lemma \ref{lemma:sig}.

Finally, we consider $f(x_0,x_1)=x_0^2+2x_1^2+x_0$, having signature
$[4, 9, 4, 4, 4]$. This is also regular bent.

\vskip .15in
{\footnotesize{
\begin{Verbatim}[fontsize=\scriptsize,fontfamily=courier,fontshape=tt,frame=single,label=\sage]

sage: ff = x0^2-x1^2+x0
sage: flist = [ff(x0=v[0],x1=v[1]) for v in V]
sage: f = lambda x: FF(flist[Vlist.index(x)])
sage: [CC(hadamard_transform(f,a)).abs() for a in V]
[5.00000000000000, 5.00000000000000, 5.00000000000000, 
 5.00000000000000, 5.00000000000000, 5.00000000000000, 
 5.00000000000000, 5.00000000000000, 5.00000000000000, 
 5.00000000000000, 5.00000000000000, 5.00000000000000, 
 5.00000000000000, 5.00000000000000, 5.00000000000000, 
 5.00000000000000, 5.00000000000000, 5.00000000000000, 
 5.00000000000000, 5.00000000000000, 5.00000000000000, 
 5.00000000000000, 5.00000000000000, 5.00000000000000, 
 5.00000000000000]
sage: flist = [FF(x) for x in flist]; flist   
[0, 2, 1, 2, 0, 4, 1, 0, 1, 4, 1, 3, 2, 3, 1, 1, 3, 2, 3, 1, 4, 1, 0, 1, 4]
sage: fcount = [flist.count(x) for x in FF]; fcount
[4, 9, 4, 4, 4]

\end{Verbatim}
}}

\begin{example}
{\rm
Consider the bent function
\[
f_4(x_0,x_1)=-x_1^4-2x_1^2+x_0x_1.
\]
This function represents a $GL(2,GF(5))$ orbit of size $120$.
The level curves of this function do not give rise to a weighted PDS.
By the way, if we try a computation of all the $p_{ij}^k$'s as in the
above example, we do not get integers.

Similarly, the  level curves of $f_7$, $f_8$, $f_{10}$, and $f_{11}$ do not
give rise to a weighted PDS.

}
\end{example}

\begin{example}
{\rm
The example of $f_5$ above can be used to construct an edge-weighted 
strongly regular Cayley graph, hence also a weighted PDS attached to its
level curves.
The examples of $f_6$ and $f_9$ above have an isomorphic weighted 
PDS attached to their (respective) level curves.

}
\end{example}

\vskip .2in
\noindent
{\it Acknowledgements}:
We are grateful to our colleague T. S. Michael for many
stimulating conversations and suggestions on this paper.

\section{Appendix: A new search algorithm for bent functions}

The following algorithm and code is due to the first author
C. Celerier.

\begin{Verbatim}[fontsize=\tiny,fontfamily=courier,fontshape=tt,frame=single,label=Python]

from collections import defaultdict
from copy import deepcopy
from random import shuffle
from sage.crypto.boolean_function import BooleanFunction

class NoBentFunction(Exception):
    pass

class BentFinder(object):
    def __init__(self, n):
        self.V = GF(2)**n
        self.n = n

    def searchForBent(self):
        self.walshTrace = []
        A = defaultdict(int)
        W = defaultdict(int)
        B = list(self.V.list())
        shuffle(B)
        R = self.__searchForBent(A,B,W,0)
        self.walshTrace.reverse()
        return BooleanFunction([R[tuple(x)] for x in self.V]), self.walshTrace

    def __searchForBent(self, A, B, W, wgt):
        n = self.n
        if len(A) == 2**n:
            return A
        if wgt > 2**(n-1)-2**(n/2-1):
            raise NoBentFunction

        A,B,W,wgt = self.__deepcopy(A,B,W,wgt)
        v = B.pop()
        cf = self.__coinFlip()
        values = (cf, (1+cf) % 2)
        for a in values:
            try:
                A[tuple(v)] = a
                update = self.__getUpdate(v, a)
                W_a = self.__applyUpdate(W, update, 2**n-len(A))
                R = self.__searchForBent(A,B,W_a,wgt+a)
                self.walshTrace.append((v, A[tuple(v)], [W_a[x] for x in W_a]))
                return R
            except NoBentFunction:
                pass
        raise NoBentFunction

    def __applyUpdate(self, W, update, leftToFill):
        n = self.n
        W = deepcopy(W)
        for u in self.V:
            W[tuple(u)] += update[tuple(u)]
            Wmin = W[tuple(u)] - leftToFill
            Wmax = W[tuple(u)] + leftToFill
            if not ((Wmin <= 2**(n/2) and 2**(n/2) <= Wmax) or (Wmin <= -2**(n/2) and -2**(n/2) <= Wmax)):
                raise NoBentFunction
        return W

    def __getUpdate(self,v,a):
        update = {}
        for u in self.V:
            update[tuple(u)] = Integer(-1)**(a+u.dot_product(v))
        return update

    def __deepcopy(self, *args):
        R = []
        for a in args:
            R.append(deepcopy(a))
        return R

    def __coinFlip(self):
        return 1 if random() > .5 else 0

\end{Verbatim}

An example:

\begin{Verbatim}[fontsize=\scriptsize,fontfamily=courier,fontshape=tt,frame=single,label=\sage]

sage: %attach bentFunctions.sage
sage: B=BentFinder(4)
sage: B.searchForBent()
(Boolean function with 4 variables,
 [((0, 0, 0, 1), 1, [1, 1, 1, 1, 1, 1, 1, 1, 1, 1, 1, 1, 1, 1, 1, 1]),
  ((0, 1, 0, 0), 0, [2, 2, 0, 0, 0, 0, 2, 2, 2, 0, 2, 0, 0, 2, 0, 2]),
  ((1, 1, 0, 1), 1, [1, 3, 1, 1, -1, -1, 3, 3, 3, 1, 1, 1, -1, 1, -1, 1]),
  ((0, 1, 1, 1), 0, [2, 2, 2, 2, 0, -2, 4, 4, 2, 0, 2, 0, 0, 0, -2, 0]),
  ((0, 0, 1, 0), 1, [3, 3, 1, 3, 1, -3, 3, 5, 1, -1, 1, 1, -1, 1, -1, -1]),
  ((0, 0, 1, 1), 0, [2, 4, 2, 2, 0, -4, 2, 6, 0, 0, 2, 2, -2, 0, 0, 0]),
  ((0, 1, 0, 1), 0, [1, 5, -1, 3, 1, -3, 1, 7, -1, -1, 3, 1, -1, 1, -1, 1]),
  ((0, 1, 1, 0), 0, [2, 4, 4, 2, 2, -4, 0, 6, -2, 0, 4, 0, -2, 0, -2, 2]),
  ((1, 1, 0, 0), 1, [1, 3, -1, 5, 1, -5, 1, 5, -1, -1, 5, 3, -3, 1, -1, 3]),
  ((1, 0, 1, 1), 0, [0, 4, 6, 2, 2, -6, 0, 4, 0, -2, 4, 2, -2, 0, -2, 4]),
  ((1, 0, 0, 0), 1, [1, 3, -1, 5, 1, -7, -1, 5, 1, -3, 3, 3, -1, 3, -1, 5]),
  ((1, 1, 1, 0), 0, [2, 4, 6, 4, 2, -6, -2, 4, 2, -4, 4, 0, -2, -2, 0, 4]),
  ((1, 0, 0, 1), 0, [3, 5, -3, 5, 1, -5, -1, 3, 1, -5, 3, 3, -3, 5, -1, 5]),
  ((1, 1, 1, 1), 0, [4, 6, 4, 4, 4, -6, -2, 2, 2, -4, 4, 2, -2, -4, -2, 4]),
  ((1, 0, 1, 0), 1, [3, 5, -3, 3, 3, -5, -3, 3, 3, -5, 5, 5, -3, 5, -3, 3]),
  ((1, 0, 1, 0), 1, [4, 4, 4, 4, 4, -4, -4, 4, 4, -4, 4, 4, -4, -4, -4, 4])])

\end{Verbatim}

\printindex

\end{document}